\documentclass[12pt]{amsart}
\usepackage{a4wide}
\usepackage[latin9]{inputenc}
\usepackage{amsmath}
\usepackage{amsfonts}
\usepackage{amssymb}
\usepackage{amsthm}
\usepackage{subfigure}
\usepackage[usenames,dvipsnames]{color}
\usepackage{pstricks}
\usepackage{graphicx}
\usepackage[all]{xy}
\newtheorem{theo}{Theorem}[section]
\newtheorem{prop}[theo]{Proposition}
\newtheorem{coro}[theo]{Corollary}
\newtheorem{lemm}[theo]{Lemma}

\theoremstyle{definition}
\newtheorem{def1}[theo]{Definition}
\theoremstyle{remark}
\newtheorem{rema}[theo]{Remark}

\newcommand{\nwc}{\newcommand}
\nwc{\eps}{\epsilon}
\nwc{\ep}{\epsilon}
\nwc{\vareps}{\varepsilon}
\nwc{\Oph}{\operatorname{Op}_\hbar}
\nwc{\la}{\langle}
\nwc{\ra}{\rangle}

\nwc{\mf}{\mathbf} 
\nwc{\blds}{\boldsymbol} 
\nwc{\ml}{\mathcal} 

\nwc{\defeq}{\stackrel{\rm{def}}{=}}

\nwc{\cE}{\ml{E}}
\nwc{\cN}{\ml{N}}
\nwc{\cO}{\ml{O}}
\nwc{\cP}{\ml{P}}
\nwc{\cU}{\ml{U}}
\nwc{\cV}{\ml{V}}
\nwc{\cW}{\ml{W}}
\nwc{\tU}{\widetilde{U}}
\nwc{\IN}{\mathbb{N}}
\nwc{\IR}{\mathbb{R}}
\nwc{\IZ}{\mathbb{Z}}
\nwc{\IC}{\mathbb{C}}
\nwc{\IT}{\mathbb{T}}
\nwc{\tP}{\widetilde{P}}
\nwc{\tPi}{\widetilde{\Pi}}
\nwc{\tV}{\widetilde{V}}
\nwc{\supp}{\operatorname{supp}}
\nwc{\rest}{\restriction}

\begin{document}

\title[Conormal cycle of random nodal sets]{Equidistribution of the conormal cycle of random nodal sets}

\author[Nguyen Viet Dang]{Nguyen Viet Dang}

\author[Gabriel Rivi\`ere]{Gabriel Rivi\`ere}
\address{Laboratoire Paul Painlev\'e (U.M.R. CNRS 8524), U.F.R. de Math\'ematiques, Universit\'e Lille 1, 59655 Villeneuve d'Ascq Cedex, France}

\email{nguyen-viet.dang@math.univ-lille1.fr}

\email{gabriel.riviere@math.univ-lille1.fr}

\begin{abstract} 
We study the asymptotic properties of the conormal cycle of nodal sets associated to a random superposition of eigenfunctions of the Laplacian on a smooth compact Riemannian manifold without boundary. In the case where the dimension is odd, we show that the expectation of the corresponding current of integration equidistributes on the fibers of the cotangent bundle. In the case where the dimension is even, we obtain an upper bound of lower order on the expectation. Using recent results of Alesker, we also deduce some properties on the asymptotic expectation of any smooth valuation including the Euler characteristic of random nodal sets. 
\end{abstract}

\maketitle

\section{Introduction}

We consider $(M,g)$ a smooth ($\ml{C}^{\infty}$) compact connected oriented Riemannian manifold without boundary, and we study the nonzero solutions of
\begin{equation}\label{e:eigenmode}-\Delta_g e_{\lambda}=\lambda^2e_{\lambda},\end{equation}                                                                              
where $\Delta_g$ is the Laplace-Beltrami operator induced by the Riemannian metric $g$ and $\lambda>0$. 
Our geometric assumption implies that there exists a non decreasing sequence 
$$0=\lambda_1<\lambda_2\leq\lambda_3\leq\ldots\leq \lambda_j \ldots\longrightarrow +\infty,$$
and an orthonormal basis $(e_j)_{j\geq 1}$ of $L^2(M)$ such that
$$\forall j\geq 1,\ -\Delta_g e_j=\lambda_j^2 e_j.$$
More generally, we are in fact interested in finite superpositions of such solutions, i.e.
\begin{equation}\label{e:quasimode}f\in\ml{H}_{\Lambda}:=\mathbf{1}_{[0,\Lambda^2]}(-\Delta_g)L^2(M).\end{equation}
Recall that, if we denote by $N(\Lambda)$ the dimension of $\ml{H}_{\Lambda}$, one has the following H\"ormander-Weyl asymptotics~\cite{Ho68}
\begin{equation}\label{e:weyl}
 N(\Lambda)= \frac{\text{Vol}_g(M)}{2^n\pi^{\frac{n}{2}}\Gamma\left(\frac{n}{2}+1\right)}\Lambda^n+\ml{O}(\Lambda^{n-1}),
\end{equation}
where $n$ is the dimension of $M$ and $\text{Vol}_g$ is the Riemannian volume induced by $g$. In the present article, we aim at studying elements in $\ml{H}_{\Lambda}$ in the high-frequency limit $\Lambda\rightarrow+\infty$. There are several natural ways to study these objects: among others, one can for instance look at their $L^p$ norms, their nodal sets, their nodal domains, or their quantum limits. In the present article, we will address the question of the geometry of the nodal sets,
$$\forall f\in\ml{H}_{\Lambda},\ \ml{N}_f:=\left\{x\in M: f(x)=0\right\}.$$
Again, there are many natural questions regarding the geometry of a nodal set $\ml{N}_f$ or its distribution on $M$ -- see for instance the book of Han and Lin~\cite{HanLi} or the recent survey of Zelditch~\cite{Ze13} on nodal sets of eigenfunctions. For instance, one can try to compute their Hausdorff volume, the number of their connected components, or more general geometric quantities like their Betti numbers.

The simplest question seems to be a priori the computation of the volume. In fact, in this case, it was conjectured by Yau~\cite{Y82} that there exist two constants $0<c_g\leq C_g<+\infty$ such that, for every nonzero solution $e_{\lambda}$ of~\eqref{e:eigenmode},
$$c_g\lambda\leq \ml{V}^{n-1}(\ml{N}_{e_{\lambda}})\leq C_g\lambda,$$
where $\ml{V}^{n-1}$ is the Riemannian hypersurface volume. In the case of real-analytic metrics, this conjecture was proved by Donnelly and Fefferman~\cite{DF88} -- see also~\cite{JeLeb99} for the case of arbitrary elements in $\ml{H}_{\Lambda}$. In the $\ml{C}^{\infty}$ case, the lower bound was proved earlier in dimension $2$ by Br\"uning~\cite{Br78} while in dimension $n\geq 3$, the best known lower bounds are of order $\lambda^{\frac{3-n}{2}}$ as recently proved by Colding and Minicozzi~\cite{CM} and Hezari, Sogge and Zelditch~\cite{HS12, SZ12}. Concerning the upper bound in the $\ml{C}^{\infty}$ case, Dong, Donnelly and Fefferman obtained an upper bound of order $\lambda^{\frac{3}{2}}$ in dimension $2$~\cite{DF90, Do92}, and in dimension $n\geq 3$, Hardt and Simon proved an upper bound of order $\lambda^{c\lambda}$ for some $c>0$~\cite{HaSi89}. Thus, even the a priori simpler question of estimating the volume of nodal sets remains to our knowledge far from being completely understood in a general 
setting.

A natural approach is to ask whether the properties expected for any nodal set are valid for a ``generic'' solution of~\eqref{e:eigenmode} or at least for a ``generic'' element $f$ of $\ml{H}_{\Lambda}$. In order to formulate a notion of genericity in $\ml{H}_{\Lambda}$, we introduce the following Gaussian probability measure on the space $\ml{H}_{\Lambda}$~:
$$\mu_{\Lambda}(df)=d\mu_{\Lambda}(f):=e^{-\frac{N(\Lambda)\|f\|^2}{2}}\left(\frac{N(\Lambda)}{2\pi}\right)^{\frac{N(\Lambda)}{2}}dc_1\ldots dc_{N(\Lambda)},\quad \text{with}\quad f=\sum_{j=1}^{N(\Lambda)}c_je_j.$$
\begin{rema}
We note that, with this convention, the expectation of the $L^2$ norm is
$$\int_{\ml{H}_{\Lambda}}\|f\|^2_{L^2(M)}d\mu_{\Lambda}(f)=1.$$
\end{rema}
In this setting, B\'erard gave an ``average'' version~\cite{Be85} of Yau's conjecture, i.e. one has
$$\int_{\ml{H}_{\Lambda}}\ml{V}^{n-1}(\ml{N}_{f})d\mu_{\Lambda}(f)=\frac{\text{Vol}(\mathbb{S}^{n-1})}{\text{Vol}(\mathbb{S}^{n})}\frac{\Lambda}{\sqrt{n+2}}\text{Vol}_g(M)(1+o(1)).$$
More recently, Rudnick and Wigman also estimated the variance in the case where $(M,g)$ is a rational torus~\cite{RuWi08} or the canonical sphere~\cite{Wi10}, and Letendre generalized B\'erard's result to the intersection of several ``random'' nodal sets~\cite{Le14}. In~\cite{Ze09}, Zelditch also showed that something slightly stronger than B\'erard's result holds. More precisely, he proved that, given any smooth function $\omega$ on $M$, one has
\begin{equation}\label{e:zelditch}\int_{\ml{H}_{\Lambda}}\left\la f^*(\delta_0)\|df\|,\omega\right\ra d\mu_{\Lambda}(f)=\frac{\text{Vol}(\mathbb{S}^{n-1})}{\text{Vol}(\mathbb{S}^{n})}\frac{\Lambda}{\sqrt{n+2}}\int_M\omega(x)d\text{Vol}_g(M)(1+o(1)),\end{equation}
where $f^*(\delta_0)$ is the pullback by $f$ of the Dirac distribution, and $\la.,.\ra$ is the duality bracket. This result shows that, if we average over the Gaussian measure, the nodal sets become equidistributed in $M$.

Using this probabilistic approach, we can in fact say much more on the geometry and the topology of the nodal sets. In the case of the canonical $2$-sphere, Nazarov and Sodin gave large deviation estimates 
for the number of connected components of a random nodal set~\cite{NaSo09}. Lerario and Lundberg also obtained lower and upper bounds for the expectation of the number of connected components 
when one considers more general families of random nodal sets~\cite{LerarioLundberg}. In~\cite{Ni11}, Nicolaescu estimated the number of critical points of a random function $f$ in $\ml{H}_{\Lambda}$ from which 
one can also deduce upper bounds on the number of connected components. In~\cite{So12, NaSo15}, Nazarov and Sodin generalized their large deviation result on random spherical harmonics and they studied the number 
of connected components of $\ml{N}_f$ for general Gaussian random functions. In our context, if we denote by $N(f)$ the 
number of connected components of $\ml{N}_f$, their result reads
$$\forall\delta>0,\quad \lim_{\Lambda\rightarrow+\infty}\mu_{\Lambda}\left(\left\{f\in\ml{H}_{\Lambda}:\left|\frac{N(f)}{\Lambda^n}-a_{g}\right|>\delta\right\}\right)=0,$$
for some \emph{non explicit constant} $a_{g}>0$ depending only on $(M,g)$. Note that it does not only give the rate of convergence for the expectation but also a large deviation estimate. 
It is also slightly more precise in the sense that the result remains true if we count the number of connected components in some rescaled 
geodesic ball $B_g(x,R\Lambda^{-1})$ for some large $R>0$ -- see Theorem~$5$ in~\cite{So12}.

In~\cite{GaWe15}, Gayet and Welschinger proved that, given any closed hypersurface $\Sigma$ in $\IR^n$, the probability to find $\Sigma$ in the intersection 
of $\ml{N}_f$ with a geodesic ball $B_g(x,R\Lambda^{-1})$ is uniformly bounded from below by an \emph{explicit positive constant} -- see also~\cite{GaWe14c} for earlier results in the case of random polynomials. 
This allows to deduce lower bounds on all the Betti numbers of our random nodal sets. In fact, they also proved that, on any smooth compact connected Riemannian manifold $(M,g)$ without boundary and for every $0\leq i\leq n-1$, 
one can find \emph{explicit constants} $0<c_i(M,g)\leq C_i(M,g)<+\infty$ such that
$$c_i(M,g)\Lambda^n\leq \int_{\ml{H}_{\Lambda}}b_i(\ml{N}_{f})d\mu_{\Lambda}(f)\leq C_i(M,g)\Lambda^n,$$
where $b_i$ is the $i$-th Betti number~\cite{GaWe14b, GaWe15}. Recall that $b_0(\ml{N}_f)=N(f)$. Their result is in fact more general than what we claim in the sense that it is valid for any elliptic pseudo-differential
operator of positive order\footnote{In this case, the asymptotics are of order $\Lambda^{\frac{2n}{m}}$.} $m$. In the case where $n$ is odd, Letendre~\cite{Le14} showed the following asymptotic property of the Euler characteristic $\chi(\ml{N}_{f})$ of a random nodal set:
\begin{equation}\label{e:letendre}\int_{\ml{H}_{\Lambda}}\chi(\ml{N}_{f})d\mu_{\Lambda}(f)=\frac{2(-1)^{\frac{n-1}{2}}}{\pi\text{Vol}(\mathbb{S}^{n-1})}\left(\frac{\Lambda}{\sqrt{n+2}}\right)^n\text{Vol}_g(M)+\ml{O}(\Lambda^{n-1}).\end{equation}
We recall that the Euler characteristic is given by the alternate sum of Betti numbers. Finally, in~\cite{SaWi15}, Sarnak and Wigman recently described the universal laws 
satisfied by the topologies of random nodal sets -- see also~\cite{CanSa14}.

These different results show that more is known on the geometry of nodal sets if one only aims at probabilistic results. 
Even if we do not address this kind of questions in the present article, we note that probabilistic approaches have been considered for a long time in algebraic geometry -- see for instance~\cite{Kac43} 
where Kac estimated the number of real zeros of a random polynomial as its degree goes to infinity. In ``random'' algebraic geometry, one aims at studying the zero set of random 
polynomials of several variables instead of random superposition of eigenfunctions, and the spectral parameter $\Lambda$ is replaced by the degree of the 
polynomials~\cite{Kos93, ShubSmale, ShZe99, BlShZe00, BlShZe01, GaWe14a, Le14}.

So far, we have only discussed the mathematical point of view of random Gaussian superposition of high-frequency eigenfunctions but we emphasize that they play an important role in the physics literature. For instance, in quantum chaos, one is interested in understanding the semiclassical properties of a quantum system whose underlying classical system enjoys chaotic features such as ergodicity or mixing. In~\cite{B77}, Berry predicted that semiclassical eigenmodes of chaotic systems should exhibit the same behaviour as a random superposition of waves. This is known as the \emph{Berry random wave conjecture}. In particular, any result on random superposition of waves should provide an intuition on what one could expect for a chaotic system. In our framework, the main example of chaotic system is given by manifolds with negative curvature 
whose geodesic flows are exponentially mixing. 
Motivated by this conjecture of Berry, random nodal sets (and nodal domains\footnote{This means the subsets of $M$ where $f$ has a fixed sign.}) were also extensively studied in the physics literature especially in the last fifteen years starting from the works of Blum, Gnutzmann and Smilansky~\cite{BluGnSm02} on nodal domains. The case of nodal sets was studied by Berry in~\cite{B02} where he considered (on $2$-dimensional domains) similar questions as the ones mentioned above on the volume and the curvature of random nodal sets. In the $2$-dimensional case, Bogomolny, Dubertrand and Schmit conjectured that random nodal sets are in fact well described in terms of
percolation processes and 
Schramm-Loewner Evolution models~\cite{BogDuSc07}.

\section{Statement of the main results}

In the sequel, we denote by $T^*M$ the cotangent bundle of $M$ and an element of $T^*M$
is denoted by $(x;\xi)$ for $x\in M,\xi\in T_x^*M$.
From the point of view of ``microlocal geometry'', it is often more natural to consider the conormal cycle of $\ml{N}_f=\{f=0\}$ rather than the set $\ml{N}_f$ itself \cite[Chapter 9]{KSsheaves}.
Recall that the conormal cycle is defined as follows, for every $f$ in $\ml{H}_{\Lambda}$,
\begin{equation}
N^*(\{f=0\})=\{(x;\xi)\in T^*M\text{ s.t. }f(x)=0, \xi=td_xf \text{ for some } t\neq 0\} \subset T^*M.
\end{equation}
Note that the above defines the conormal only in the set theoretical sense and we will later
discuss the central issue of orientation.
Outside of its singular points, this subset of $T^*M$ is a \emph{Lagrangian conical submanifold} and it contains in some sense more informations on the geometry of the nodal sets. For that reason, it sounds natural to us to focus on the properties of this set. 
Before giving more explanations on the relevance of the conormal cycle in ``microlocal geometry'', we mention, as a first motivation, a celebrated index Theorem of Kashiwara~\cite[Thm 9.5.3 p.~385]{KSsheaves}. In the real analytic case, this result
expresses the Euler characteristic of $\mathcal{N}_f$ as a Lagrangian
intersection in $T^*M$ between the conormal cycle of $\mathcal{N}_f$ and the
graph of $dg$ for a generic function $g$ on $M$~:
\begin{equation}
[\text{ Graph }dg] \cap [N^*(\{f=0\})]=\chi\left(\ml{N}_f\right) 
\end{equation}
where the intersection is in the oriented sense.

To our knowledge, the properties of this set have not been studied neither from the deterministic point of view, nor from the probabilistic one. In this article, we aim at studying the probabilistic properties of the set $N^*(\{f=0\})$ which are somewhat easier to consider. For that purpose, we need to introduce the following subset of $\ml{H}_{\Lambda}$:
$$D_{\Lambda}:=\left\{ f\in\ml{H}_{\Lambda}:\ \exists x\in \ml{N}_f\ \text{such that}\ d_xf=0\right\}.$$
By a Sard type argument -- see for instance paragraph~$2.3$ in~\cite{Le14}, one can verify that $\mu_{\Lambda}(D_{\Lambda})=0$ for $\Lambda$ large enough. Thus, for $\mu_{\Lambda}$-a.e. $f$ in $\ml{H}_{\Lambda}$, $N^*(\{f=0\})$ is a smooth $n$-dimensional submanifold of 
$$T^{\bullet} M:=\{(x,\xi)\in T^*M: \xi\neq 0\}.$$
In particular, once we choose an orientation, $N^*(\{f=0\})$ can be viewed as a $n$-current $[N^*(\{f=0\})]$ in the sense of de Rham~\cite{dRh80, Sch66}. 
It means that, for every smooth compactly supported $n$--form $\omega$ on $T^\bullet M$, we define
$$\left\la [N^*(\{f=0\})],\omega\right\ra=\int_{N^*(\{f=0\})}\omega.$$
\begin{rema}
 We note that $N^*(\{f=0\})$ has two components for $f$ in $D_{\Lambda}$:
 $$N^*_{\pm}(\{f=0\}):=\left\{(x,\xi)\in T^{\bullet}M\ \text{s.t.}\ f(x)=0\ \text{and}\ \xi=td_xf\ \text{for some}\ t\in\IR_{\pm}\backslash\{0\}\right\}.$$
We will discuss precisely orientability questions in paragraph~\ref{ss:def-conormal-cycle}. We will in fact use the classical conventions of~\cite[p.~682]{Ch45}.
\end{rema}

Our first main result shows that this defines in fact an $L^1$ function:
\begin{theo}\label{integtheo}
Let $(M,g)$ be a smooth oriented connected compact Riemannian manifold without boundary of dimension $n$. Then, the map
\begin{equation}
f\in\mathcal{H}_\Lambda \mapsto  [N^*(\{f=0\})]\in\mathcal{D}_n^\prime(T^{\bullet}M)
\end{equation}
is integrable with respect to the Gaussian measure $\mu_\Lambda$. Equivalently, for every test form $\omega$ on $T^{\bullet} M$, the map
$$f\in\mathcal{H}_\Lambda \mapsto  \la[N^*(\{f=0\})],\omega\ra \in\IR$$
belongs to $L^1(\ml{H}_{\Lambda},d\mu_{\Lambda})$.
\end{theo}
In order to state our second result, let us denote by $\Omega_g$ the Riemannian volume form on $(M,g)$, by $\pi:T^*M\rightarrow M$ the natural projection, and by $\pi^*\Omega_g$ the pull--back of the Riemannian volume on $T^*M$.
Then, our second theorem shows the following universal behaviour:
\begin{theo}\label{asympttheo}
Let $(M,g)$ be a smooth oriented connected compact Riemannian manifold without boundary of dimension $n$. Then, one has
\begin{equation}\label{normalcyclecontact}
\int_{\ml{H}_{\Lambda}}[N^*(\{f=0\})] d\mu_{\Lambda}(f)= C_n\left(\frac{\Lambda}{\sqrt{n+2}}\right)^n\pi^*\Omega_g + \ml{O}(\Lambda^{n-1}),
\end{equation}
with $$C_n=\frac{2(-1)^{\frac{n+1}{2}}}{\pi\operatorname{Vol}(\mathbb{S}^{n-1})}\text{ if } n \text{ is odd},\text{ and } C_n=0\text{ otherwise}.$$ 
Equivalently, it means that, for every test form $\omega$ on $T^{\bullet} M$, one has
$$\int_{\ml{H}_{\Lambda}}\left\la[N^*(\{f=0\})],\omega\right\ra d\mu_{\Lambda}(f)=C_n\left(\frac{\Lambda}{\sqrt{n+2}}\right)^n\int_{T^*M}\pi^*\Omega_g\wedge\omega + \ml{O}(\Lambda^{n-1}).$$
\end{theo}
We note that an important feature of this result is that it shows a very different behaviour depending on whether the dimension is odd or not. In the case where $n$ is odd, our theorem shows that \emph{the submanifold $N^*(\{f=0\})$ becomes uniformly equidistributed on the fibers $T_x^*M$} while, when $n$ is even, we only obtain an upper bound of \emph{lower order}.

\subsection{A brief outline of the proof using Berezin integrals}

Our asymptotic formula can be elegantly derived by using a representation of the
conormal cycle as an oscillatory integral over odd and even variables in the formalism
of Berezin~\cite{Be85, GuSt99, Dis11, Tao13} -- see section~\ref{s:berezin} for a brief reminder. Let us now sketch the principle of our derivation. 
In a geodesic coordinate chart, we first introduce, for $f$ in $\mathcal{H}_{\Lambda}\setminus D_\Lambda$, a map
$$G(f):(x,\xi,t)\in T^\bullet M\times\mathbb{R}^*\rightarrow y=(f(x),td_xf-\xi)\in\mathbb{R}^{n+1}.$$
By classical arguments involving wave front sets, we then represent 
the conormal cycle  as a current in
the sense of De Rham which is obtained by pull--back operation by $G(f)$~:
$$[N^*(f=0)]=\int_{t\in\mathbb{R}^*}G(f)^*\left(\delta_{0}^{n+1}(y)dy^1\wedge dy^2\ldots\wedge dy^{n+1}\right).$$
Then we use the formalism of Berezin integrals to write the current $\delta_{0}^{n+1}(y)dy^1\wedge dy^2\ldots\wedge dy^{n+1}$ under exponential form~:
\begin{eqnarray*}
\delta_{0}^{n+1}(y)dy^1\wedge dy^2\ldots\wedge dy^{n+1}  =\left(\int_{\mathbb{R}^{n+1}}e^{-2i\pi p.y}dp\right)dy^1\wedge dy^2\ldots\wedge dy^{n+1}\\
  =\frac{1}{(-2i\pi)^{n+1}}\int_{\mathbb{R}^{(n+1|n+1)}}e^{-2i\pi (p.y+\Pi.dy)}dp d\Pi.
\end{eqnarray*}
Thus, the conormal cycle can be expressed as an oscillatory integral over even and odd variables~:
\begin{equation}\label{e:oscillatory-berezin}[N^*(f=0)]=\frac{1}{(-2i\pi)^{n+1}}\int_{t\in\mathbb{R}^*}\left(\int_{\mathbb{R}^{(n+1|n+1)}}e^{-2i\pi (p.G(f)+\Pi.dG(f))}dpd\Pi\right).\end{equation}
Then the calculation of the expected conormal cycle reduces to an evaluation of 
Gaussian integrals except that we are in presence of odd and even variables and this exponential representation of the current dramatically simplifies
the combinatorics. Integrating over the measure $\mu_{\Lambda}$ and inverting (formally) the integrals yields that we shall compute
$$\int_{\mathbb{R}^{(n+1|n+1)}}\left(\int_{\mathcal{H}_{\Lambda}}e^{-2i\pi (p.G(f)+\Pi.dG(f))}
d\mu_{\Lambda}(f)\right)dpd\Pi.$$ 
Note that this inversion of the order of integration requires rigorous justification which are provided in subsection \ref{ss:integrability}. We also refer to Remark~\ref{r:kac-rice} for a discussion on this question and the so-called Kac-Rice formula. 
Finally, we compute the Gaussian integral
\begin{equation}\label{e:berezin-gaussian}
 \int_{\mathcal{H}_{\Lambda}}e^{-2i\pi (p.G(f)+\Pi.dG(f))}d\mu_{\Lambda}(f)
\end{equation}
using H\"ormander pointwise Weyl's asymptotics on the spectral projector~\cite{Ho68} and supersymmetric Gaussian integration~\cite{GuSt99}.
It must be noticed that the Berezin
integral was also used in
the works~\cite{AshokDouglas,ZelditchBleherShiffman,ZelditchDouglasShiffman3} devoted
to the statistics of the critical
points of random holomorphic sections with applications to counting vacuas in string theory.
The general philosophy is that Berezin
integrals allow to elegantly express determinants as supersymmetric integrals
which hugely simplifies
calculations.

\subsection{Applications: smooth valuations on manifolds}\label{ss:alesker}

We will now explain the interest of the concept of conormal cycle by giving some applications of our equidistribution
Theorem. First, we recall a ``microlocal version'' of the 
generalized Gauss--Bonnet Theorem of Chern~\cite[equation (20) p.~679--683]{Ch45} later generalized by Fu for subanalytic sets~\cite[p.~832]{Fu} (see also \cite[Definition 4.4.1 and Proposition 4.4.2]{Park} for a detailed exposition). 

\begin{theo}\label{Fugaussbonnet}
Let $(M,g)$ be a smooth compact Riemannian manifold of dimension $n$. Then, there exists a smooth, closed,
compactly supported $n$--form $\theta$ on $T^\bullet M$, so that
$\int_{T^\bullet_x M} \theta=1, \forall x\in M$ and for any smooth oriented submanifold of codimension $\geq 1$, one has  
\begin{equation}
\chi(X)=\langle[N^*\left(X\right)],\theta\rangle
\end{equation}
where $[N^*\left(X\right)]$ is the 
conormal cycle of $X$ and where $\chi$ is its Euler characteristic.
\end{theo}
\begin{proof}
Our Theorem is a microlocal interpretation
of the main result of~\cite{Ch45} that we briefly recall. Let $\pi_0:\mathbb{U}M\mapsto M$ denote the unit tangent
bundle over $M$.
Then Chern constructed a $(n-1)$--form $\Theta$ on the unit tangent bundle $\mathbb{U}M$ defined in~\cite[equation (9) p.~676]{Ch45}
satisfying the following properties
\begin{itemize}
\item $d\Theta=0$ (\cite[equation (11) p.~677]{Ch45}),
\item $-\int_{\pi_0^{-1}(x)}\Theta=1$, \text{(see \cite[p.~51]{Park})}.
\item for any oriented submanifold $X$
of $M$, if we denote by $\mathbb{U}NX\subset \mathbb{U}M$ the unit normal bundle of $X$ in $M$ then
we have a generalized Chern--Gauss--Bonnet identity
\begin{eqnarray*}
-\int_{\mathbb{U}NX}\Theta=\chi(X),\text{ (\cite[equation (20) p.~679]{Ch45} and \cite[Proposition 4.
4.2 p.~52]{Park})}.
\end{eqnarray*}
\end{itemize}
The form $-\Theta$ is called \textbf{geodesic curvature form} by Park~\cite[Definition 4.1.1 p.~51]{Park}. Since $M$ is Riemannian, the metric gives an isomorphism
$T^*M\simeq TM$ which induces an isomorphism of cones
$T^\bullet M=T^*M\setminus\{\underline{0}\}\simeq TM\setminus\{\underline{0}\}$. By the above isomorphism and the natural trivial fibration
$TM\setminus\{\underline{0}\}\mapsto \mathbb{U}M$
whose fiber is the group $(\mathbb{R}_{>0},\times)$,
there is an
isomorphism of cones $p:T^\bullet M\mapsto \mathbb{U}M\times \mathbb{R}_{>0}$. 
Choose any smooth function $\varphi\geq 0$ on $\mathbb{R}_{>0}$ s.t. $d\varphi$ is compactly supported and $\int_0^{+\infty}d\varphi=1$, such function is easy to construct by 
considering $\chi\in C^\infty_c(\mathbb{R}_{>0}),\chi\geqslant 0$ such that $\int_0^\infty\chi(t)dt=1$ and set $\varphi(x)=\int_0^x \chi(t)dt$.
Then consider the $n$--form $\theta=p^*\left(-\Theta\wedge d\varphi\right)$ and let us check it satisfies the claim of the Theorem. First the integral of $\int_{T^\bullet_x M} \theta$ over a fiber $T^\bullet_x M$ satisfies~:
$\int_{T^\bullet_x M} \theta=
\int_{T^\bullet_x M} p^*\left(-\Theta\wedge d\varphi \right)=\int_{\mathbb{U}_xM\times \mathbb{R}_{>0}}(-\Theta)\wedge d\varphi=\int_{\mathbb{U}_xM}(-\Theta)\int_{\mathbb{R}_{>0}} d\varphi=1$.
Then the pairing $\langle[N^*\left(X\right)],\theta\rangle$ satisfies the identity~:
$\langle[N^*\left(X\right)],\theta\rangle=\int_{N^*(X)}\theta=\int_{N^*(X)}
p^*\left(-\Theta\wedge d\varphi\right)=\int_{\mathbb{U}NX\times \mathbb{R}_{>0}} -\Theta\wedge d\varphi$ $=\underset{\chi(X)}{\underbrace{\left(\int_{\mathbb{U}N(X)} -\Theta\right)}}\underset{1}{\underbrace{\left(\int_{\mathbb{R}_{>0}} d\varphi\right)}}=\chi(X)$.
\end{proof}

\begin{rema} In our main theorems, we chose to consider the conormal cycle in $T^{\bullet}M$ instead of $\mathbb{U}M$ in order to make to make some aspects of the calculation slightly simpler. 
\end{rema}

From this result, we deduce an alternative proof of equality~\eqref{e:letendre} which gives the mean Euler characteristic of random nodal sets~:
\begin{theo}[Letendre \cite{Le14}]
Let $(M,g)$ be a smooth oriented compact Riemannian manifold of odd dimension $n$. Then, one has
\begin{equation}
\int_{\mathcal{H}_\Lambda}\chi\left(\mathcal{N}_f\right)d\mu_\Lambda(f)=\frac{2(-1)^{\frac{n-1}{2}}}{\pi\operatorname{Vol}(\mathbb{S}^{n-1})}\operatorname{Vol}_g(M)\left(\frac{\Lambda}{\sqrt{n+2}}\right)^n 
+ \ml{O}(\Lambda^{n-1}).
\end{equation}
\end{theo}
\begin{proof} Let $\theta$ be the $n$-form of Theorem~\ref{Fugaussbonnet}. Then, one finds
\begin{eqnarray*}
\int_{\mathcal{H}_\Lambda}\chi\left(\mathcal{N}_f\right)d\mu_\Lambda(f)&=& \int_{\ml{H}_{\Lambda}}\left\langle [N^*(\{f=0\})],\theta\right\rangle d\mu_{\Lambda}(f)
\end{eqnarray*}
since the map $f\mapsto \left\langle [N^*(\{f=0\})],\theta\right\rangle$ is integrable with respect to the measure $d\mu_\Lambda$ by Theorem \ref{integtheo}. Hence by Theorem~\ref{asympttheo}, we obtain an integral on the cotangent
cone~:
\begin{eqnarray*}
\int_{\mathcal{H}_\Lambda}\chi\left(\mathcal{N}_f\right)d\mu_\Lambda(f)&=&\int_{T^\bullet M}C_n\left(\frac{\Lambda}{\sqrt{n+2}}\right)^n\pi^*\Omega_g\wedge \theta + \ml{O}(\Lambda^{n-1})\\
&=&\int_{x\in M} \left(\int_{T^\bullet_xM} (-1)^n\theta\right)\wedge C_n\left(\frac{\Lambda}{\sqrt{n+2}}\right)^n\pi^*\Omega_g + \ml{O}(\Lambda^{n-1})\\
&=&(-1)^n\int_{x\in M} C_n\left(\frac{\Lambda}{\sqrt{n+2}}\right)^n\Omega_g + \ml{O}(\Lambda^{n-1})\\
&=&\frac{2(-1)^n(-1)^{\frac{n+1}{2}}}{\pi\text{Vol}(\mathbb{S}^{n-1})}\text{Vol}(M)\left(\frac{\Lambda}{\sqrt{n+2}}\right)^n + \ml{O}(\Lambda^{n-1})
\end{eqnarray*}
where we used a Fubini Theorem on the cotangent
cone $T^\bullet M$ and we integrated on the fibers of $T^\bullet M$ first.
Note that since $n$ is odd, $(-1)^{n}(-1)^{\frac{n+1}{2}}=(-1)^{\frac{3n+1}{2}}=(-1)^{\frac{n-1}{2}}$ which explains
the constant found in the statement of our Theorem.
\end{proof}
In~\cite{Le14}, Letendre computed the Euler characteristic by expressing it as the integral over $\{f=0\}$ of a certain curvature form \emph{depending on $f$} as for instance in~\cite[Eq.~(9)]{Ch44}. Here, instead of this approach, we use Theorem~\ref{Fugaussbonnet} which gives the Euler characteristic as the integral of a \emph{fixed curvature form} over $N^*(\{f=0\})$. In some sense, this point of view allows us to extend Letendre's Theorem into an equidistribution result in the same way as Zelditch's equidistribution result~\eqref{e:zelditch} generalized B\'erard's result on the volume.

Chern's formula for the Euler characteritic of a smooth submanifold can in fact be understood in the more general framework of the theory of smooth valuations recently developped by Alesker~\cite{Al07} and that we will now briefly review. We consider $X$ a smooth oriented manifold (not necessarily endowed with a Riemannian structure) and we denote by $\ml{P}(X)$ the set of all compact submanifolds of $M$ with corners. In the terminology of~\cite{Al07}, we say that a map
$$\phi:\ml{P}(X)\rightarrow\mathbb{C}$$
is a smooth valuation if it is a finitely additive functional, and if it satisfies certain continuity properties. For simplicity of exposition, we remain vague on these two notions which need to be defined carefully on $\ml{P}(X)$ -- see Part~II of~\cite{Al06} for details. As was already mentionned, smooth valuations on manifolds generalize classical concepts from integral geometry
such as volumes, Euler characteristic and mixed volumes. One of the remarkable property of these valuations is that
they can be represented as follows~: for every smooth valuation $\phi$, there exists a smooth differential $n$-form $\omega_{\phi}$ on $T^*X$ such that, for any $P$ in $\ml{P}(X)$, one has 
$$\phi(P)=\int_{N^*(P)}\omega_{\phi},$$
where $N^*(P)$ is the conormal cycle of $P$~\cite{Al07}. We emphasize that the converse statement is also true, and that the $n$-form is a priori non unique. In the example of the Euler characteristic $\chi$, Theorem \ref{Fugaussbonnet} gives such a $n$-form as soon as we have fixed some arbitrary Riemannian metric on $X$.

Thanks to this interpretation of smooth valuations as integrals over conormal cycles, Theorems~\ref{integtheo} and~\ref{asympttheo} imply the following corollary on smooth valuations:
\begin{coro}
Let $(M,g)$ be a smooth connected compact Riemannian manifold without boundary of dimension $n$. Let $\phi$ be a smooth valuation in the sense of~\cite{Al06} such that $\omega_{\phi}$ can be chosen compactly supported in $T^{\bullet}M$. Then, the map
$$f\mapsto \phi(\ml{N}_f)$$
belongs to $L^1(\ml{H}_{\Lambda},d\mu_{\Lambda})$. Moreover, one has
\begin{equation}
\int_{\ml{H}_{\Lambda}}\phi(\ml{N}_f) d\mu_{\Lambda}(f)= C_n\left(\frac{\Lambda}{\sqrt{n+2}}\right)^n\int_{T^*M}\pi^*\Omega_g\wedge\omega_{\phi} + \ml{O}(\Lambda^{n-1}),
\end{equation}
where 
$$C_n=\frac{2(-1)^{\frac{n+1}{2}}}{\pi\operatorname{Vol}(\mathbb{S}^{n-1})}\text{ if } n \text{ is odd},\text{ and } C_n=0\text{ otherwise}.$$ 
\end{coro}

\subsection{Further questions}

Let us now mention natural questions that can be asked about the conormal cycles attached to nodal sets.

\begin{itemize}
 \item In the spirit of~\cite{Le14}, it would be natural to consider intersections of nodal sets and compute the expectation of their conormal cycle. This could probably be directly obtained with the methods of the article but would require some slightly more complicated combinatorial arguments. For simplicity of exposition, we decided to consider only one nodal set.
 \item Regarding Theorem~\ref{asympttheo}, it would be natural to understand the asymptotics in the case where $n$ is even. At least in the case where the set of closed geodesics is of zero Liouville measure, we believe that one should be able to compute the term of order $\Lambda^{n-1}$ in the asymptotics using the fact that in this geometric framework, the remainders in Weyl's asymptotics are of order $o(\Lambda^{n-1})$. Understanding the geometric meaning of the lower order term (and also checking that it is a nonzero term) is probably more subtle than computing the leading term we obtain here.
 \item In this article, we only computed the expectation of the conormal cycle and it would be of course natural to look for variance or large deviations estimates in the spirit of~\cite{B02, RuWi08, NaSo09, Wi10, NaSo15}.
 \item Finally, it would be natural to understand what happens in the deterministic case. Can one obtain at least some upper bounds on the conormal cycle of a deterministic nodal set? If yes, what are the rates, and do they also depend on the dimension? This kind of questions would require completely different techniques like the ones in~\cite{HanLi}.
\end{itemize}

\subsection{Organization of the article} Section~\ref{firsttechnicalsection} gathers preliminary results that will be used in our proof. More precisely, our first task in the present article is to express the conormal 
cycle $[N^*(f=0)]$ explicitely in terms of Dirac distributions and differential forms. This kind of representation appears in the book of Schwartz~\cite{Sch66} and it is in some sense slightly more adapted to 
our problem than the geometrically appealing definitions appearing in the works of Kashiwara--Schapira, Fu, Alesker, Bernig (\cite{Bernig, KSsheaves, Kashiwaraindex, Fu}). Thus, we give in section~\ref{firsttechnicalsection} 
an explicit integral formula for the conormal cycle. Moreover, we also introduce in this section other relevant tools needed for the proof of our main Theorem.

In section~\ref{s:proof} which forms the core of our paper,
we give the complete proofs of Theorems~\ref{integtheo} and~\ref{asympttheo}
using microlocal analysis and combinatorics.

Then, in section~\ref{s:berezin}, building on tools from quantum field theory, we give a new integral formula which expresses the conormal cycle of a nodal set as
an oscillatory integral in bosonic (even) and fermionic (odd) variables. Our formula makes use of the so called \emph{Berezin integral}~\cite{Be85, GuSt99, Dis11, Tao13} which was already mentioned above. 
This formula is inspired from a very general formula expressing integration currents
as oscillatory integrals on odd and even variables used by A. Losev et al. in their works on
instantonic quantum field theories~\cite{LosevI,LosevSli}. 
We give a leisurely introduction to the necessary tools to grasp the meaning of our formula. Using this very simple integral formula representing $[N^*(f=0)]$, we give a fast derivation 
of the leading term of the asymptotics of $\int_{\ml{H}_{\Lambda}}[N^*(\{f=0\})]d\mu_{\Lambda}(f)$ which becomes a simple exercise of Gaussian
integration with respect to odd and even variables and completely avoids the heavy combinatorics
of the second part. The introduction of even (fermionic) variables allows us to avoid complicated combinatorics involving 
sums over partitions at the expense of a little bit of abstraction. The calculations are similar to those encountered in quantum field theory
on curved Riemannian manifolds. This somewhat alternative derivation of the main asymptotic formula is presented in this last part in a purely formal way and it relies on 
several inversions of integrals and a priori estimates on remainder terms which could in fact be justified by using the arguments of the second part of the article.

The article also contains 4 appendices. Appendix~\ref{derivativesspectralproj} recalls some classical results on the derivatives of the spectral projector of the Laplacian. Appendix~\ref{intertwiningsectionappendix} 
justifies the inversion formulas used in the article via tools from microlocal analysis. In appendix~\ref{sectionappendixwhitneyampleness}, we briefly discuss a question of independent interest. More precisely, 
building on the proof of lemma~\ref{l:ample}, we give a functional analytic proof of Whitney embedding theorem. Finally, appendix~\ref{a:currents} gives a brief account on the properties of currents needed for our purpose 
and carefully justify the meaning of our formula using tools from microlocal analysis.

\subsection*{Acknowledgements} We warmly thank Semyon Alesker and St\'ephane Nonnenmacher for very useful discussions related to this work. We also thank Didier Robert and Jean-Yves Welschinger for giving inspiring 
talks in Lille on random waves from which this work started, and Patrick Popescu Pampu for explanations related to the Euler characteristic. Finally, the referees made several useful suggestions which helped to improve 
the exposition of the article and we would like to thank them. This work was done while the first author was a post-doc in the framework of the CEMPI program (ANR-11-LABX-0007-01) from the Agence Nationale de la Recherche. 
The second author is also partially supported by the Agence Nationale de la Recherche through the Labex CEMPI (ANR-11-LABX-0007-01) and the ANR project GERASIC (ANR-13-BS01-0007-01).

\section{Preliminary results}
\label{firsttechnicalsection}

In this section, we give the three main ingredients for the proof of our main Theorem.
\emph{The first ingredient} is Proposition~\ref{p:current-local-chart} which gives a representation formula for the conormal cycle in terms of pull-back of Dirac distributions.
As we shall see in this proposition, the conormal cycle can be thought as a local functional
on $2$-jets of functions. Thus, \emph{our second ingredient} (Prop.~\ref{p:pushforwardGaussmeasure}) describes the pushforward of the Gaussian measure
$\mu_\Lambda$ on $\mathcal{H}_\Lambda$
by the map
$$J_{\Lambda}^2(y):f\in \ml{H}_{\Lambda}\longmapsto F=\left(f(y);\left(\frac{\partial_{y^l}f(y)}{\Lambda}\right)_{1\leq l\leq n};\left(\frac{\partial_{y^ry^s}^2f(y)}{\Lambda^2}\right)_{1\leq r, s\leq n}\right)\in \IR^{n^2+n+1}$$
where we identify $\IR^{n^2+n+1}$ with the fibers of the $2$--jet bundle.  
The pushforward measure $J_{\Lambda *}^2\mu_\Lambda$ is in fact defined by its characteristic function which is expressed in terms of the kernel of the spectral projector defining our Gaussian measure $\mu_{\Lambda}$. 
Then \emph{the third ingredient} (Prop.~\ref{intertwiningproposition}) is a kind of ``Fubini statement'' which allows to intertwine the order of integration of the variables $(y,\eta,t,f)$ involved in our problem. 

\begin{rema} For any smooth oriented manifold $X$ of dimension $m$, we 
will denote by $(\mathcal{D}_d(X))_{0\leq d\leq m}$ the smooth compactly supported
$d$--forms on $X$ and by $(\mathcal{D}_d^\prime(X))_{0\leq d\leq m}$ its dual, i.e. the set of currents of dimension $d$ on $X$. We refer to appendix~\ref{a:currents} for a brief account on the theory of currents which is necessary in our proof. 
\end{rema}

\subsection{Representation of the conormal cycle.}\label{ss:def-conormal-cycle}
Our first goal is to represent the conormal cycle $N^*(\{f=0\})$ in terms of Dirac distributions. For that purpose, we first need to recall
the definition of push--forward in the context of currents. Let $\lambda:X\mapsto Y$ be a proper smooth map between two smooth oriented manifolds $X$ and $Y$, and
let $T$ be a current in $\mathcal{D}^{\prime}_d(X)$, where $d\leq \text{dim} Y$. Then the push--forward $\lambda_*(T)$ is defined by duality from the pull--back~:
\begin{eqnarray*}
\forall \omega\in\mathcal{D}_{d}(Y), \la \lambda_*T,\omega\ra_{Y}=\la T,\lambda^*\omega\ra_X.
\end{eqnarray*}
An important fact concerning the push--forward operation is that
an orientation of $T$ induces canonically an orientation
of $\lambda_*T$. 

\subsubsection{Abstract definition of the current of integration}

Let us start by recalling the definition of the conormal cycle of $\{f=0\}$~:
\begin{def1} Let $M$ be a smooth connected compact oriented manifold of dimension $n$ without boundary. Let $f\in C^\infty(M)$, if $d_xf$ never vanishes on
$\{f=0\}$ then
we define the conormal cycle
$[N^*(\{f=0\})]$
as the integration current on the conical Lagrangian
submanifold $\{(x;\xi) \text{ s.t. } f(x)=0\text{ and }\xi=td_xf\text{ for some }t\in \mathbb{R}\setminus \{0\} \}$ in $\mathcal{D}^\prime_n(T^{\bullet}M)$.
\end{def1}
A subtle fact we must add to our definition concerns the \textbf{orientation of the conormal cycle}.
First, note that the conormal cycle contains two components
$$N_{\pm}^*(\{f=0\})=\{(x;\xi) \text{ s.t. } f(x)=0\text{ and }\xi=td_xf\text{ for some }t\in \mathbb{R}_{\pm}\setminus\{0\} \} $$
since it does not meet the zero section of the cotangent bundle. Hence an orientation of $[N^*(\{f=0\})]$ 
consists of a choice of an orientation for each component separately. 
 In the sequel, we denote by $\mathbb{R}^*$ the set $\mathbb{R}\setminus \{0\}$.
Let us give another description of $[N^*(\{f=0\})]$ 
in terms of Lagrange immersion which will be useful in the sequel since we will interpret the conormal cycle $[N^*(\{f=0\})]$ as the push--forward of a cone $S\times \mathbb{R}^*$ and the orientation of $[N^*(\{f=0\})]$ will be induced by an orientation of $S\times \mathbb{R}^*$~:
\begin{lemm}\label{pushforwardlagrangeimmersion}
Let $f\in C^\infty(M)$ where $df$ never vanishes on
$\{f=0\}$. Set $S$ to be the abstract
manifold $\{f=0\}$ and $i:S\hookrightarrow M$ the canonical
immersion of $S$ in $M$. Consider the 
cone $S\times \mathbb{R}^*$. Then, the conormal cycle $N^*(\{f=0\})$ is the image
of the immersion~:
\begin{eqnarray}
\lambda~: (x;t)\in S\times \mathbb{R}^*\longmapsto (i(x);td_{i(x)}f)\in T^\bullet M,
\end{eqnarray}
and, at the level of currents, we find that
\begin{equation}
[N^*(\{f=0\})]=[\lambda (S\times\mathbb{R}^*)]=\lambda_*[S\times\mathbb{R}^*],
\end{equation}
where we orient the current $[S\times\mathbb{R}^*]$
by any differential form $\alpha\in \Omega^{n}( S\times \mathbb{R}^*)$ 
such that $df\wedge \alpha|_{S\times \mathbb{R}_{>0}}=\Omega_g\wedge dt|_{S\times \mathbb{R}_{>0}}$ and $df\wedge \alpha|_{S\times \mathbb{R}_{<0}}=-\Omega_g\wedge dt|_{S\times \mathbb{R}_{<0}}$.
\end{lemm}
\begin{rema}
The cone $S\times\mathbb{R}^*$ contains two components
$S\times \mathbb{R}_{>0}$ and $S\times \mathbb{R}_{<0}$ that we should orient separately.
\end{rema}

\begin{proof}
By definition, $[N^*(\{f=0\})]$ is just the integration current
on $N^*(\{f=0\})$ which is equal to $\lambda(S\times\mathbb{R}^*)$. Therefore, for any test form $\omega$ in $\ml{D}_n(T^{\bullet}M)$~:
$$\la [N^*(\{f=0\})],\omega\ra=\int_{N^*(\{f=0\})} \omega= 
\int_{\lambda(S\times\mathbb{R}^*)} \omega$$ $$=\int_{(S\times\mathbb{R}^*)}\lambda^*\omega=
\la [S\times\mathbb{R}^*],\lambda^*\omega\ra=\la \lambda_*[S\times\mathbb{R}^*],\omega\ra $$
by definition of the push--forward which yields the claim. 
\end{proof}

\subsubsection{Expression in local coordinates}

Our goal in this paragraph is to show that the conormal cycle $[N^*(\{f=0\})]$ can be written (in local coordinates) as the current
\begin{equation}
\int_{t\in\IR^*}\delta_0^{n+1} \left(f,\frac{t}{\Lambda}df-\xi\right)df\wedge\bigwedge_{i=1}^nd\left(\frac{t}{\Lambda}\partial_{x_i}f-\xi_i\right),
\end{equation}
where $\Lambda$ is any positive number -- see Prop.~\ref{p:Dirac} in appendix~\ref{a:currents} for the meaning of the convenient notation 
$\delta_0^{n+1}(f,td_xf/\Lambda-\xi)$. We need to do this progressively and the first step is to define
some integration current $[Z_{f,\Lambda}]$ on $\mathbb{R}^*\times T^{\bullet}M$ whose push--forward along the first factor
$\mathbb{R}^*$ yields the conormal cycle $[N^*(\{f=0\})]$. For that purpose, we set
\begin{equation}\label{Zflambda1}
Z_{f,\Lambda}:=\left\{(t;x,\xi)\in \mathbb{R}^*\times T^{\bullet}M: f(x)=0,\ \text{and}\ \xi=\frac{t}{\Lambda}d_xf\ \right\}.
\end{equation}
The following lemma holds:
\begin{lemm}\label{conormpushforward}
Let $f\in C^\infty(M)$ where $df$ never vanishes on $\{f=0\}$. Let $[Z_{f,\Lambda}]$ be the integration current on the set $Z_{f,\Lambda}$
defined by equation (\ref{Zflambda1}) and $\tilde{\pi}$ the projection
$(t;x,\xi)\in\mathbb{R}^*\times T^\bullet M\longmapsto (x,\xi) \in  T^\bullet M $. 
Then
\begin{equation}
[N^*(\{f=0\})]=\tilde{\pi}_*[Z_{f,\Lambda}].
\end{equation}
\end{lemm}
 \begin{proof}
Let $\omega(x,\xi;dx,d\xi)$ be an element in $\mathcal{D}^n(T^\bullet M)$. One has~:
$$\la\tilde{\pi}_*[Z_{f,\Lambda}],\omega\ra= \la[Z_{f,\Lambda}],\tilde{\pi}^*\omega\ra
=\int_{\{f(x)=0, \xi=td_xf,t\in\mathbb{R}^*\}} \omega(x,\xi;dx,d\xi)$$ by variable change.\\
Hence $\la\tilde{\pi}_*[Z_{f,\Lambda}],\omega\ra= \int_{\{f(x)=0,t\in\mathbb{R}^*\}} \omega(x,td_xf;dx,d(td_xf))
= \int_{S\times\mathbb{R}^*} (\lambda^*\omega)(t,x;dt,dx) $
since
$(\lambda^*\omega)(t,x;dt,dx)=\omega(x,td_xf;dx,d(td_xf))$, finally
$\la\tilde{\pi}_*[Z_{f,\Lambda}],\omega\ra=\lambda_*[S\times\mathbb{R}^*](\omega)=[N^*(\{f=0\})](\omega)$
by lemma \ref{pushforwardlagrangeimmersion}.
\end{proof}

The next lemma aims to represent the current $[Z_{f,\Lambda}]$
by Dirac distributions in a local coordinate chart. Let $(U,\phi)$ be a sufficiently small coordinate chart centered at a point $x_0$ in $M$. It induces a coordinate chart on $T^*M$ as follows:
$$\Phi:=T^*\phi:T^*U\rightarrow\phi(U)\times\IR^n,\quad (x,\xi)\mapsto (y,\eta)=(\phi(x),(d_x\phi^*)^{-1}\xi).$$
We also set 
$$\mathring{T}^*U:=\{(y,\eta)\in \phi(U)\times\IR^n:\eta\neq 0\}.$$
In the following of the article, we will use these conventions. In particular, $(x,\xi)$ will always denote a point in $T^*M$ and $(y,\eta)$ its image in local coordinates. Using these conventions, one has

\begin{lemm}
Let $f\in C^\infty(M)$ where $df$ never vanishes on
$\{f=0\}$. Then, for any system 
of local coordinates $(y;\eta)$ on $\mathring{T}^* U$ where $U\subset M$ is some bounded open subset, one 
finds that, in these local coordinates,
\begin{equation}
[Z_{f,\Lambda}]|_{\mathbb{R}^*\times  \mathring{T}^* U}=\delta^{n+1}_0\left(f(y),\frac{t}{\Lambda}d_yf-\eta\right)df\wedge\bigwedge_{i=1}^nd\left(\frac{t}{\Lambda}\partial_{y_i}f-\eta_i\right)
\end{equation}
\end{lemm}

\begin{proof}
The submanifold $Z_{f,\Lambda}$ in $\mathbb{R}^*\times T^\bullet U$
is defined in the local coordinates $(y,\eta)$ by the system of $(n+1)$ equations~:
\begin{eqnarray*}
f(y)&=&0\\
\frac{t}{\Lambda}\frac{\partial f}{\partial y^i}-\eta_i&=&0, 1\leqslant i\leqslant n.
\end{eqnarray*}
The collection of one forms
$\left(d_yf,d\eta_i-d\left(\frac{t}{\Lambda}\frac{\partial f}{\partial y^i}\right)\right)_{1\leqslant i\leqslant n}$
is linearly independent along $Z_{f,\Lambda}$ as
$$ d_yf\wedge \bigwedge_{i=1}^n \left(d\eta_i-d\left(\frac{t}{\Lambda}\frac{\partial f}{\partial y^i}\right)\right)=d_yf\wedge d\eta_1\wedge\ldots \wedge d\eta_n +\ \text{lower order terms in }d\eta,$$ 
where $d_yf\wedge d\eta_1\wedge\dots \wedge d\eta_n$ does not vanish on $Z_{f,\Lambda}$ since $df$ does not vanish on $\{f=0\}$. Then, by Corollary \ref{lemmadeltasubmanifolds}, 
the current
$[Z_{f,\Lambda}]$ is represented in local coordinates $(y,\eta)$ by the formula~:
\begin{equation}
[Z_{f,\Lambda}]|_{\mathbb{R}^*\times \mathring{T}^* U}=\delta^{n+1}_0 \left(f(y),\frac{t}{\Lambda}d_yf-\eta\right)d_yf\wedge\bigwedge_{i=1}^nd\left(\frac{t}{\Lambda}\partial_{y_i}f-\eta_i\right).
\end{equation}
\end{proof}

Then combining the above Lemma to Lemma \ref{conormpushforward} 
which claimed the identity $[N^*(\{f=0\})]=\tilde{\pi}_*[Z_{f,\Lambda}]$, we conclude
this first part by the \emph{first key ingredient of our proof}:
\begin{prop}\label{p:current-local-chart}
Let $f\in C^\infty(M)$ where $df$ never vanishes on
$\{f=0\}$. Then, for any system 
of local coordinates $(y,\eta)$ on $\mathring{T}^* U$ where $U\subset M$ is some bounded open subset, the
restriction of the conormal cycle
$[N^*(\{f=0\})]$ on $T^\bullet U$ can be described in the local coordinates $(y,\eta)$ by the
integral formula~:
\begin{equation}\label{e:ingredient1}
[N^*(\{f=0\})]|_{T^\bullet U}=\int_{t\in\mathbb{R}^*} \delta^{n+1}_0\left(f(y),\frac{t}{\Lambda}d_yf-\eta\right)d_yf\wedge\bigwedge_{i=1}^nd\left(\frac{t}{\Lambda}\partial_{y^i}f-\eta_i\right)
\end{equation}
where we use the convention
$\int_{t\in\mathbb{R}^*}=\int_0^{+\infty}+\int_{0}^{-\infty}$.
\end{prop}
Our convention for the integral
$\int_{t\in\mathbb{R}^*}=\int_0^{+\infty}+\int_{0}^{-\infty}$
reflects our choice of orientation
for the conormal cycle. We note that we use in fact the same convention as in~\cite[p.~682]{Ch45}.

\subsection{Pushforward of the Gaussian measure on $2$--jets over $U$.}

Thanks to~\eqref{e:ingredient1}, the conormal cycle of $\{f=0\}$ at a point $x\in U$ can be thought of as a functional 
of the $2$--jet of the function $f$ at $x\in U$. As we aim at computing its expectation with respect to the variable $f$, the pushforward of the measure $\mu_{\Lambda}$ on the $2$--jet bundle will 
naturally appear in our problem. The pushforward of this measure will then ``live'' on the space $\IR^{n^2+n+1}$ since we are working in local charts. From this point on, we will in fact make use 
of \emph{geodesic normal coordinate charts} for which one has proper asymptotics for the spectral projector associated to the Gaussian measure $\mu_{\Lambda}$ -- see appendix~\ref{derivativesspectralproj}.

\begin{rema}\label{r:geod-chart}
By geodesic normal coordinate chart, we mean the following. Consider the exponential map $\text{Exp}$ induced by the Riemannian metric on $M$. There exists a neighborhood $\ml{U}$ of $M\times\{0\}$ in $TM$ such that $\text{Exp}$ induces a smooth diffeomorphism $(x,v)\mapsto (x,\exp_x(v))$ from $\ml{U}$ onto a small neighborhood of the diagonal in $M\times M$. Then, we fix a point $x_0$ in $M$, and we consider the preimage of $T_{x_0} M$ under this diffeomorphism. This defines a local chart around $x_0$ into $T_{x_0}M$ which can be identified with $\IR^n$ via a measure preserving linear isomorphism.
\end{rema}

 We fix $(y,\eta,t)$ in $\mathring{T}^*U\times\IR^*$ and $m\geq 1$. We define the following map from $\ml{H}_{\Lambda}$ to $\IR^{n^2+n+1}$:
$$J_{\Lambda}^2(y):f\longmapsto F=\left(f(y);\left(\frac{\partial_{y^l}f(y)}{\Lambda}\right)_{1\leq l\leq n};\left(\frac{\partial_{y^ry^s}^2f(y)}{\Lambda^2}\right)_{1\leq r, s\leq n}\right).$$ 
For the coordinates in $\IR^{n^2+n+1}$, we set
$$F= (F_0;F_1,\ldots , F_n; F_{1,1},\ldots F_{n,1}, F_{1,2},\ldots F_{n,2},\ldots, F_{1,n}, F_{2,n},\ldots, F_{n,n}).$$ 
The next Proposition describes the Fourier
transform of the pushforward of the Gaussian measure $\mu_{\Lambda}$ under the map $J_{\Lambda}^2(y)$. It follows directly from the 
results on the asymptotics of the derivatives of the spectral projector recalled in appendix
~\ref{derivativesspectralproj} (more precisely Corollary~\ref{r:useful-asymp}).
\begin{prop}\label{p:pushforwardGaussmeasure}
The pushforward of the Gaussian measure $\mu_{\Lambda}$ under the map $J_{\Lambda}^2(y)$ induces a Gaussian measure
$\tilde{\nu}_{\Lambda}^y(F)$ on $\IR^{n^2+n+1}$:
\begin{equation}\label{e:pushforward-measure}\tilde{\nu}_{\Lambda}^y(F):=\left(J_{\Lambda}^2(y)_*\mu_{\Lambda}\right)(F),\end{equation}
whose Fourier transform has the following expression:
\begin{equation}\label{e:pushforward-fourier}
\int_{\IR^{n^2+n+1}}e^{-2i\pi T.F}\tilde{\nu}_{\Lambda}^y(F)=e^{-\frac{2\pi^2}{\text{Vol}_g(M)}\la T, A_{\Lambda}(y)T\ra}.
\end{equation}
where the matrix $A_{\Lambda}(y)$ is defined as follows~: 
$$A_{\Lambda}(y):=A_0+\ml{O}(\Lambda^{-1})=\left(\begin{array}{cc} A_{\Lambda}^{1,1}(y) & A_{\Lambda}^{1,2}(y)\\
              A_{\Lambda}^{2,1}(y) & A_{\Lambda}^{2,2}(y)
             \end{array}\right),$$
with the constant in the remainder uniformly bounded in terms of $y$ and $\Lambda$, and 
$$A_0:=\left(\begin{array}{cc} A_{1,1} & A_{1,2}\\
              A_{2,1} & A_{2,2}
             \end{array}\right),
$$
where the matrices $A_{i,j}$ do not depend on $y$ and are defined as follows. 

$A_{1,1}$ is an $(n+1)\times (n+1)$ matrix, $A_{1,2}$ is an $(n+1)\times n^2$ matrix, $A_{2,1}$ is an $n^2\times (n+1)$ matrix, and $A_{2,2}$ is an $n^2\times n^2$ matrix. Their expressions are given by
$$A_{1,1}:=\left(\begin{array}{cc} 1 & 0_{1\times n}\\
              0_{n\times 1} & \frac{1}{n+2} \text{Id}_n
             \end{array}\right),\quad A_{2,1}=A_{1,2}^T=-\frac{1}{n+2}\left( \left(\delta_{r,s}\right)_{1\leq r,s\leq n},
              0_{n^2\times n} \right),$$
and
$$A_{2,2}:=\frac{1}{(n+2)(n+4)}\left(B_{s,s'}\right)_{1\leq s, s'\leq n},$$
with $B_{s,s'}$ an $n\times n$ matrix which is equal to
$$B_{s,s}:=\text{Id}+2\ \text{diag}((\delta_{j,s})_{1\leq j\leq n})\ \text{if}\ s=s',$$
and 
$$B_{s,s'}=B_{s',s}^T:=\left(\delta_{i,s}\delta_{j,s'}+\delta_{i,s'}\delta_{j,s}\right)_{1\leq i,j\leq n}\ \text{if}\ s\neq  s'.$$
\end{prop}

\subsection{Intertwining the orders of integration.}

In our proof, we will need to integrate functions involving the variables $(y,\eta,t)$ in $\mathring{T}^*U\times\IR^*$ but also $f$ in $\ml{H}_{\Lambda}$. 
The fact that we can ``intertwine the order of integration'' plays a central role in the following -- see e.g. the proof of Proposition~\ref{l:L1}. 
Recall from paragraph~$2.3$ in~\cite{Le14} that
$$D_{\Lambda}:=\left\{ f\in\ml{H}_{\Lambda}:\ \exists y\in f^{-1}(0)\ \text{such that}\ d_yf=0\right\}.$$
is of $\mu_{\Lambda}$-measure $0$ for $\Lambda$ large enough. One can also verify that $\Omega_{\Lambda}=\ml{H}_{\Lambda}\backslash D_{\Lambda}$ 
is \emph{an open subset} of $\ml{H}_{\Lambda}$. For every $f$ in $\ml{H}_{\Lambda}$, we define the following map on $\ml{H}_{\Lambda}\times \phi(U)\times\IR^n\times\IR^*$:
$$G_{\Lambda}:(f,y,\eta,t)\in \ml{H}_{\Lambda}\times \phi(U)\times\IR^n\times\IR^*\longmapsto\left(f(y);\frac{t}{\Lambda}\partial_{y_1}f-\eta_1,\ldots,\frac{t}{\Lambda}\partial_{y_n}f-\eta_n\right)\in\IR\times \IR^n,$$
Then, one has

\begin{prop}\label{intertwiningproposition}
Let $\Psi(f,y,\eta,t)$ be a test function in $\ml{D}(\Omega_{\Lambda}\times\mathring{T}^*U\times\IR^*)$, then
\begin{eqnarray}\label{e:fubini-distribution}
 \left\la 1,\left\la G_{\Lambda}(f)^*(\delta_0^{n+1}),\Psi\right\ra_{\mathring{T}^*U\times\IR^*}\right\ra_{\Omega_{\Lambda}} & = & \left\la G_{\Lambda}^*(\delta_0^{n+1}),\Psi\right\ra_{\Omega_{\Lambda}\times \mathring{T}^*U\times\IR^*}\\
\nonumber & =  & \left\la 1,\left\la G_{\Lambda}(y,\eta,t)^*(\delta_0^{n+1}),\Psi\right\ra_{\Omega_{\Lambda}}\right\ra_{\mathring{T}^*U\times\IR^*}.
\end{eqnarray}
\end{prop}

\begin{proof}
See section~\ref{intertwiningsectionappendix} of the appendix.
\end{proof}

\section{Proof of Theorems~\ref{integtheo} and~\ref{asympttheo}}
\label{s:proof}

Recall that we aim at computing the expectation of the current of integration on the submanifold $N^{*}(\{f=0\})$. According to Proposition~\ref{p:current-local-chart} which gives an integral representation
of the conormal cycle, 
one knows that, for every $f$ in $\Omega_{\Lambda}$, and for every $\omega$ in $\ml{D}_n(\mathring{T}^*U)$, one has
{\small
\begin{eqnarray*}
\frac{1}{\Lambda^n}\int_{N^{*}(\{f=0\})}\omega(y,dy,\eta,d\eta)=\left\la G_{\Lambda}(f)^*(\delta_{0}^{n+1}),df\wedge\bigwedge_{j=1}^nd\left(\frac{t}{\Lambda}\partial_{y_j}f-\eta_j\right)\wedge\omega(y,\eta,dy,d\eta)\right\ra_{\mathring{T}^*U\times\IR^*},
\end{eqnarray*}
}
where $\left\la.,.\right\ra_{\mathring{T}^*U\times\IR^*}$ is the duality bracket in $\ml{D}'(\mathring{T}^*U\times\IR^*)\times\ml{D}(\mathring{T}^*U\times\IR^*)$. Recall that we chose a particular orientation for $\IR^*$ in paragraph~\ref{ss:def-conormal-cycle}.

The main result of this section is the following Theorem from which Theorems~\ref{integtheo} and~\ref{asympttheo} follow by partition of unity:
\begin{theo}\label{p:expectation-local-chart} We use the above notations. For any element 
$$\omega:=\sum_{|\alpha|+|\beta|=n}\omega_{\alpha,\beta}(y,\eta)dy^{\alpha}\wedge d\eta_{\beta}$$
in $\ml{D}_n(\mathring{T}^*U)$, let
$$\Vert\omega\Vert=
\sup_{\alpha,\beta,(y,\eta)\in \mathring{T}^*U} \vert\omega_{\alpha,\beta}(y,\eta)\vert.$$   
For every compact set $\ml{K}\subset \mathring{T}^*U$, there exists a constant $C_{\ml{K}}$ such that 
for all test form $\omega$ supported in $\ml{K}$
\begin{eqnarray*}
\int_{\ml{H}_{\Lambda}} \left| \int_{N^{*}(\{f=0\})}\omega(y,dy,\eta,d\eta)\right| d\mu_{\Lambda}\leqslant C_{\ml{K}}\Vert\omega\Vert
\end{eqnarray*} 
and 
\begin{eqnarray*}
&&\int_{\ml{H}_{\Lambda}}\left(\int_{N^{*}(\{f=0\})}\omega(y,dy,\eta,d\eta)\right)d\mu_{\Lambda}(f)\\&=&C_n\left(\frac{\Lambda}{\sqrt{n+2}}\right)^n\int_{\mathring{T}^*U}dy^1\wedge dy^2\ldots\wedge dy^n\wedge \omega(y,dy,\eta,d\eta)+\ml{O}(\Lambda^{n-1}),
\end{eqnarray*}
with $C_n=0$ if $n\equiv 0\ (\text{mod}\ 2)$ and 
$$C_n:=\frac{2(-1)^{\frac{n+1}{2}}}{\pi\operatorname{Vol}(\mathbb{S}^{n-1})}$$
if $n\equiv 1\ (\text{mod}\ 2)$.
\end{theo}

\begin{rema}
It follows from the first statement of Theorem \ref{p:expectation-local-chart} that for every test form $\omega$, the map
$$f\mapsto \int_{N^{*}(\{f=0\})}\omega(y,dy,\eta,d\eta)$$
belongs to $L^1(\ml{H}_{\Lambda},d\mu_{\Lambda})$. 
For any integral current $T$, let $\mathbb{M}(T)$ denote the \emph{mass} of the current $T$~\cite{GiMoSou98}.
Our result is in fact slightly stronger than what we claimed in Theorem~\ref{integtheo} as it also means that, for any smooth compactly supported cut--off function $\varphi\in \mathcal{D}^0(T^\bullet M)$, 
$$\int_{\ml{H}_{\Lambda}}\mathbb{M}\left(\varphi[N^*\left(\{f=0\}\right)]\right) <+\infty. $$
In other words, the mass of the cut--off conormal cycle
$\varphi[N^*\left(\{f=0\}\right)]$ is a $L^1$ function of $f\in\mathcal{H}_\Lambda$.
\end{rema}

The purpose of this section is to prove this Theorem. We emphasize that \emph{there are two parts in this statement}. On the one hand, we have to show the integrability property and on the other hand, we have to compute the precise value of the expectation. The first part of the statement requires a delicate analysis which is carried out in the first part of this section. After that step, we can combine this first part to some combinatorial arguments in order to obtain the second part of the Theorem. 

More precisely, the proof is organized as follows. First, in paragraph~\ref{ss:prel-simpl}, we write a formal asymptotic expansion in powers of $\Lambda^{-1}$ of
$$J(f,\omega,U):=\frac{1}{\Lambda^n}\int_{N^{*}(\{f=0\})}\omega(y,dy,\eta,d\eta).$$ 
Paragraph~\ref{ss:integrability} is the more delicate part where we prove the first part of Theorem~\ref{p:expectation-local-chart} concerning the integrability of the conormal cycle
as a function of $f\in\mathcal{H}_\Lambda$ with respect to the measure $d\mu_\Lambda$. 
After that, in paragraph~\ref{ss:leading}, we compute explicitely the expression of the leading term and in the remaining paragraphs, we compute the combinatorial constant appearing in the leading term.

\subsection{Preliminary simplification}\label{ss:prel-simpl}
Before getting into the details of the proof of Theorem~\ref{p:expectation-local-chart}, we start by making a few reduction that will make the calculation slightly simpler:
\begin{lemm}\label{l:approx-dirac} Let 
$$\omega:=\sum_{|\alpha|+|\beta|=n}\omega_{\alpha,\beta}(y,\eta)dy^{\alpha}\wedge d\eta_{\beta}$$
be an element in $\ml{D}_n(\mathring{T}^*U)$. With the above conventions, one has, for every $\Lambda$ and for every $f$ in $\Omega_{\Lambda}$,
\begin{eqnarray*}
&&\frac{1}{\Lambda^n}\int_{N^{*}(\{f=0\})}\omega(y,dy,\eta,d\eta)\\
&=&-\sum_{k=0}^{n-1}\Lambda^{-k}\left\la G_{\Lambda}(f)^*(\delta_{0}^{n+1}),t^{n-3-k}P_k\left((\eta_j)_j,(\partial^2_{y^jy^{l}}f/\Lambda^2)_{j,l},(\omega_{\alpha,\beta})_{\alpha,\beta}\right)\right\ra_{\mathring{T}^*U\times\IR^*},
\end{eqnarray*}
where 
\begin{itemize}
 \item for every $0\leq k\leq n-1$, $P_k(R,S,T)$ is a polynomial which does not depend on $\Lambda$, $t$ and $f$,
 \item for every $0\leq k\leq n-1$, $P_k(R,S,T)$ is homogeneous of degree $2$ in the variables $R$, homogeneous of degree $n-1-k$ in the variables $S$, homogeneous of degree $1$ in the $T$ variables,
 \item in the case $k=0$, one has
 $$P_0\left((R_j)_j,(S_{j,l})_{j,l},(T_{\alpha,\beta})_{\alpha,\beta}\right)=\left(\sum_{p=1}^n\sum_{\sigma\in S_n}\varepsilon(\sigma)R_pR_{\sigma(p)}\prod_{j\neq p}S_{j,\sigma(j)}\right)T_{(0,\ldots,0),(1,\ldots,1)}.$$
\end{itemize}
\end{lemm}

At first sight, this lemma seems rather technical; yet, writing the integrals this way will simplify the presentation afterwards. This lemma gives at least a formal expansion in powers of $\Lambda^{-1}$, and we will verify in the following paragraphs that each term in the sum is in fact integrable and has uniformly bounded $L^1$ norm in $\ml{H}_{\Lambda}$.


\begin{proof} Write first that
$$\frac{1}{\Lambda^n}\int_{N^{*}(\{f=0\})}\omega(y,dy,\eta,d\eta)=\sum_{p=1}^n(-1)^{p}\frac{1}{\Lambda^n}\int_{\IR_+}I^{(p)}(f,t)dt,$$
where 
$$I^{(p)}(f,t)=\left\la G_{\Lambda}(f)^*(\delta_{0}^{n+1}), \frac{\partial_{y_p}f}{\Lambda} df\wedge\bigwedge_{j\neq p}d\left(\frac{t}{\Lambda}\partial_{y_j}f-\eta_j\right)\wedge\omega(y,\eta,dy,d\eta)\right\ra_{\mathring{T}^*U}.$$
Expanding the $df$ term in the wedge product, one can also write
$$\frac{1}{\Lambda^n}\int_{N^{*}(\{f=0\})}\omega(y,dy,\eta,d\eta)=\sum_{p,q=1}^n(-1)^{n+p-1}\int_{\IR_+}I^{(p,q)}(f,t)dt,$$
where
$$I^{(p,q)}(f,t):=\left\la G_{\Lambda}(f)^*(\delta_{0}^{n+1}), \frac{ \partial_{y_p}f}{\Lambda}\frac{ \partial_{y_q}f}{\Lambda}\bigwedge_{j\neq p}d\left(\frac{t}{\Lambda^2}\partial_{y_j}f-\frac{\eta_j}{\Lambda}\right)\wedge dy_q\wedge\omega(y,\eta,dy,d\eta)\right\ra_{\mathring{T}^*U}.$$
Thanks to the definition of $G_{\Lambda}(f)^*(\delta_{0}^{n+1})$, this can be rewritten as
$$I^{(p,q)}(f,t)=\frac{1}{t^2}\left\la G_{\Lambda}(f)^*(\delta_{0}^{n+1}),\eta_p\eta_q\bigwedge_{j\neq p}d\left(\frac{t}{\Lambda^2}\partial_{y_j}f-\frac{\eta_j}{\Lambda}\right)\wedge dy_q\wedge\omega(y,\eta,dy,d\eta)\right\ra_{\mathring{T}^*U}.$$
Write now
$$\omega:=\sum_{|\alpha|+|\beta|=n}\omega_{\alpha,\beta}(y,\eta)dy^{\alpha}\wedge d\eta_{\beta},$$
where $\omega_{\alpha,\beta}$ is compactly supported in $\mathring{T}^{*}U$. By expanding the wedge product involved in the integrals, we can rearrange the sum in the following way:
\begin{eqnarray*}
&&\frac{1}{\Lambda^n}\int_{N^{*}(\{f=0\})}\omega(y,dy,\eta,d\eta)\\
&=&-\sum_{k=0}^{n-1}\Lambda^{-k}\left\la G_{\Lambda}(f)^*(\delta_{0}^{n+1}),t^{n-3-k}P_k\left((\eta_j)_j,(\partial^2_{y^jy^l}f/\Lambda^2)_{j,l},(\omega_{\alpha,\beta})_{\alpha,\beta}\right)\right\ra_{\mathring{T}^*U\times\IR^*},
\end{eqnarray*}
where 
\begin{itemize}
 \item for every $0\leq k\leq n-1$, $P_k(R,S,T)$ is a polynomial which does not depend on $\Lambda$, $t$ and $f$,
 \item for every $0\leq k\leq n-1$, $P_k(R,S,T)$ is homogeneous of degree $2$ in the variables $R$, homogeneous of degree $n-1-k$ in the variables $S$, and homogeneous of degree $1$ in the $T$ variables.
 \item in the case $k=0$, one has
 $$P_0\left((R_j)_j,(S_{j,l})_{j,l},(T_{\alpha,\beta})_{\alpha,\beta}\right)=\left(\sum_{p=1}^n\sum_{\sigma\in S_n}\varepsilon(\sigma)R_pR_{\sigma(p)}\prod_{j\neq p}S_{j,\sigma(j)}\right)T_{(0,\ldots,0),(1,\ldots,1)}.$$
\end{itemize}
This concludes the proof of the lemma.
\end{proof}

\subsection{Notations and conventions.}\label{r:convention}
In the following, we fix the following convention, for every $0\leq k\leq n-1$,
$$J^{(k)}(f,\omega,U)=\left\la G_{\Lambda}(f)^*(\delta_{0}^{n+1}),t^{n-3-k}P_k\left((\eta_j)_j,(\partial^2_{y^jy^l}f/\Lambda^2)_{j,l},(\omega_{\alpha,\beta})_{\alpha,\beta}\right)\right\ra_{\mathring{T}^*U\times\IR^*}.$$
In particular, one has
\begin{equation}\label{e:asymp-exp-current}J(f,\omega,U):=\frac{1}{\Lambda^n}\int_{N^{*}(\{f=0\})}\omega(y,dy,\eta,d\eta)=-\sum_{k=0}^{n-1}\Lambda^{-k}J^{(k)}(f,\omega,U).\end{equation}
In the following, the letter $J$ (resp. $K$, resp. $L$) will denote functionals of $(y,\eta,t,f)$ that have been integrated against the variables $(y,\eta,t)$ (resp. $(y,\eta,t,f)$, resp. $f$).

\subsection{Integrability}\label{ss:integrability}

For any element 
$\omega:=\sum_{|\alpha|+|\beta|=n}\omega_{\alpha,\beta}(y,\eta)dy^{\alpha}\wedge d\eta_{\beta}$
in $\ml{D}_n(\mathring{T}^*U)$, let
$$\Vert\omega\Vert=
\sup_{\alpha,\beta,(y,\eta)\in \mathring{T}^*U} \vert\omega_{\alpha,\beta}(y,\eta)\vert.$$  
Using the conventions of paragraph~\ref{r:convention}, we will now prove the first part of Theorem~\ref{p:expectation-local-chart} which follows immediately from the next proposition
\begin{prop}\label{l:L1} Let $\omega$ be an element in $\ml{D}_n(\mathring{T}^*U)$. Then, for every $0\leq k\leq n-1$, one has
$$f\mapsto J^{(k)}(f,\omega,U)$$
 belongs to $L^1(\ml{H}_{\Lambda},d\mu_{\Lambda})$ and for every
compact set $\mathcal{K}$, there exists a constant $C_{\mathcal{K},k}>0$ such that
for any test form $\omega$ in $\ml{D}_n(\mathring{T}^*U)$ supported in $\mathcal{K}$,
 $$\int_{\Omega_{\Lambda}} |J^{(k)}(f,\omega,U)|d\mu_{\Lambda}(f)\leq C_{\mathcal{K},k}\Vert\omega\Vert.$$
\end{prop}

Showing this proposition is our more delicate task from the analytical point of view as we have to justify carefully several inversions and convergences of integrals.

\begin{rema}\label{r:leading-term} As was already mentionned, a direct consequence of this lemma is that, for every $\omega$ in $\ml{D}_n(\mathring{T}^*U)$,
$$f\mapsto \frac{1}{\Lambda^n}\int_{N^{*}(\{f=0\})}\omega(y,dy,\eta,d\eta)$$
belongs to $L^1(\ml{H}_{\Lambda},d\mu_{\Lambda})$ which is exactly the content of the first part of Theorem~\ref{p:expectation-local-chart}. In fact, we have something slightly more precise in the sense that 
$$\int_{\Omega_{\Lambda}}\left(\int_{N^{*}(\{f=0\})}\omega(y,dy,\eta,d\eta)\right)d\mu_{\Lambda}(f)=-\Lambda^n\int_{\Omega_{\Lambda}}J^{(0)}(f,\omega,U)d\mu_{\Lambda}(f)+\ml{O}(\Lambda^{n-1}).$$
Recall from paragraph~\ref{r:convention} that
{\small
\begin{eqnarray*}
J^{(0)}(f,\omega,U):=\sum_{p=1}^n\sum_{\sigma\in S_n}\varepsilon(\sigma)\left\la G_{\Lambda}(f)^*(\delta_0^{n+1}),t^{n-3}\eta_p\eta_{\sigma(p)}\left(\prod_{j\neq p}\frac{\partial^2_{y^j,y^{\sigma(j)}}f}{\Lambda^2}\right)\omega_{(0,\ldots,0),(1,\ldots,1)}(y,\eta)\right\ra_{\mathring{T}^*U\times\IR^*}.
\end{eqnarray*}
}
In particular, the second part of Theorem~\ref{p:expectation-local-chart} is reduced to the computation of the expectation of $J^{(0)}(f,\omega,U)$ with respect to the Gaussian measure $d\mu_{\Lambda}(f)$. This computation will be performed in paragraphs~\ref{ss:leading} and~\ref{ss:combinatorics}.
\end{rema}

We will now prove proposition~\ref{l:L1} in several steps. At several stages of the proof, we will use results on pullback of distributions which are discussed in appendix~\ref{intertwiningsectionappendix}.

\subsubsection{Proof of proposition~\ref{l:L1} -- First step}
As a first step, we will give a crude upper bound on $|J^{(k)}(f,\omega,U)|$. Let $\tilde{\omega}$ be a smooth function compactly supported on $\mathcal{K}\subset\mathring{T}^*U$ such that, for every multiindices $(\alpha,\beta)$ satisfying $|\alpha|+|\beta|=n$, and for every $1\leq p,q\leq n$, one has
$$\forall\ (y,\eta)\in\mathring{T}^*U,\ |\omega_{\alpha,\beta}(y,\eta)\eta_p\eta_q|\leq\tilde{\omega}(y,\eta).$$
Moreover, we choose $\tilde{\omega}$ so that
$\sup_{(y,\eta)\in \mathring{T}^*U} \tilde{\omega}(y,\eta)\leq C \Vert\omega\Vert$ where $C>0$ is a constant depending only on the support of $\omega$, and  where 
$\Vert.\Vert$ is the norm on smooth forms defined 
at the beginning of paragraph \ref{ss:integrability}. 
As explained in Remark~\ref{r:positive}, $G_{\Lambda}(f)^*(\delta_0^{n+1})$ 
is a positive distribution on $\mathring{T}^*U\times\IR^*$ hence a positive Radon measure $G_{\Lambda}(f)^*(\delta_0^{n+1})(d^ny,d^n\eta, dt)$ on $\mathring{T}^*U\times\IR^*$. 
In particular, one has
\begin{equation}\label{e:triang-ineq}
|J^{(k)}(f,\omega,U)|\leq\int_{\mathring{T}^*U\times\IR^*}\tilde{\omega}(y,\eta)\vert t\vert^{n-3-k}Q_k\left(\left(\left|\partial^2_{y^jy^l}f/\Lambda^2\right|\right)_{j,l}\right)G_{\Lambda}(f)^*(\delta_0^{n+1})(d^ny,d^n\eta ,dt),
\end{equation}
for some homogeneous polynomial $Q_k$ of degree $n-1-k$ with nonnegative coefficients and which is independent of $\omega$, $U$, $(y,\eta,t)$ and $\Lambda$. Then, according to~\eqref{e:triang-ineq} and to proposition~\ref{intertwiningproposition}, the integrability of $J^{(k)}(f,\omega,U)$ with respect to $d\mu_{\Lambda}$ will follow from the integrability of the map
$$(y,\eta,t,f)\in\mathring{T}^*U\times\IR^*\times\Omega_{\Lambda}\longmapsto \tilde{\omega}(y,\eta)\vert t\vert^{n-3-k}Q_k\left(\left(\left|\partial^2_{y^jy^l}f/\Lambda^2\right|\right)_{j,l}\right)\in\IR_+$$
with respect to the Radon measure 
$\mu_{\Lambda}(f)G_{\Lambda}^*(\delta_0^{n+1})(df,d^ny,d^n\eta, dt)$ where
$$\mu_{\Lambda}(f):=e^{-\frac{N(\Lambda)\|f\|^2}{2}}\left(\frac{N(\Lambda)}{2\pi}\right)^{\frac{N(\Lambda)}{2}}.$$ 
Thus, according to~\eqref{e:triang-ineq} and in order to prove the lemma, it remains to show that
$$\int_{\mathring{T}^*U\times\IR^*\times\Omega_{\Lambda}}\tilde{\omega}(y,\eta)\vert t\vert^{n-3-k}Q_k\left(\left(\left|\partial^2_{y^jy^l}f/\Lambda^2\right|\right)_{j,l}\right)\mu_{\Lambda}(f)G_{\Lambda}^*(\delta_0^{n+1})(df,d^ny,d^n\eta ,dt)$$
is uniformly bounded with respect to $\Lambda$. Let $\mathbf{j}:=(j_1,j_2,\ldots, j_{2(n-k)-3},j_{2(n-k-1)})$ be an element in $\{1,\ldots,n\}^{2(n-1-k)}$. As $Q_k$ is an homogeneous polynomial of degree $n-1-k$, it is sufficient to prove that
\begin{equation}\label{e:qty-to-bound}K^{(k)}(\mathbf{j},\omega,U):=\int_{\mathring{T}^*U\times\IR^*\times\Omega_{\Lambda}}\tilde{\omega}(y,\eta)\vert t\vert^{n-3-k}\prod_{p=1}^{n-1-k}\left|\frac{\partial^2_{y^{j_{2p-1}}y^{j_{2p}}}f}{\Lambda^2}\right|\mu_{\Lambda}(f)G_{\Lambda}^*(\delta_0^{n+1})(df,d^ny,d^n\eta, dt)\end{equation}
is uniformly bounded with respect to $\Lambda$. 

\subsubsection{Proof of proposition~\ref{l:L1} -- Step 2: Intertwining integrals} We can now use Proposition~\ref{intertwiningproposition} in order to intertwine the order of integration, and to integrate first with respect to the $f$ variable. For that purpose, we define the following nonnegative Radon measure on $\Omega_{\Lambda}$:
$$\nu_{\Lambda,1}^{(y,\eta,t)}(df)=
\mu_\Lambda(f)
G_{\Lambda}(y,\eta,t)^*(\delta_0^{n+1})(df),$$
and
$$\nu_{\Lambda,2}^{(y,\eta,t)}(df)=\prod_{p=1}^{n-1-k}\left|\frac{\partial^2_{y^{j_{2p-1}}y^{j_{2p}}}f(y)}{\Lambda^2}\right|^2\mu_\Lambda(f) G_{\Lambda}(y,\eta,t)^*(\delta_0^{n+1})(df).$$
Then, combining~\eqref{e:qty-to-bound} and the ``Fubini identity''~\eqref{e:fubini-distribution} of Proposition~\ref{intertwiningproposition} to H\"older inequality, one finds that~:
{\small
\begin{eqnarray*}
K^{(k)}(\mathbf{j},\omega,U)&= &\int_{\mathring{T}^*U\times\IR^*\times\Omega_{\Lambda}}\tilde{\omega}(y,\eta)\vert t\vert^{n-3-k}\prod_{p=1}^{n-1-k}\left|\frac{\partial^2_{y^{j_{2p-1}}y^{j_{2p}}}f}{\Lambda^2}\right|\mu_{\Lambda}(f)G_{\Lambda}^*(\delta_0^{n+1})(df,d^ny,d^n\eta ,dt) \\
 &=& \int_{\mathring{T}^*U\times\IR^*}\tilde{\omega}(y,\eta)\vert t\vert^{n-3-k}\left(\int_{\Omega_\Lambda}\prod_{p=1}^{n-1-k}\left|\frac{\partial^2_{y^{j_{2p-1}}y^{j_{2p}}}f}{\Lambda^2}\right|\mu_{\Lambda}(f)G_{\Lambda}(y,\eta,t)^*(\delta_0^{n+1})(df)\right)d^nyd^n\eta dt\\
&\leq & \int_{\mathring{T}^*U\times\IR^*}\tilde{\omega}(y,\eta)\vert t\vert^{n-3-k}\sqrt{\nu_{\Lambda,1}^{(y,\eta,t)}\left(\Omega_{\Lambda}\right)}\sqrt{\nu_{\Lambda,2}^{(y,\eta,t)}\left(\Omega_{\Lambda}\right)}d^nyd^n\eta dt.
\end{eqnarray*}
}
Finally, we get that~:
\begin{equation}\label{e:CS}
K^{(k)}(\mathbf{j},\omega,U)\leq \int_{\mathring{T}^*U\times\IR^*}\tilde{\omega}(y,\eta)\vert t\vert^{n-3-k}\sqrt{\nu_{\Lambda,1}^{(y,\eta,t)}\left(\Omega_{\Lambda}\right)}\sqrt{\nu_{\Lambda,2}^{(y,\eta,t)}\left(\Omega_{\Lambda}\right)}d^nyd^n\eta dt.
\end{equation}

The problem is that we do not know a priori if $(y,\eta,t)\mapsto\nu_{\Lambda,1}^{(y,\eta,t)}\left(\Omega_{\Lambda}\right)$ and $(y,\eta,t)\mapsto \nu_{\Lambda,2}^{(y,\eta,t)}\left(\Omega_{\Lambda}\right)$ 
are integrable functions of $(y,\eta,t)$. In particular, the r.h.s. of~\eqref{e:CS} can a priori be infinite. 
We will now compute $\nu_{\Lambda,q}^{(y,\eta,t)}\left(\Omega_{\Lambda}\right)$ for $q=1,2$ and show that it is in fact uniformly bounded in terms of $\Lambda$ by an integrable function. We note that the case $q=1$ is a particular case of the case $q=2$ (when $k=n-1$). Thus, we only need to compute $\nu_{\Lambda,2}^{(y,\eta,t)}\left(\Omega_{\Lambda}\right)$. For that purpose, we introduce the following regularization of the Dirac distribution, for every fixed $t$ in $\IR^*$,
$$\delta_{t,\frac{1}{m}}^{n+1}(F_0,F_1,\ldots,F_n):=\int_{\IR^{n+1}}e^{-\frac{|\tau_0|^2+t^2|\tau|^2}{m}}e^{-2i\pi(\tau_0,\tau).F}d\tau_0d\tau.$$

\begin{rema}\label{r:continuity-regul}
We denote by $\Gamma=\{0\}\times(\IR^{n+1})^*$ the wave front of $\delta_0^{n+1}$. Using the conventions of paragraph $8.2$ in~\cite{Ho90}, $\delta_0^{n+1}$ belongs to the space $\ml{D}_{\Gamma}'(\IR^{n+1})$. One can verify that $(\delta_{t,\frac{1}{m}}^{n+1})_{m\geq 1}$ sequentially converges to $\delta_0^{n+1}$ for the topology on $\ml{D}'_{\Gamma}(\IR^{n+1})$ for every fixed $t$ in $\IR^*$.
Combining the sequential continuity of the pull-back operation (Th.~$8.2.4$ in~\cite{Ho90}, see also~\cite{BrDaHe14}) to lemma~\ref{l:ample}, we find that $(G_{\Lambda}(y,\eta,t)^*(\delta_{t,\frac{1}{m}}^{n+1}))_{m\geq 1}$ weakly converges to $G_{\Lambda}(y,\eta,t)^*(\delta_0^{n+1})$ for the topology on $\ml{D}'(\IR^{N(\Lambda)})$ and for every fixed $(y,\eta,t)$ in $\mathring{T}^*U\times\IR^*$.
\end{rema}
We can then define the following sequences of regularized positive Radon measures, for every $m\geq 1$:
\begin{equation}\label{defimeasurenu}
\nu_{\Lambda,2,m}^{(y,\eta,t)}(df)=\prod_{p=1}^{n-1-k}\left|\frac{\partial^2_{y^{j_{2p-1}}y^{j_{2p}}}f(y)}{\Lambda^2}\right|^2 \mu_\Lambda(f)G_{\Lambda}(y,\eta,t)^*(\delta_{t,\frac{1}{m}}^{n+1})(f)df,
\end{equation}
We will now verify that $(\nu_{\Lambda,2,m}^{(y,\eta,t)}(\Omega_{\Lambda}))_{m\geq 1}$ is uniformly bounded in terms of $m$ as follows:
\begin{equation}\label{e:mass-measure}
\forall t>0,\ \nu_{\Lambda,2,m}^{(y,\eta,t)}(\Omega_{\Lambda})\leq \frac{C_0}{t^n}e^{-\frac{1}{C_0 t^2}},
\end{equation}
where the constant $C_0$ is uniform for $(y,\eta)$ in the support of $\tilde{\omega}$, for $0<m\leq 1$ and $\Lambda>0$ large enough. Then, combining this to remark~\ref{r:continuity-regul}, one can deduce that $\nu_{\Lambda,2}^{(y,\eta,t)}(\Omega_{\Lambda})$ is bounded by the same quantity and thanks to~\eqref{e:CS}, we can conclude that $K^{(k)}(\mathbf{j},\omega,U)$ is uniformly bounded in terms of $\Lambda$ which was exactly what we were aiming for.

\subsubsection{Proof of proposition~\ref{l:L1} -- Final step: Gaussian integral bounds on $\nu_{\Lambda,2}^{(y,\eta,t)}(\Omega_{\Lambda})$ } \label{sss:gaussian-matrix}
It remains to prove~\eqref{e:mass-measure}.
Observe by the defining equation~(\ref{defimeasurenu}) of $\nu_{\Lambda,2,m}^{(y,\eta,t)}$ that $$ \nu_{\Lambda,2,m}^{(y,\eta,t)}(\Omega_{\Lambda})=\int_{\Omega_\Lambda}\prod_{p=1}^{n-1-k}\left|\frac{\partial^2_{y^{j_{2p-1}}y^{j_{2p}}}f(y)}{\Lambda^2}\right|^2 \mu_\Lambda(f)G_{\Lambda}(y,\eta,t)^*(\delta_{t,\frac{1}{m}}^{n+1})(df)$$
$$=\int_{\mathbb{R}^{n^2+n+1}}\left(\prod_{p=1}^{n-1-k}F_{j_{2p-1},j_{2p}}^2\right)\left(G_{\Lambda}(y,\eta,t)_* \mu_\Lambda\right)(F)(\delta_{t,\frac{1}{m}}^{n+1})(F)dF $$
where we recognize the pushforward measure $\tilde{\nu}_{\Lambda}^y(F)=\left(G_{\Lambda}(y,\eta,t)_* \mu_\Lambda\right)(F)$ on $\mathbb{R}^{n^2+n+1}$ of the Gaussian measure $\mu_\Lambda$
described in Proposition~\ref{p:pushforwardGaussmeasure}. This allows us to bound $\nu_{\Lambda,2,m}^{(y,\eta,t)}(\Omega_{\Lambda})$~:
\begin{equation}\label{e:upper-bound-mass-fiber}\nu_{\Lambda,2,m}^{(y,\eta,t)}(\Omega_{\Lambda})\leq\int_{\IR^{n+1}}e^{2i\pi \tau.\eta} e^{-\frac{|\tau_0|^2+t^2|\tau|^2}{m}}\hat{\delta}_{t}^y(\tau_0,\tau)d\tau_0d\tau,\end{equation}
where
\begin{equation}\label{e:fourier-transform}\hat{\delta}_{t}^y(\tau_0,\tau):=\int_{\IR^{n^2+n+1}}\prod_{p=1}^{n-1-k}F_{j_{2p-1},j_{2p}}^2 e^{-2i\pi(\tau_0,t\tau).(F_0,\ldots F_n)} \tilde{\nu}_{\Lambda}^y(F)\end{equation}
and $ \tilde{\nu}_{\Lambda}^y(F)$ is the Gaussian measure of Proposition~\ref{p:pushforwardGaussmeasure}.
This quantity can also be rewritten more explicitely using Wick's lemma as
\begin{equation}\label{e:average-Dirac}\hat{\delta}_{t}^y(\tau_0,\tau)=(-1)^{n-1-k}\prod_{p=1}^{n-1-k}\left(\frac{1}{2\pi}\frac{\partial}{\partial T_{j_{2p-1},j_{2p}}}\right)^2\left(e^{-\frac{2\pi^2\la T,A_{\Lambda}(y)T\ra}{\text{Vol}_g(M)}}\right)_{|T=(\tau_0,t\tau,0)},\end{equation}
where $A_{\Lambda}(y)$ was defined in proposition~\ref{p:pushforwardGaussmeasure}. Equivalently, one has
$$\hat{\delta}_{t}^y(\tau_0,\tau)=R_{y,\Lambda}^{\mathbf{j}}(\tau_0,t\tau)e^{-\frac{2\pi^2}{\text{Vol}_g(M)}\la (\tau_0,t\tau),A_{\Lambda}^{1,1}(y)(\tau_0,t\tau)\ra},$$
for some polynom $R_{y,\Lambda}^{\mathbf{j}}$ depending only on $y$, $\Lambda$ and $\mathbf{j}$ and whose coefficients are uniformly bounded in terms of $\Lambda$ and $y$. Combining this to~\eqref{e:upper-bound-mass-fiber}, one finds that
\begin{equation}\label{e:upper-bound-mass-fiber2}\nu_{\Lambda,2,m}^{(y,\eta,t)}(\Omega_{\Lambda})\leq\frac{1}{t^n\sqrt{\text{det} B_{\Lambda}(y)}}R_{y,\Lambda}^{\mathbf{j}}\left(\frac{1}{2i\pi}\frac{\partial}{\partial\eta_0'},\left(\frac{1}{2i\pi}\frac{\partial}{\partial\eta_l'}\right)_{1\leq l\leq n}\right)\left(e^{-\pi\la\eta',B_{\Lambda,m}(y)\eta'\ra}\right)_{|\eta'=(0,\eta/t)},\end{equation}
where $B_{\Lambda,m}(y)=\frac{\pi}{\text{Vol}_g(M)}A_{\Lambda}^{1,1}(y)+\frac{1}{2\pi m}$. From this expression, we deduce that there exists a constant $C_0>0$, for every $(y,\eta)$ in the support of $\tilde{\omega}$ (which is a compact subset of $\mathring{T}^*U$) and for every $m\geq 1$, one has
$$\forall t>0,\ \nu_{\Lambda,2,m}^{(y,\eta,t)}(\Omega_{\Lambda})\leq \frac{C_0}{t^n}e^{-\frac{1}{C_0 t^2}},$$
which is exactly the upper bound~\eqref{e:mass-measure} from which we get~:
\begin{eqnarray*}
K^{(k)}(\mathbf{j},\omega,U)
&\leq & \int_{\mathring{T}^*U\times\IR^*}\tilde{\omega}(y,\eta)\vert t\vert^{n-3-k}\sqrt{\nu_{\Lambda,1}^{(y,\eta,t)}\left(\Omega_{\Lambda}\right)}\sqrt{\nu_{\Lambda,2}^{(y,\eta,t)}\left(\Omega_{\Lambda}\right)}d^nyd^n\eta dt\\
&\leq &C_K\Vert\omega \Vert \underset{<+\infty}{\underbrace{\int_{\mathbb{R}^*}\vert t\vert^{n-3-k} \frac{1}{t^n}e^{-\frac{1}{C_{\ml{K}} t^2}}dt}},
\end{eqnarray*}
for some constant $C_{\ml{K}}$ depending only on the compact support $\ml{K}\subset \mathring{T}^*U$ of $\omega$. The conclusion of the proposition \ref{l:L1} follows.

\subsubsection{Consequence of the integrability.}\label{r:fubini} 
We fix $k=0$. We note that we have in fact proven something slightly stronger than what we claim in the proposition. Namely, the map
$$(y,\eta,t,f)\mapsto  t^{n-3}\eta_p\eta_{\sigma(p)}\left(\prod_{j\neq p}\frac{\partial^2_{y^j,y^{\sigma(j)}}f}{\Lambda^2}\right)\mu_{\Lambda}(f)\omega_{(0,\ldots,0),(1,\ldots,1)}(y,\eta)$$
belongs to $L^1(\mathring{T}^*U\times\IR^*\times\Omega_{\Lambda},G_{\Lambda}^*(\delta_0^{n+1}))$ and, thanks to Proposition~\ref{intertwiningproposition}, we can integrate this quantity in any order, i.e. either first w.r.t. $f$ (and then $(y,\eta,t)$) or first w.r.t. $(y,\eta, t)$ (and then $f$).

\subsection{Expectation of the leading term}\label{ss:leading} 

In the previous paragraph, we have verified that all the terms in the asymptotic expansion
$$J(f,\omega,U)=\frac{1}{\Lambda^n}\int_{N^{*}(\{f=0\})}\omega(y,dy,\eta,d\eta)=-\sum_{k=0}^{n-1}\Lambda^{-k}J^{(k)}(f,\omega,U)$$
belong to $L^{1}(\ml{H}_{\Lambda},d\mu_{\Lambda})$ and that their $L^1$ norm is uniformly bounded with respect to $\Lambda$. We will now compute precisely the expectation of the leading term $J^{(0)}(f,\omega,U)$ from which the second part of Theorem~\ref{p:expectation-local-chart} will follow.

Recall from remarks~\ref{r:leading-term} and paragraph~\ref{r:fubini} that the
expectation of the leading term decomposes as
\begin{equation}\label{e:leading-expectation}
 \int_{\ml{H}_{\Lambda}}J^{(0)}(f,\omega,U)d\mu_{\Lambda}(f)=\sum_{p=1}^n\sum_{\sigma\in S_n}\vareps(\sigma)K_{p,\sigma}(\omega,U),
\end{equation}
where
$$K_{p,\sigma}(\omega,U)=\int_{\mathring{T}^*U\times\IR^*}t^{n-3}\eta_p\eta_{\sigma(p)}\omega_{(0,\ldots,0),(1,\ldots,1)}(y,\eta)L_{p,\sigma}(y,\eta,t)d^nyd^n\eta dt,$$
and
$$L_{p,\sigma}(y,\eta,t):=\int_{\Omega_{\Lambda}}\left(\prod_{j\neq p}\frac{\partial^2_{y^j,y^{\sigma(j)}}f}{\Lambda^2}\right)\mu_{\Lambda}(f)G_{\Lambda}(y,\eta,t)^*(\delta_0^{n+1})(df).$$
In particular, it remains to compute the value of $K_{p,\sigma}(\omega,U)$ and before that the value of $L_{p,\sigma}(y,\eta,t)$. 

\begin{rema}\label{r:kac-rice} If one wants to compute a Gaussian integral like
$$\int_{\ml{H}_{\Lambda}}\left(\int_{N^*(\{f=0\})}\omega\right)d\mu_{\Lambda}(f),$$
then a natural approach is to use the Kac-Rice formula which follows from the coarea formula -- see e.g. Th.~4.2 
in~\cite{BlShZe01} or~Appendix~$C$ in~\cite{Le14}. This kind of formula allows to ``intertwine'' the order of 
integration in our integral and to integrate \emph{first} with respect to the Gaussian variable. As far as we know, 
there is not (at least explicitely) a Kac-Rice formula for the conormal cycle avalaible in the literature even if 
this kind of approach should a priori allow one (modulo some work) to ``intertwine'' the order of integration. Here, we decided to use 
an alternative approach to treat this issue which is more based on microlocal techniques such as the pull-back theorem for distributions -- see e.g. 
Appendix~\ref{intertwiningsectionappendix}. This point of view of course also leads to an ``inversion'' of the order 
of integration as shown by~\eqref{e:leading-expectation}.
\end{rema}

\subsubsection{Regularization of $K_{p,\sigma}(\omega)$.}
As in the proof of Proposition~\ref{l:L1}, we will regularize the distributions $\delta_0^{n+1}$ in order to proceed to the computation of $K_{p,\sigma}(\omega,U)$. In the notations of paragraph~\ref{sss:gaussian-matrix} we have~:
\begin{prop}\label{Wickgaussprop}
Consider the sequence $\delta_{t,\frac{1}{m}}^{n+1}$ of approximations of $\delta_{0}^{n+1}$. For every $\sigma$ in $S_n$, $1\leq p\leq n$ and for every fixed $(y,\eta,t)$ in $\mathring{T}^*U\times\IR^*$, set
\begin{eqnarray*}
L^{(m)}_{p,\sigma}(y,\eta,t):=\int_{\ml{H}_\Lambda}\mu_\Lambda(f)\prod_{j\neq p}\frac{\partial^2_{y^{j}y^{\sigma(j)}}f}{\Lambda^2} G_{\Lambda}(y,\eta,t)^*(\delta_{t,\frac{1}{m}}^{n+1})(df).
\end{eqnarray*}
Then $L^{(m)}_{p,\sigma}(y,\eta,t)$ equals the integral expression~:
\begin{eqnarray*}
\int_{\IR^{n+1}}e^{2i\pi \tau.\eta} e^{-\frac{|\tau_0|^2+t^2|\tau|^2}{m}}\prod_{j\neq p}\left(\frac{1}{2\pi}\frac{\partial}{\partial T_{j,\sigma(j)}}\right)^2\left(e^{-\frac{2\pi^2\la T,A_{\Lambda}(y)T\ra}{\text{Vol}_g(M)}}\right)_{|T=(\tau_0,t\tau,0)}d\tau_0d\tau
\end{eqnarray*}
\end{prop}
\begin{proof}
By the argument of paragraph~\ref{sss:gaussian-matrix}, we know that
$$\int_{\ml{H}_\Lambda}\mu_\Lambda(f)\prod_{j\neq p}\frac{\partial^2_{y^{j}y^{\sigma(j)}}f}{\Lambda^2} G_{\Lambda}(y,\eta,t)^*(\delta_{t,\frac{1}{m}}^{n+1})(df)=\int_{\mathbb{R}^{n^2+n+1}}\left(\prod_{p=1}^{n-1-k}F_{j,\sigma(j)}^2\right)\tilde{\nu}_{\Lambda}^y(F)(\delta_{t,\frac{1}{m}}^{n+1})(F)dF.$$
Since $\tilde{\nu}_{\Lambda}^y(F)$ is a Gaussian measure with covariance 
$A_\Lambda$, application of the Wick Lemma yields that $\int_{\mathbb{R}^{n^2+n+1}}\left(\prod_{p=1}^{n-1-k}F_{j,\sigma(j)}^2\right)\tilde{\nu}_{\Lambda}^y(F)(\delta_{\frac{1}{m}}^{n+1})(F)dF$ equals 
$$\int_{\IR^{n+1}}e^{2i\pi \tau.\eta} e^{-\frac{|\tau_0|^2+t^2|\tau|^2}{m}}\prod_{j\neq p}\left(\frac{1}{2\pi}\frac{\partial}{\partial T_{j,\sigma(j)}}\right)^2\left(e^{-\frac{2\pi^2\la T,A_{\Lambda}(y)T\ra}{\text{Vol}_g(M)}}\right)_{|T=(\tau_0,t\tau,0)}d\tau_0d\tau$$ which concludes the proof.
\end{proof}
Note that 
$$f\mapsto \left(\prod_{j\neq p}\frac{\partial^2_{y^j,y^{\sigma(j)}}f}{\Lambda^2}\right)\mu_{\Lambda}(f)^{\frac{1}{2}}$$
tends to $0$ as $\|f\|$ tends to infinity; thus, it can be approximated (in the $\ml{C}^0$ topology) by a function $\psi_1(f)$ in $\ml{C}^{\infty}_c(\ml{H}_{\Lambda})$. Moreover, mimicking the final step of the proof of proposition~\ref{l:L1} (namely paragraph~\ref{sss:gaussian-matrix}), one can verify that the sequence 
$$\left(\int_{\ml{H}_{\Lambda}}\mu_{\Lambda}(f)^{\frac{1}{2}}G_{\Lambda}(y,\eta,t)^*(\delta_{t,\frac{1}{m}}^{n+1})(df)\right)_{m\geq 1}$$
is uniformly bounded in terms of $m\geq 1$ (not necessarily in terms of $\Lambda$, $(y,\eta,t)$). Combining these two observations, we find that, for every $1\leq p\leq n$, for every $\sigma$ in $S_n$, and for a.e. $(y,\eta,t)$ in $\mathring{T}^*U\times\IR^*$,
$$\lim_{m\rightarrow+\infty}L_{p,\sigma}^{(m)}(y,\eta,t)=\int_{\ml{H}_{\Lambda}}\mu_{\Lambda}(f)\prod_{j\neq p}\frac{\partial^2_{y^j,y^{\sigma(j)}}f}{\Lambda^2}G_{\Lambda}(y,\eta,t)^*(\delta_0^{n+1})(df)=L_{p,\sigma}(y,\eta,t),$$
where the last equality follows from~\eqref{e:distrib-restriction} from the appendix. Mimicking one more time the proof of paragraph~\ref{sss:gaussian-matrix}, one also knows that there exists a constant $C_0>0$ such that, for every $m\geq 1$, for every $1\leq p\leq n$, for every $\sigma$ in $S_n$, one has
$$\forall (y,\eta,t)\in\text{supp}(\omega_{(0,\ldots,0),(1,\ldots,1)})\times\IR^*,\ |L_{p,\sigma}^{(m)}(y,\eta,t)|\leq \frac{C_0}{t^n}e^{-\frac{1}{C_0 t^2}}.$$
In particular, from the dominated convergence theorem, we can deduce that
\begin{equation}\label{e:leading-term}K_{p,\sigma}(\omega)=\lim_{m\rightarrow+\infty}\int_{\mathring{T}^*U\times\IR^*} t^{n-3}\eta_p\eta_{\sigma(p)}\omega_{(0,\ldots,0),(1,\ldots,1)}(y,\eta)L_{p,\sigma}^{(m)}(y,\eta,t)d^nyd^n\eta dt.\end{equation}

\subsubsection{Computing the limit as $m\rightarrow+\infty$}

Using Gaussian integration, we will now compute explicitely the limit appearing in~\eqref{e:leading-term} as $m$ tends to $+\infty$ and as $\Lambda\rightarrow+\infty$. According to Proposition~\ref{Wickgaussprop}, one has
$$L_{p,\sigma}^{(m)}(y,\eta,t)=(-1)^{\frac{n-1}{2}}\int_{\IR^{n+1}}e^{2i\pi \tau.\eta} e^{-\frac{|\tau_0|^2+t^2|\tau|^2}{m}}\prod_{j\neq p}\left(\frac{1}{2\pi}\frac{\partial}{\partial T_{j,\sigma(j)}}\right)\left(e^{-\frac{2\pi^2\la T,A_{\Lambda}(y)T\ra}{\text{Vol}_g(M)}}\right)_{|T=(\tau_0,t\tau,0)}d\tau_0d\tau.$$
Recall that the expression of $A_{\Lambda}(y)$ was given in Propostion~\ref{p:pushforwardGaussmeasure}. If we plug in this expression into the previous equality, one can also verify that there exists a constant $C_0>0$, such that for every $(y,\eta)$ in the support of $\tilde{\omega}$ and for every $m\geq 1$, one has
\begin{equation}\label{e:constant-matrix}\forall t>0,\ |L_{p,\sigma}^{(m)}(y,\eta,t)-\tilde{L}_{p,\sigma}^{(m)}(\eta,t)|\leq \frac{C_0\Lambda^{-1}}{t^n}e^{-\frac{1}{C_0 t^2}},\end{equation}
where
$$\tilde{L}_{p,\sigma}^{(m)}(\eta,t)=\frac{(-1)^{\frac{n-1}{2}}}{(2\pi)^{n-1}}\int_{\IR^{n+1}}e^{2i\pi \tau.\eta} e^{-\frac{|\tau_0|^2+t^2|\tau|^2}{m}}\prod_{j\neq p}\left(\frac{\partial}{\partial T_{j,\sigma(j)}}\right)\left(e^{-\frac{2\pi^2\la T,A_0T\ra}{\text{Vol}_g(M)}}\right)_{T=(\tau_0,t\tau,0)}d\tau_0d\tau.$$
In particular, provided that the limit $\lim_{m\rightarrow+\infty}\tilde{L}_{p,\sigma}^{(m)}(\eta,t)$ exists a.e., one can combine~\eqref{e:leading-term} to the dominated convergence theorem in order to get 
\begin{equation}\label{e:leading-term-2}
K_{p,\sigma}(\omega,U)=\int_{\mathring{T}^*U\times\IR^*} t^{n-3}\eta_p\eta_{\sigma(p)}\omega_{(0,\ldots,0),(1,\ldots,1)}(y,\eta)\left(\lim_{m\rightarrow+\infty}\tilde{L}_{p,\sigma}^{(m)}(\eta,t)\right)d^nyd^n\eta dt+\ml{O}(\Lambda^{-1}).\end{equation}
\begin{rema}\label{r:simplification} We note that, thanks to the particular form of the matrix $A_0$, one has
$$\tilde{L}_{p,\sigma}^{(m)}(\eta,t)=\frac{(-1)^{\frac{n-1}{2}}}{(2\pi)^{n-1}}\int_{\IR^{n+1}}e^{2i\pi \tau.\eta} e^{-\frac{|\tau_0|^2+t^2|\tau|^2}{m}}e^{-\frac{2\pi^2t^2|\tau|^2}{\text{Vol}_g(M)(n+2)}}H_{p,\sigma}(\tau_0)d\tau_0d\tau,$$
where
$$H_{p,\sigma}(\tau_0):=\prod_{j\neq p}\left(\frac{\partial}{\partial T_{j,\sigma(j)}}\right)\left(e^{-\frac{2\pi^2}{\text{vol}_g(M)}\la (T_0;0;(T_{i,j})),A_0(T_0;0;(T_{i,j}))}\right)_{T=(\tau_0;0;0)}.$$
After integrating over the $\tau$ variable and taking the limit $m\rightarrow+\infty$, one finds
$$\lim_{m\rightarrow+\infty}\tilde{L}_{p,\sigma}^{(m)}(\eta,t)=\frac{(-1)^{\frac{n-1}{2}}}{(2\pi)^{n-1}}\frac{1}{t^n}\left(\frac{(n+2)\text{Vol}_g(M)}{2\pi}\right)^{\frac{n}{2}} e^{-\frac{(n+2)\text{Vol}_g(M)|\eta|^2}{2t^2}}\int_{\IR}H_{p,\sigma}(\tau_0)d\tau_0.$$
\end{rema}
\subsubsection{Summing over $p$ and $\sigma$}

Recall that everything we want to calculate is the leading term $-\int_{\ml{H}_{\Lambda}}J_{\Lambda}^{(0)}(f,\omega,U)d\mu_{\Lambda}(f)$ of $\frac{1}{\Lambda^n}\int_{N^{*}(\{f=0\})_\Lambda}\omega(y,dy,\eta,d\eta)$ which, thanks to~\eqref{e:leading-term-2} and to Remark~\ref{r:simplification}, is given by
\begin{eqnarray*}
&&\int_{\ml{H}_{\Lambda}}J^{(0)}(f,\omega,U)d\mu_{\Lambda}(f)=\sum_{p=1}^n\sum_{\sigma\in S_n}\vareps(\sigma)K_{p,\sigma}(\omega)\\
&=&\sum_{p=1}^n\sum_{\sigma\in S_n}\vareps(\sigma)\int_{\mathring{T}^*U\times\IR^*} t^{n-3}\eta_p\eta_{\sigma(p)}\omega_{(0,\ldots,0),(1,\ldots,1)}(y,\eta)\lim_{m\rightarrow+\infty}\left(\tilde{L}_{p,\sigma}^{(m)}(\eta,t)\right)d^nyd^n\eta dt\\
&=&\frac{(-1)^{\frac{n-1}{2}}}{(2\pi)^{n-1}}\left(\frac{(n+2)\text{Vol}_g(M)}{2\pi}\right)^{\frac{n}{2}}\sum_{p=1}^n\sum_{\sigma\in S_n}\vareps(\sigma)\int_{\mathring{T}^*U\times\IR^*} \eta_p\eta_{\sigma(p)}\omega_{(0,\ldots,0),(1,\ldots,1)}(y,\eta)
\frac{1}{t^3}\\ &&e^{-\frac{(n+2)\text{Vol}_g(M)|\eta|^2}{2t^2}}\int_{\IR}H_{p,\sigma}(\tau_0)d\tau_0
d^nyd^n\eta dt
\end{eqnarray*}
In order to alleviate notations, we introduce:
\begin{equation}\label{e:I_0}L(\eta)=\sum_{p=1}^n\sum_{\sigma\in S_n}\varepsilon(\sigma)\eta_p\eta_{\sigma(p)}\int_{\IR}H_{p,\sigma}(\tau_0)d\tau_0.\end{equation}
Then, we express everything in terms of $L(\eta)$:
\begin{eqnarray*}
&&\int_{\ml{H}_{\Lambda}}J^{(0)}(f,\omega,U)d\mu_{\Lambda}(f)\\&=&\frac{(-1)^{\frac{n-1}{2}}}{(2\pi)^{n-1}}\left(\frac{(n+2)\text{Vol}_g(M)}{2\pi}\right)^{\frac{n}{2}}\sum_{p=1}^n\sum_{\sigma\in S_n}\varepsilon(\sigma)\int_{T^*U\times \IR^*}\omega_{(0,\ldots,0),(1,\ldots,1)}(y,\eta) \eta_p\eta_{\sigma(p)}
\\&&\frac{1}{t^3}e^{-\frac{(n+2)\text{Vol}_g(M)|\eta|^2}{2t^2}}\int_{\IR}H_{p,\sigma}(\tau_0)d\tau_0
d^nyd^n\eta dt\\
&=&\frac{(-1)^{\frac{n-1}{2}}}{(2\pi)^{n-1}}\left(\frac{(n+2)\text{Vol}_g(M)}{2\pi}\right)^{\frac{n}{2}}\int_{T^*U\times\IR^*} \omega_{(0,\ldots,0),(1,\ldots,1)}(y,\eta)\\
&&\frac{1}{t^3}e^{-\frac{(n+2)\text{Vol}_g(M)|\eta|^2}{2t^2}}L(\eta)
d^nyd^n\eta  dt\text{ by definition of }L(\eta)\\
&=&\frac{(-1)^{\frac{n-1}{2}}}{(2\pi)^{n-1}}\left(\frac{(n+2)\text{Vol}_g(M)}{2\pi}\right)^{\frac{n}{2}}
\int_{T^*U} \omega_{(0,\ldots,0),(1,\ldots,1)}(y,\eta)L(\eta)d^nyd^n\eta\\
&&\times  \left( \int_0^{+\infty}t^3e^{-\frac{(n+2)\text{Vol}_g(M)|\eta|^2t^2}{2}}
 \frac{dt}{t^2}+\int_{0}^{-\infty}t^3e^{-\frac{(n+2)\text{Vol}_g(M)|\eta|^2t^2}{2}}
 \frac{dt}{t^2}\right)\\
\end{eqnarray*}
\begin{eqnarray*} 
&=&\frac{(-1)^{\frac{n-1}{2}}}{(2\pi)^{n-1}}\left(\frac{(n+2)\text{Vol}_g(M)}{2\pi}\right)^{\frac{n}{2}}
\int_{T^*U} \omega_{(0,\ldots,0),(1,\ldots,1)}(y,\eta)L(\eta)d^nyd^n\eta\\
&&  \int_0^{+\infty}e^{-\frac{(n+2)\text{Vol}_g(M)|\eta|^2t^2}{2}}
 2tdt\\&=&\frac{2(-1)^{\frac{n-1}{2}}}{(2\pi)^{n-1}}\frac{\left((n+2)\text{Vol}_g(M)\right)^{\frac{n}{2}-1}}{(2\pi)^{\frac{n}{2}}}\int_{T^*U} \omega_{(0,\ldots,0),(1,\ldots,1)}(y,\eta)\frac{L(\eta)}{\vert\eta\vert^2}d^nyd^n\eta. 
\end{eqnarray*}
Therefore,
\begin{equation}\label{e:leading-term2}
 \int_{\ml{H}_{\Lambda}}\left(\frac{1}{\Lambda^n}\int_{N^{*}(\{f=0\})}\omega(y,dy,\eta,d\eta)\right)d\mu_{\Lambda}(f)=A_n\int_{\mathring{T}^*U}\frac{L(\eta)}{|\eta|^2}dy^1\wedge\ldots dy^n\wedge\omega(y,dy,\eta,d\eta)+\ml{O}(\Lambda^{-1}),
\end{equation}
where 
$$A_n:=\frac{2(-1)^{\frac{n+1}{2}}}{(2\pi)^{n-1}}\frac{\left((n+2)\text{Vol}_g(M)\right)^{\frac{n}{2}-1}}{(2\pi)^{\frac{n}{2}}}.$$
In order to conclude, it remains to compute $L(\eta)$. This combinatorial calculation will be the purpose of the next two paragraphs.

\subsection{Computation of $L(\eta)$}\label{ss:combinatorics} Recall that $L(\eta)$ is given by the sum~\eqref{e:I_0}. We start by computing explicitely each term in the sum. We fix $\sigma$ in $S_n$ and $1\leq p\leq n$. Moreover, we set
$$\tilde{T}:=(T_0;0;(T_{i,j})_{1\leq i,j\leq n})\in\mathbb{R}^{n^2+n+1},$$
which appears in the definition of $H_{p,\sigma}(\tau_0)$ -- see Remark~\ref{r:simplification}. Our first step will be to use the Fa\`a di Bruno's formula for partial derivatives~\cite{Har06} to give an explicit expression of 
$$\prod_{j\neq p}\left(\frac{\partial}{\partial T_{j,\sigma(j)}}\right)\left(e^{-\frac{2\pi^2\la \tilde{T},A_0\tilde{T}\ra}{\text{Vol}_g(M)}}\right).$$
Recall that this formula states that
$$\prod_{j\neq p}\left(\frac{\partial}{\partial T_{j,\sigma(j)}}\right)\left(e^{-\frac{2\pi^2\la \tilde{T},A_0\tilde{T}\ra}{\text{Vol}_g(M)}}\right)=e^{-\frac{2\pi^2\la \tilde{T},A_0\tilde{T}\ra}{\text{Vol}_g(M)}}\sum_{\mathfrak{A}\in\ml{P}_{\sigma,p}^{(0)}}\left(-\frac{2\pi^2}{\text{Vol}_g(M)}\right)^{|\mathfrak{A}|}\prod_{\mathfrak{a}\in \mathfrak{A}}\frac{\partial^{|\mathfrak{a}|}\la \tilde{T},A_0\tilde{T}\ra}{\prod_{j\in \mathfrak{a}}\partial T_{j, \sigma(j)}},$$
where $\ml{P}_{\sigma,p}^{(0)}$ is the set of partitions of $\{1\leq j\leq n: j\neq p\}$. We now fix some notations:
$$\ml{P}_{\sigma,p}:=\left\{\mathfrak{A}\ \text{partition of}\ \{1\leq j\leq n: j\neq p\}:\ \forall \mathfrak{a}\in\mathfrak{A},\ |\mathfrak{a}|\leq 2\ \text{and}\ \sigma(\mathfrak{a})=\mathfrak{a} \right\}.$$
\begin{rema}\label{r:zero-term}
We note that $\ml{P}_{\sigma,p}$ is empty if $\sigma(p)\neq p$. Moreover, for every $\sigma$ such that $\sigma(p)=p$, the above subset is empty if $\sigma$ is not a product of disjoint $2$-cycles. 
\end{rema}
For a given $\mathfrak{A}$ in $\ml{P}_{\sigma,p}$, we define
\begin{itemize}
 \item $\mathfrak{A}_1:=\{\mathfrak{a}\in\mathfrak{A}:|\mathfrak{a}|=1\}$;
 \item $\mathfrak{A}_2:=\{\mathfrak{a}\in\mathfrak{A}:|\mathfrak{a}|=2\ \text{and}\ \sigma_{|\mathfrak{a}}=\text{id}\}$;
 \item $\mathfrak{A}_3:=\{\mathfrak{a}\in\mathfrak{A}:|\mathfrak{a}|=2\ \text{and}\ \sigma_{|\mathfrak{a}}\neq\text{id}\}$.
\end{itemize}
Using these conventions and the specific form of the matrix $A_0$ -- see paragraph~\ref{sss:gaussian-matrix}, one finds that
$$\prod_{j\neq p}\left(\frac{\partial}{\partial T_{j,\sigma(j)}}\right)\left(e^{-\frac{2\pi^2\la T,A_0T\ra}{\text{Vol}_g(M)}}\right)_{|T=(\tau_0;0;0)}$$
$$=e^{-\frac{2\pi^2\tau_0^2}{\text{Vol}_g(M)}}\sum_{\mathfrak{A}\in\ml{P}_{\sigma,p}}\left(-\frac{4\pi^2}{\text{Vol}_g(M)(n+2)}\right)^{|\mathfrak{A}|}\left(\frac{1}{n+4}\right)^{|\mathfrak{A}_2|+|\mathfrak{A}_3|}(-\tau_0)^{|\mathfrak{A}_1|}.$$
Observe now that
$$\int_{\IR}\tau_0^{|\mathfrak{A}_1|}e^{-\frac{2\pi^2\tau_0^2}{\text{Vol}_g(M)}}d\tau_0=\left(\frac{\text{Vol}_g(M)}{2\pi}\right)^{\frac{1}{2}}(|\mathfrak{A}_1|-1)!!\left(\frac{\text{Vol}_g(M)}{4\pi^2}\right)^{\frac{|\mathfrak{A}_1|}{2}}$$
if $|\mathfrak{A}_1|\equiv 0\ (\text{mod}\ 2)$, and that it is equal to $0$ otherwise. We now set
$$C_{\sigma,p}:=\sum_{\mathfrak{A}\in\ml{P}_{\sigma,p}:|\mathfrak{A}_1|\equiv 0(\text{mod}\ 2)}\left(-\frac{1}{n+2}\right)^{|\mathfrak{A}|}(|\mathfrak{A}_1|-1)!!\left(\frac{4\pi^2}{\text{Vol}_g(M)}\right)^{|\mathfrak{A}|-\frac{|\mathfrak{A}_1|}{2}}\left(\frac{1}{n+4}\right)^{|\mathfrak{A}_2|+|\mathfrak{A}_3|},$$
whenever $\sigma(p)=p$ and $C_{\sigma,p}:=0$ otherwise. 
\begin{rema}\label{r:even}
We note that the sum is empty as soon as $n$ is even. In fact, requiring $|\mathfrak{A}_1|\equiv 0(\text{mod}\ 2)$ imposes that $n-1$ is even.
\end{rema}
From the above calculation, we deduce
\begin{equation}\label{e:I0}
 L(\eta)=\sum_{p=1}^n\sum_{\sigma\in S_n}\varepsilon(\sigma)\eta_p\eta_{\sigma(p)}\int_{\IR} H_{p,\sigma}(\tau_0)d\tau_0=\left(\frac{\text{Vol}_g(M)}{2\pi}\right)^{\frac{1}{2}}\sum_{p=1}^n\eta_p^2\sum_{\sigma\in S_n}\varepsilon(\sigma)C_{\sigma,p}.
\end{equation}
In particular, from remark~\ref{r:even}, one has $L(\eta)=0$ if $n$ is even. We also note that $\sum_{\sigma\in S_n}C_{\sigma,p}$ is independent of the choice of $p$. Thus, one finds, whenever $n$ is odd,
\begin{equation}\label{e:I0-bis}
 L(\eta)=\left(\frac{\text{Vol}_g(M)}{2\pi}\right)^{\frac{1}{2}}|\eta|^2\sum_{\sigma\in S_n}\varepsilon(\sigma)C_{\sigma,p}.
\end{equation}
From~\eqref{e:leading-term2}, one finally obtains
\begin{equation}\label{e:leading-term3}
 \int_{\ml{H}_{\Lambda}}\left(\frac{1}{\Lambda^n}\int_{N^{*}(\{f=0\})}\omega(y,\eta,dy,d\eta)\right)d\mu_{\Lambda}(f)
\end{equation}
$$ =\left(\frac{\text{Vol}_g(M)}{2\pi}\right)^{\frac{1}{2}}A_n\left(\sum_{\sigma\in S_n}\varepsilon(\sigma)C_{\sigma,p}\right)\int_{\mathring{T}^*U}dy^1\wedge\ldots dy^n\wedge\omega(y,dy,\eta,d\eta)+\ml{O}(\Lambda^{-1}),$$
where $A_n$ was defined in~\eqref{e:leading-term2}.

\subsection{Computation of $\sum_{\sigma\in S_n}C_{\sigma,p}$}\label{ss:combinatorics2}
From~\eqref{e:leading-term3}, it remains to compute $\sum_{\sigma\in S_n}C_{\sigma,p}$ which is independent of the choice of $p$. For that purpose, we fix $0\leq k\leq\frac{n-1}{2}$ and we suppose that $\sigma$ is the product of $k$ disjoint $2$-cycles $\sigma_1,\ldots,\sigma_k$. Then, we compute the value of $C_{\sigma,p}$ in this case. We note that in this case $|\mathfrak{A}_3|$ is necessarily equal to $k$, and thus
$$C_{\sigma,p}=\left(\frac{1}{n+4}\right)^k\sum_{\mathfrak{A}\in\ml{P}_{\sigma,p}:|\mathfrak{A}_1|\equiv 0(\text{mod}\ 2)}\left(-\frac{1}{n+2}\right)^{|\mathfrak{A}|}(|\mathfrak{A}_1|-1)!!\left(\frac{4\pi^2}{\text{Vol}_g(M)}\right)^{|\mathfrak{A}|-\frac{|\mathfrak{A}_1|}{2}}\left(\frac{1}{n+4}\right)^{|\mathfrak{A}_2|}.$$
We also note the following useful relations
$$n-1=|\ml{A}_1|+2|\ml{A}_2|+2|\ml{A}_3|\ \text{and}\ \ml{A}=|\ml{A}_1|+|\ml{A}_2|+|\ml{A}_3|.$$
Recall now that the number of permutations in $\{1,2,\ldots,2l\}$ with $l$ cycles of length $2$ is equal to $\frac{(2l)!}{l! 2^l}$ -- see Th.~$6.9$ in~\cite{Bo06}. From this observation, we deduce that
$$C_{\sigma,p}=\left(\frac{4\pi^2}{(n+2)^2\text{Vol}_g(M)}\right)^{\frac{n-1}{2}}\left(-\frac{n+2}{n+4}\right)^k\sum_{l=0}^{\frac{n-1}{2}-k}\frac{(2l)!}{l! 2^l}C_{n-1-2k}^{2l}(n-2(k+l+1))!!\left(-\frac{n+2}{n+4}\right)^{l},$$
and thus
$$C_{\sigma,p}=\left(\frac{2\pi^2}{2(n+2)^2\text{Vol}_g(M)}\right)^{\frac{n-1}{2}}\left(-\frac{2(n+2)}{n+4}\right)^k\sum_{l=0}^{\frac{n-1}{2}-k}\frac{(n-1-2k)!}{l!\left(\frac{n-1}{2}-(k+l)\right)!}\left(-\frac{n+2}{n+4}\right)^{l}.$$
Thanks to our formula for $C_{\sigma,p}$ when $\sigma$ is a product of $k$ disjoint $2$-cycles (thus $\varepsilon(\sigma)=(-1)^k$), we deduce that
$$\sum_{\sigma\in S_n}\varepsilon(\sigma)C_{\sigma,p}=\left(\frac{2\pi^2}{(n+2)^2\text{Vol}_g(M)}\right)^{\frac{n-1}{2}}\sum_{k=0}^{\frac{n-1}{2}}C_{n-1}^{2k}\frac{(2k)!}{k! 2^k}2^k\sum_{l=0}^{\frac{n-1}{2}-k}(-1)^l\frac{(n-1-2k)!}{l!2^l\left(\frac{n-1}{2}-(k+l)\right)!}\left(\frac{n+2}{n+4}\right)^{l+k}.$$
After simplification, it gives
$$\sum_{\sigma\in S_n}\varepsilon(\sigma)C_{\sigma,p}=(n-1)!\left(\frac{2\pi^2}{(n+2)^2\text{Vol}_g(M)}\right)^{\frac{n-1}{2}}\sum_{k=0}^{\frac{n-1}{2}}\left(\frac{n+2}{n+4}\right)^k\sum_{l=0}^{\frac{n-1}{2}-k}\frac{\left(-\frac{n+2}{n+4}\right)^{l}}{k!l!\left(\frac{n-1}{2}-(k+l)\right)!},$$
and thus
$$\sum_{\sigma\in S_n}\varepsilon(\sigma)C_{\sigma,p}=(n-1)!\left(\frac{2\pi^2}{(n+2)^2\text{Vol}_g(M)}\right)^{\frac{n-1}{2}}\sum_{k=0}^{\frac{n-1}{2}}\frac{1}{k!\left(\frac{n-1}{2}-k\right)!}\left(\frac{n+2}{(n+4)}\right)^k\left(1-\frac{n+2}{n+4}\right)^{\frac{n-1}{2}-k}.$$
Finally, we obtain
\begin{equation}\label{e:combi-cstt}\sum_{\sigma\in S_n}\varepsilon(\sigma)C_{\sigma,p}=\frac{(n-1)!}{\left(\frac{n-1}{2}\right)!}\left(\frac{2\pi^2}{(n+2)^2\text{Vol}_g(M)}\right)^{\frac{n-1}{2}}.\end{equation}

\subsection{The conclusion}\label{ss:concl} Regarding~\eqref{e:leading-term2},~\eqref{e:leading-term3} and~\eqref{e:combi-cstt}, we define
$$B_n:=\frac{2(-1)^{\frac{n+1}{2}}}{(2\pi)^{n-1}}\frac{(n-1)!}{\left(\frac{n-1}{2}\right)!}\frac{\pi^{\frac{n-3}{2}}}{2(n+2)^{\frac{n}{2}}}=\frac{(-1)^{\frac{n+1}{2}}}{\pi\text{Vol}(\mathbb{S}^{n-1})}\frac{2}{(n+2)^{\frac{n}{2}}}.$$
if $n\equiv 1\ (\text{mod}\ 2)$ and $B_n=0$ otherwise. Finally, still thanks to~\eqref{e:leading-term3}, we conclude
$$\int_{\ml{H}_{\Lambda}}\left(\frac{1}{\Lambda^n}\int_{N^{*}(\{f=0\})}\omega(y,dy,\eta,d\eta)\right)d\mu_{\Lambda}(f)=B_n\int_{\mathring{T}^*U}\omega_{(0,\ldots,0),(1,\ldots,1)}(y,\eta)d^nyd^n\eta +\ml{O}(\Lambda^{-1}),$$
which is exactly the content of the second part of Theorem~\ref{p:expectation-local-chart}.

\section{Formal derivation of the main result via Berezin integration}
\label{s:berezin}

The derivation of the main theorem (namely Theo.~\ref{p:expectation-local-chart}) is slightly technical due to the fact that we have to justify carefully the smallness of several terms and the inversion of several integrals but also due to some complicated combinatorics. The purpose of this last section is to show how the combinatorial aspects of the proof can be treated more easily at the expense of a little bit of abstraction by making use of the so-called Berezin integral~\cite{Bz87, GuSt99, Dis11, Tao13}. 

In this section, we will not pay too much attention to the inversion of integrals and to the size of the remainder terms, and we will mostly focus on the computation of the leading term of the asymptotics. Similar arguments as the ones presented in the previous section would in fact provide a rigorous treatment of the calculation developped in the present section. The reason for presenting this alternative approach in a formal (but non completely rigorous manner) is that we believe that it can be helpful (at least for the reader familiar with Berezin integration) to understand and follow the main lines of our proof, especially its combinatorial aspects. It is also plausible that this kind of intuitive calculation could be used for other related questions.

\subsection{A brief reminder on Berezin integration} Berezin integration is a convenient formalism that allows us to represent the wedge product appearing 
in our formulas as oscillatory integrals. This formalism was introduced by Berezin~\cite{Bz87} and it is a more or less standard tool in quantum field 
theory where one aims at computing integrals over both bosonic (even) and fermionic (odd) variables. Basic introductions to this formalism are described by 
Disertori~\cite{Dis11} or by Tao~\cite{Tao13} with a view towards random matrix theory, a very good introduction can be found in the fantastic book of Takhtajan~\cite[Chapter 7]{Takhtajan} -- see also the book of Guillemin and Sternberg~\cite{GuSt99} 
for considerations on Fourier transforms in this context. Following closely the presentation of these references, we give here a brief (and somewhat simplified) overview on 
it that should be sufficient for the purpose of our formal calculation.

\subsubsection{Odd and even variables} Let $V$ be a finite dimensional real vector space. One can define the so called exterior 
algebra $\Lambda V:=\bigoplus_{k=0}^{+\infty}\Lambda^kV.$ For simplicity of notations, we will denote by $\Phi_1\Phi_2:=\Phi_1\wedge\Phi_2$ 
the product on this algebra. This algebra can be splitted in two parts. One part is made of the so-called even (or bosonic) elements 
$\Lambda V_{\text{even}}:=\bigoplus_{k=0}^{+\infty}\Lambda^{2k}V,$ while the other is made of the so-called odd (or fermionic) elements 
$\Lambda V_{\text{odd}}:=\bigoplus_{k=0}^{+\infty}\Lambda^{2k+1}V.$ Any odd element $\Pi$ commutes with every even element $p$, while it 
anticommutes with odd elements. From this observation, one can deduce that any element $\Phi$ in $\oplus_{k=1}^{+\infty}\Lambda^kV$ is nilpotent. One can verify that any element $\Phi$ of $\Lambda V$ can be exponentiated as follows:
$$\exp(\Phi):=\sum_{k=0}^{+\infty}\frac{\Phi^k}{k!}.$$
Important properties of the exponential are that $\exp(\Pi)=1+\Pi$ for any odd element $\Pi$ and that, for any even element $p$,
\begin{equation}\label{e:exp-product-berezin}\exp(p+\Phi)=\exp(p)\exp(\Phi)=\exp(\Phi)\exp(p).\end{equation}

\subsubsection{Integration of functions of odd and even variables} 

One of the aim of Berezin integration is to integrate function whose variables are in $\Lambda V$. In the following, we just need a simple version of 
this formalism. Namely, we fix two nonnegative integers $d_1$ and $d_2$, and we consider the case where $V=\IR$. We are only aiming at the integration 
of functions depending on the even variables $(p_1,\ldots,p_{d_1})\in(\Lambda^0V)^{d_1}\simeq\IR^{d_1}$ and on the odd variables 
$(\Pi_1,\ldots, \Pi_{d_2})\in(\Lambda^{1}V)^{d_2}$.
Said differently, we consider the algebra $C^\infty(\mathbb{R}^{d_1})$ whose elements
are smooth functions $f(p_1,\dots,p_{d_1})$ of the even variables $(p_1,\ldots,p_{d_1})$ then we consider the Grassmann
algebra $C^\infty(\mathbb{R}^{d_1})[\Pi_1,\dots,\Pi_{d_2}]$ which is a polynomial algebra generated
by $(\Pi_1,\ldots, \Pi_{d_2})$ satisfying the relations $\Pi_i\Pi_j+\Pi_j\Pi_i=0,\forall (i,j)\in \{1,\dots,d_2\}^2$~see \cite[Chapter 7]{Takhtajan}.
More precisely, our functions will be of the form
$$f(p,\Pi):= f_0(p)+\sum_{k\geq 1}\sum_{i_1,\ldots, i_k}f_{i_1,\ldots,i_k}(p)\Pi_{i_1}\ldots\Pi_{i_k},$$
where $f_*$ are functions defined on $\IR^{d_1}$ and with values in $\Lambda V$. Given such a function $f$, we will say that it is defined on $\IR^{(d_1|d_2)}$ and its Berezin integral is defined by the fundamental formula~:
\begin{equation}
\int_{\IR^{(d_1|d_2)}}f(p,\Pi)dpd\Pi:=\int_{\IR^{d_1}}\partial_{\Pi_1}\ldots\partial_{\Pi_{d_2}}fdp_1dp_2\ldots dp_{d_1}.
\end{equation}
We refer the reader to~\cite{Dis11} for a more general presentation of the Berezin integral, for basic properties on this integration procedure 
and for a review of simple examples. 
\begin{rema} Following~\cite{GuSt99} -- Chapter~$7$, one can also define a notion of ``super Fourier transform'', as follows:
$$\forall q\in(\Lambda^0\IR)^{d_1},\ \forall Q\in(\Lambda^1\IR)^{d_2},\ \ml{F}(f)(q,Q):=\int_{\IR^{(d_1|d_2)}}\exp\left(-2i\pi (q.p+Q.\Pi)\right)f(p,\Pi)dpd\Pi.$$
The prefix ``super'' is here to emphasize that we consider functions depending both on even (bosonic) and odd (fermionic) variables. We observe that $q.p+Q.\Pi$ is an even element. 
\end{rema}
Note that any one form $Q$ can be written as follows:
\begin{equation}\label{e:odd-Fourier}Q=-\frac{1}{2i\pi}\int_{\IR^{(0|1)}}\exp(-2i\pi \Pi Q)d\Pi,\end{equation}
and more generally any product of $d_2$ one forms $Q_1,\ldots, Q_{d_2}$ as
\begin{equation}\label{e:odd-Fourier2}Q_1\ldots Q_{d_2}=\frac{1}{(-2i\pi)^{d_2}}\int_{\IR^{(0|d_2)}}\exp\left(-2i\pi \sum_{j=1}^{d_2}\Pi_jQ_j\right)d\Pi_1\ldots d\Pi_{d_2}.\end{equation}
This formula gives us a representation of products of odd elements as ``oscillatory'' integrals. This observation is the first key point which is at the heart of the calculation below.

\subsubsection{Representation of integration currents as oscillatory integrals over even and odd variables.}
Using the conventions of Appendix~\ref{a:currents}, we consider a submanifold $S$ in some larger manifold $X$
given by equations $G=(g_i)_{1\leqslant i\leqslant d}=0$. Then
it is a well known fact that the delta function $\delta^d_0\circ G=G^*(\delta_0^d)$ supported
on $S$ can be represented as an oscillatory integral.
So a natural question is how to represent the
integration current $[S]=[\{G=0\}]$ as an oscillatory integral ? 
The answer involves the introduction of odd variables and the use of the Berezin integral. 
\begin{prop}\label{Berezinoscillatoryprop}
Let $X$ be a smooth oriented manifold with orientation $[\omega_1]$, 
\begin{eqnarray*}
G: x\in X\rightarrow y=(g_1(x),\ldots,g_d(x))\in\IR^d
\end{eqnarray*}
a smooth function such that the differentials $d_xg_i$ are linearly independent 
for every $x\in X$ satisfying $g_1(x)=\dots=g_d(x)=0$. Set $S$ to be the submanifold in $X$ defined by the regular system of equations
$\{g_1=\dots=g_d=0\}$ and oriented by $[\omega_2]$. If $dg_1\wedge \dots\wedge dg_d$ has
orientation compatible  
with the pair of orientations $[\omega_1],[\omega_2]$
then the integration current $[S]$ is represented by the oscillatory integral
\begin{equation}
[S]=G^*\left(\frac{1}{(-2i\pi)^d}\int_{\mathbb{R}^{(d\vert d)}} e^{-2i\pi\left(\sum_{j=1}^dp_jy^j+\Pi_jdy^j\right)} dp_1\dots dp_d
d\Pi_1\dots d\Pi_d\right).
\end{equation}
\end{prop}
\begin{proof}
The proof follows from a simple calculation, we start from the representation formula for the current of integration given in appendix~\ref{a:currents}~:
$$[S]=G^*\left(\delta_0^d dy^1\wedge\ldots \wedge dy^d\right).$$
Then, using the classical representation of the delta function by oscillatory
integral, one has
$$G^*\delta_0^d G^*\left(dy^1\wedge\ldots \wedge dy^d\right)=G^*\left(\int_{\mathbb{R}^{d}} e^{-2i\pi\sum_{j=1}^dp_jy^j}dp_1\ldots dp_j\right) G^*\left( dy^1\wedge\ldots \wedge dy^d\right)$$
and, using the representation~\eqref{e:odd-Fourier2} of the exterior product $dy^1\wedge \dots\wedge dy^d$ in terms of the Berezin integral
we end up with the expected formula.
\end{proof}

In the case where the ambient manifold $X$ is $\IR^D$ (e.g. in a local chart), we can in fact write (at least formally)
$$[S]=\frac{1}{(-2i\pi)^d}\int_{\mathbb{R}^{(d\vert d)}} e^{-2i\pi\left(\sum_{j=1}^dp_jg_j+\Pi_jdg_j\right)} dp_1\dots dp_d
d\Pi_1\dots d\Pi_d$$
This oscillatory integral formula is already known in the literature since it appears
in the work of Frenkel, Losev and Nekrasov on instantonic quantum field theory~\cite[p.~23]{LosevI}.
\subsubsection{Fourier transform of Gaussian integrals} We observe that the following holds:
\begin{lemm}\label{l:formalseriesgaussian}
Let $N\geq 1$ and $A$ be a symmetric, positive definite $N\times N$ matrix. Let $p=(p_1,\ldots, p_N)$ be an element in $\left(\Lambda V_{\text{even}}\right)^N$. Then, one has
\begin{equation}\label{fouriergaussianformal}
e^{-\frac{\langle p\vert A^{-1}\vert p\rangle}{2}}=\frac{\det(A)^{\frac{1}{2}}}{\left(2\pi\right)^{\frac{N}{2}}}
\int_{\mathbb{R}^N} e^{i\langle c,p\rangle} e^{-\frac{\langle c\vert A\vert c\rangle}{2}}d^Nc,
\end{equation}
where $\la c,p\ra=\sum_{j=1}^N c_jp_j$, and $\langle p\vert A^{-1}\vert p\rangle=\sum_{i,j}p_ip_j(A^{-1})_{ij}$.
\end{lemm}
This lemma will be the second key point in the computation of the asymptotics via Berezin integration.
\begin{proof} Let $\mathbb{C}[[X_1,\ldots, X_N]]$ be the commutative algebra of formal power series in $N$ indeterminates $(X_1,\dots,X_N)$. First, we note that
\begin{equation}\label{e:fouriergaussianformal}
e^{-\frac{\langle X\vert A^{-1}\vert X\rangle}{2}}\quad \text{and}\quad \frac{\det(A)^{\frac{1}{2}}}{\left(2\pi\right)^{\frac{N}{2}}}
\int_{\mathbb{R}^N} e^{i\langle c,X\rangle} e^{-\frac{\langle c\vert A\vert c\rangle}{2}}d^Nc\end{equation}
are well defined as formal power series since the integral on the right hand side is defined by expanding the exponential
function $e^{i\langle c,X\rangle}$ in powers of $X$ under the integral sign where, for each $(i_1,\ldots,i_k)$, the term $\int_{\mathbb{R}^N} (c_{i_1}\ldots c_{i_k}) e^{-\frac{\langle c\vert A\vert c\rangle}{2}}d^Nc $
is a convergent integral that can be explicitely calculated by the Wick lemma.
Then the equality between the two formal series in~\eqref{e:fouriergaussianformal} follows from the classical
result on the Fourier transform of the Gaussian measure replacing the formal indeterminate $X$ by $y\in\mathbb{R}^N$ and by uniqueness of the Taylor series expansion in $y$ at $y=0$.
\end{proof}

\begin{rema}
Conceptually, we note that this proof shows that 
we can define the Fourier transform of a Gaussian measure
in a purely algebraic language. 
\end{rema}

\subsection{Asymptotic expansion in a local chart}

We fix a local geodesic normal coordinate chart $(U,\phi)$ and $\omega(y,dy,\eta,d\eta)$ an element in $\mathcal{D}_n(\mathring{T}^*U)$. Recall that we want to compute the expectation of 
$$J(f,\omega,U):=\int_{N^{*}(\{f=0\})}\omega(y,dy,\eta,d\eta)$$
with respect to the Gaussian measure $d\mu_{\Lambda}$. We can write formally
$$J(f,\omega,U)=\int_{\mathring{T}^*U\times\mathbb{R}^*}\delta_0^{n+1}\left(f,\frac{td_yf}{\Lambda}-\eta\right)df\wedge\bigwedge_{j=1}^nd\left(\frac{t\partial_{y^j} f}{\Lambda}-\eta_j\right)\wedge \omega(y,dy,\eta,d\eta).$$
Recall that one of the reason why the calculation of the expectation against the Gaussian measure is possible is the fact that the Dirac distribution can be represented by an oscillatory integral over even variables, i.e.
$$\delta_0^{n+1}\left(f,\frac{td_yf}{\Lambda}-\eta\right)=\int_{\mathbb{R}^{(n+1|0)}}\exp\left(-2i\pi\left(p_0f+p.\left(\frac{td_yf}{\Lambda}-\eta\right)\right)\right)dp_0dp.$$
Then, some of the (combinatorial) difficulties in the proof come from the fact that the wedge product has not at first sight the same simple structure. 
By Proposition \ref{Berezinoscillatoryprop}, we see that the advantage of the Berezin formalism described above is that it allows us to write the wedge 
product as an oscillatory integral over odd variables in $\mathbb{R}^{(0|n)}$. Before writing this oscillatory 
representation -- see equation~\eqref{e:oscillatory-Berezin} below, we proceed to a few simplifications that will make the calculation slightly simpler. 
More precisely, we expand the wedge product, and we obtain that
$$J(f,\omega,U)=\Lambda\sum_{r,s=1}^n(-1)^r\int_{\mathring{T}^*U\times\IR^*}\frac{\eta_r\eta_s}{t^2} dy^s\wedge W^{r}(f,t,y,dy,\eta, d\eta)\wedge \omega(y,dy,\eta,d\eta)\wedge dt,$$
where
$$W^{r}(f,t,y,dy,\eta, d\eta):=\delta_0^{n+1}\left(f,\frac{td_yf}{\Lambda}-\eta\right)\bigwedge_{j\neq r}d\left(\frac{t\partial_{y^j} f}{\Lambda}-\eta_j\right).$$
For every $t$ in $\IR^*$ and every $f$ in $\ml{H}_{\Lambda}$, this defines an element in $\ml{D}_{n+1}'(\mathring{T}^*U)$. The main contribution in this integral will come from the terms involving derivatives of $f$ -- this is exactly the content of Prop.~\ref{l:L1} above. Thus, one can set
$$W^r_0(f,t,y,dy,\eta):=\delta_0^{n+1}\left(f,\frac{td_yf}{\Lambda}-\eta\right)\bigwedge_{j\neq r}d\left(\frac{\partial_{y^j} f}{\Lambda^2}\right),$$
and then one finds
\begin{equation}\label{e:berezin-step-1}J(f,\omega,U)=\Lambda^n\sum_{r,s=1}^n(-1)^r\int_{\mathring{T}^*U\times\IR^*}t^{n-3}\eta_r\eta_s dy^s\wedge W^r_0(f,t,y,dy,\eta)\wedge \omega(y,dy,\eta,d\eta)\wedge dt+\ml{O}(\Lambda^{n-1}).\end{equation}
In order to get the conclusion, everything boils down to computing the expectation of $W^r_0(f,t,y,dy,\eta)$ with respect to $d\mu_{\Lambda}$ for a given $(y,\eta,t)$ in $\mathring{T}^*U\times\IR^*$. In order to alleviate notations, we will just write $W^r_0(f)$ and re-establish the dependence in $(y,\eta,t)$ after. Using the representation~\eqref{e:odd-Fourier2} of $1$-forms via Berezin integrals, one can write the following ``oscillatory'' integral:
\begin{equation}\label{e:oscillatory-Berezin}W^r_0(f)=\frac{(-1)^{\frac{n-1}{2}}}{(2\pi)^{n-1}}\int_{\mathbb{R}^{(n+1|n-1)}}\exp\left(-2i\pi\left(p_0f+p.\left(\frac{td_yf}{\Lambda}-\eta\right)+\sum_{j\neq r, k}\frac{\partial^2_{y^jy^k}f}{\Lambda^2}\Pi_jdy^k\right)\right)dp_0dp d\Pi.\end{equation}

\begin{rema}
The fact that $W_0^r(f)$ has a simple representation as an oscillatory integral is the \emph{central observation} that makes our formal calculation much simpler than the detailed combinatorial proof we gave before.
Note that the element $\Pi_jdy^k$ is even since it is the product of two odd elements and thus, we are in a situation where we can apply lemma~\ref{l:formalseriesgaussian}. As was already mentionned, we do not pay too much attention to the inversion of integrals and to the size of the remainders, and we mostly focus on the conceptual aspects that make our proof works.
\end{rema}

Recall that we aim at computing the expectation of $W_0^r(f)$ with respect to the Gaussian measure $d\mu_{\Lambda}$. For that purpose, we fix $(p_0,p,\Pi)$ in $\mathbb{R}^{(n+1|n-1)}$, and we use lemma~\ref{l:formalseriesgaussian} to write
$$\int_{\ml{H}_{\Lambda}}\exp\left(-2i\pi\left(p_0f+p.\left(\frac{td_yf}{\Lambda}\right)+\sum_{j\neq r, k}\frac{\partial^2_{y^jy^k}f}{\Lambda^2}\Pi_jdy^k\right)\right)d\mu_{\Lambda}(f)$$
\begin{equation}\label{e:gaussian-Berezin}=\exp\left(-\frac{2\pi^2}{\text{Vol}_g(M)}\la (p_0,tp,(\Pi_jdy^k)_{j\neq r,k}),A_0(p_0,tp,(\Pi_jdy^k)_{j\neq r,k})\ra\right)(1+\ml{O}(\Lambda^{-1})),\end{equation}
where the covariance matrix $A_0$ was defined in proposition~\ref{p:pushforwardGaussmeasure}. We can now expand 
$$\la (p_0,tp,(\Pi_jdy^k)_{j\neq r,k}),A_0(p_0,tp,(\Pi_jdy^k)_{j\neq r,k})\ra =  \frac{1}{n+2}t^2|p|^2+p_0^2-\frac{2}{n+2}p_0\sum_{j\neq r}\Pi_jdy^j$$
  $$ +  \frac{2}{(n+2)(n+4)}\sum_{j,k\neq r}(\Pi_jdy^k\Pi_kdy^j+\Pi_jdy^j\Pi_kdy^k).$$
Thanks to the anti-commutation rules for odd variables, one finds that the last term in this sum is in fact equal to $0$. Recall that we are in fact interested in computing $W_0^r(f)$ which is defined by~\eqref{e:oscillatory-Berezin}. In particular, we have to multiply this Gaussian term by $e^{2i\pi p.\eta}$ and then integrate over the variables $(p_0,p,\Pi)$ in $\IR^{(n+1|n-1)}$. After integrating over the even variables $(p_0,p)$, one finds that the following equalities hold:
\begin{equation}\label{e:berezin-tau}\int_{\IR^{(n|0)}}e^{2i\pi p.\eta}\exp\left(-\frac{2\pi^2t^2|p|^2}{(n+2)\text{Vol}_g(M)}\right)dp=\left(\frac{\text{Vol}_g(M)(n+2)}{2\pi t^2}\right)^{\frac{n}{2}}e^{-\frac{(n+2)\text{Vol}_g(M)|\eta|^2}{2t^2}},\end{equation}
and, by lemma~\ref{l:formalseriesgaussian},
\begin{equation}\label{e:berezin-tau0}\int_{\IR^{(1|0)}}\exp\left(-\frac{2\pi^2}{\text{Vol}_g(M)}\left(p_0^2-\frac{2}{n+2}p_0\sum_{j\neq r}\Pi_jdy^j\right)\right)dp_0\hspace{3cm}\end{equation}
$$\hspace{3cm}=\left(\frac{\text{Vol}_g(M)}{2\pi}\right)^{\frac{1}{2}}\exp\left(\frac{2\pi^2}{\text{Vol}_g(M)(n+2)^2}(\sum_{j\neq r}\Pi_jdy^j)^2\right).$$

It now remains to integrate~\eqref{e:berezin-tau0} with respect to the odd variables $(\Pi_j)_{j\neq r}$. Regarding the definition of the Berezin integral, we have to compute that
$$\frac{\partial}{\partial\Pi_1}\ldots\frac{\hat{\partial}}{\partial\Pi_r}\ldots\frac{\partial}{\partial\Pi_n}\exp\left(\frac{2\pi^2}{\text{Vol}_g(M)(n+2)^2}(\sum_{j\neq r}\Pi_jdy^j)^2\right).$$
One can verify that \emph{if $n$ is even, this quantity vanishes, and then the total expectation vanishes}. In the case where $n$ is odd, one finds from the multinomial Newton formula that
$$\frac{\partial}{\partial\Pi_1}\ldots\frac{\hat{\partial}}{\partial\Pi_r}\ldots\frac{\partial}{\partial\Pi_n}\exp\left(\frac{2\pi^2}{\text{Vol}_g(M)(n+2)^2}(\sum_{j\neq r}\Pi_jdy^j)^2\right)\hspace{3cm}$$
$$\hspace{3cm}=\frac{(n-1)!}{\left(\frac{n-1}{2}\right)!}\left(\frac{2\pi^2}{\text{Vol}_g(M)(n+2)^2}\right)^{\frac{n-1}{2}}dy^1\wedge\ldots \hat{dy^r}\wedge\ldots dy^n.$$
Combining this to~\eqref{e:berezin-tau0}, we obtain
\begin{equation}\label{e:berezin-tau0-Pi}\int_{\IR^{(1|n-1)}}\exp\left(-\frac{2\pi^2}{\text{Vol}_g(M)}\left(p_0^2-\frac{2}{n+2}p_0\sum_{j\neq r}\Pi_jdy^j\right)\right)dp_0d\Pi\hspace{3cm}\end{equation}
$$\hspace{3cm}=\left(\frac{\text{Vol}_g(M)}{2\pi}\right)^{\frac{1}{2}}\frac{(n-1)!}{\left(\frac{n-1}{2}\right)!}\left(\frac{2\pi^2}{\text{Vol}_g(M)(n+2)^2}\right)^{\frac{n-1}{2}}dy^1\wedge\ldots \hat{dy^r}\wedge\ldots dy^n.$$
From~\eqref{e:berezin-tau} and~\eqref{e:berezin-tau0-Pi}, we deduce that
$$\int_{\ml{H}_{\Lambda}}W_0^r(f,t,y,dy,\eta)d\mu_{\Lambda}(f)\hspace{5cm}$$
$$=\frac{(-1)^{\frac{n-1}{2}}}{(2\pi)^{n-1}}\frac{(n-1)!}{\left(\frac{n-1}{2}\right)!}\frac{\pi^{\frac{n-3}{2}}}{2(n+2)^{\frac{n}{2}}}(n+2)\text{Vol}_g(M)e^{-\frac{(n+2)\text{Vol}_g(M)|\eta|^2}{2t^2}}dy^1\wedge\ldots \hat{dy^r}\wedge\ldots dy^n(1+\ml{O}(\Lambda^{-1})).$$
We now use equality~\eqref{e:berezin-step-1}. Summing over $r$ and $s$ and integrating over $t$, one finally gets that, if $n$ is odd,
$$\int_{\ml{H}_{\Lambda}}J(f,\omega,U)d\mu_{\Lambda}(f)=\frac{2(-1)^{\frac{n+1}{2}}}{\pi\operatorname{Vol}(\mathbb{S}^{n-1})}\left(\frac{\Lambda}{\sqrt{n+2}}\right)^n\int_{\mathring{T}^*U}dy^1\wedge\ldots dy^n\wedge\omega(y,dy,\eta,d\eta)+\ml{O}(\Lambda^{n-1}).$$

\appendix

\section{Derivatives of the Spectral projector.}
\label{derivativesspectralproj}

Recall that $(e_j)_{1\leq j\leq N(\Lambda)}$ is an orthonormal basis of $\ml{H}_{\Lambda}$ made of eigenfunctions of $\Delta$. We define the normalized projection kernel:
$$C_{\Lambda}(y,z):=\frac{1}{N(\Lambda)}\sum_{j=1}^{N(\Lambda)}e_j(y)e_j(z),$$
where $N(\Lambda)=\text{dim}\ \ml{H}_{\Lambda}.$ The reason why we can perform some computations with respect to the Gaussian measure $\mu_{\Lambda}$ is that we have very precise asymptotic informations on the kernel $C_{\Lambda}(y,z)$, at least on the diagonal. Precisely, Bin proved the following result~\cite{Bin04} building on earlier arguments of H\"ormander~\cite{Ho68}:
\begin{theo}\label{t:bin} Let $(U,\phi)$ be a sufficiently small geodesic normal coordinate chart. For every multiindices $\alpha$ and $\beta$ in $\IZ^n_+$, the following estimates hold uniformly for $y$ in $\phi(U)$, as $\Lambda\rightarrow+\infty$, 
$$\partial_y^{\alpha}\partial_{z}^{\beta}C_{\Lambda}(y,z)|_{y=z}=C_{n,\alpha,\beta}\Lambda^{|\alpha+\beta|}+\ml{O}(\Lambda^{|\alpha+\beta|-1}),$$
where, for $\alpha=\beta=0$,
$$C_{n,0,0}:=\frac{1}{\operatorname{Vol}_g(M)},$$
for $\alpha\equiv\beta\ (\operatorname{mod} 2)$ and $(\alpha,\beta)\neq (0,0)$,
$$C_{n,\alpha,\beta}:=(-1)^{\frac{|\alpha|-|\beta|}{2}}\frac{\prod_{j=1}^n(\alpha_j+\beta_j-1)!!}{\operatorname{Vol}_g(M)\left(|\alpha+\beta|+n\right)\ldots \left(n+2\right)},$$
and 
$$C_{n,\alpha,\beta}:=0\quad\text{otherwise}.$$
\end{theo}
\begin{rema}
We used the following conventions:
$$(-1)!!:=1\ \text{and}\ (2m-1)!!:=(2m-1)(2m-3)\ldots 3\times 1.$$
\end{rema}

We gather the values of the derivatives of the spectral projector which are used to prove Proposition~\ref{p:pushforwardGaussmeasure}.
\begin{coro}\label{r:useful-asymp} One has
\begin{itemize}
 \item for derivatives of order $0$, one has $$C_{\Lambda}(y,z)|_{y=z}=\frac{1}{\operatorname{Vol}_g(M)}+\ml{O}(\Lambda^{-1}),$$
 \item for every $1\leq j\leq n$, one has $$\frac{1}{\Lambda^2}\partial_{y^j}^2C_{\Lambda}(y,z)|_{y=z}=-\frac{1}{(n+2)\operatorname{Vol}_g(M)}+\ml{O}(\Lambda^{-1}),$$
 \item for every $1\leq j\leq n$, one has $$\frac{1}{\Lambda^2}\partial_{y^j}\partial_{z^j}C_{\Lambda}(y,z)|_{y=z}=\frac{1}{(n+2)\operatorname{Vol}_g(M)}+\ml{O}(\Lambda^{-1}),$$
 \item for every $1\leq j\neq k\leq n$, one has $$\frac{1}{\Lambda^4}\partial_{y^j}^2\partial_{z^k}^2C_{\Lambda}(y,z)|_{y=z}=\frac{1}{\Lambda^4}\partial_{y^jy^k}^2\partial_{z^jz^k}^2C_{\Lambda}(y,z)|_{y=z}=\frac{1}{(n+2)(n+4)\operatorname{Vol}_g(M)}+\ml{O}(\Lambda^{-1}),$$
 \item for every $1\leq j\leq n$, one has $$\frac{1}{\Lambda^4}\partial_{y^j}^2\partial_{z^j}^2C_{\Lambda}(y,z)|_{y=z}=\frac{3}{(n+2)(n+4)\operatorname{Vol}_g(M)}+\ml{O}(\Lambda^{-1}),$$
 \item for every $\alpha\neq \beta\ (\text{mod}\ 2)$, one has $$\frac{1}{\Lambda^{|\alpha+\beta|}}\partial_y^{\alpha}\partial_{z}^{\beta}C_{\Lambda}(y,z)|_{y=z}=\ml{O}(\Lambda^{-1}).$$
\end{itemize}
\end{coro}

\section{Proof of Proposition \ref{intertwiningproposition}.}
\label{intertwiningsectionappendix}
We fix $f$ in $\Omega_{\Lambda}=\ml{H}_{\Lambda}\backslash D_{\Lambda}$. Recall that $D_{\Lambda}$ is a zero measure subset of $\ml{H}_{\Lambda}$ and that its complement is an open.
Given any test function $\Psi(f,y,\eta,t)$ in $\ml{D}(\Omega_{\Lambda}\times\mathring{T}^*U\times\IR^*)$, 
our goal in this paragraph is to make sense of the following Fubini equality~:
\begin{eqnarray*}
 \left\la 1,\left\la G_{\Lambda}(f)^*(\delta_0^{n+1}),\Psi\right\ra_{\mathring{T}^*U\times\IR^*}\right\ra_{\Omega_{\Lambda}} & = & \left\la G_{\Lambda}^*(\delta_0^{n+1}),\Psi\right\ra_{\Omega_{\Lambda}\times \mathring{T}^*U\times\IR^*}\\
 & =  & \left\la 1,\left\la G_{\Lambda}(y,\eta,t)^*(\delta_0^{n+1}),\Psi\right\ra_{\Omega_{\Lambda}}\right\ra_{\mathring{T}^*U\times\IR^*}.
\end{eqnarray*}

The fact that the distribution $G_\Lambda(f)^*\delta_0^{n+1}$ is well defined for $f$ in $\Omega_{\Lambda}$ follows from the following theorem of H\"ormander -- Th.~$8.2.4$ in~\cite{Ho90}:
\begin{theo}\label{t:hormander-pull-back} Let $\Omega_1\subset\IR^{d_1}$ and $\Omega_2\subset\IR^{d_2}$ be two open subsets. For any smooth map $G:\Omega_1\rightarrow \Omega_2$, the normal of $G$ is defined as
$$N_G:=\left\{(G(x),\tau)\in\Omega_2\times\IR^{d_2*}:\tau\circ d_xG=0\right\},$$
the pull--back operation extends uniquely to the distributions $u\in\ml{D}'(\Omega_2)$ whose wavefront $WF(u)$ does not intersect the normal $N_G$ of $G$.
Moreover, the wavefront set of $G^*u$ is contained in the set
$$G^*WF(u):=\left\{(x,\tau\circ d_xG):(G(x),\tau)\in WF(u)\right\}.$$
\end{theo}

\begin{rema} 
For the definition of the wavefront set, we refer to~\cite{Ho90, BrDaHe14a}. Recall that, in the case of $\delta_0^{n+1}$, one has
$$WF(\delta_0^{n+1}):=\left\{(0,\tau):\tau\neq 0\right\}\subset T^{*}\IR^{n+1}.$$
\end{rema}

In the following, we will apply this Theorem in three distinct situations.
Recall that, for every $f$ in $\ml{H}_{\Lambda}$, we defined the following map on $\ml{H}_{\Lambda}\times \phi(U)\times\IR^n\times\IR^*$:
$$G_{\Lambda}:(f,y,\eta,t)\in \ml{H}_{\Lambda}\times \phi(U)\times\IR^n\times\IR^*\longmapsto\left(f(y);\frac{t}{\Lambda}\partial_{y_1}f-\eta_1,\ldots,\frac{t}{\Lambda}\partial_{y_n}f-\eta_n\right)\in\IR\times \IR^n,$$
and our idea is to
think of $\ml{H}_{\Lambda}\times \phi(U)\times\IR^n\times\IR^*$
as the cartesian product of $\ml{H}_{\Lambda}$ with $\phi(U)\times\IR^n\times\IR^*$
and to think of the global distribution
$G_\Lambda^*\delta_0^{n+1}$ 
also as a distribution on $\ml{H}_{\Lambda}$ 
depending smoothly on the parameters in $\phi(U)\times\IR^n\times\IR^* $
and conversely as a distribution on  $\phi(U)\times\IR^n\times\IR^*$ 
depending smoothly on the parameters in
$\ml{H}_{\Lambda}$.
To be more precise, we will use H\"ormander's Theorem in the following three cases~:
\begin{enumerate}
 \item We fix $f$ in $\Omega_{\Lambda}$, and we consider the partial map $G_{\Lambda}(f):\mathring{T}^*U\times\IR^*\rightarrow \IR^{n+1}$. As $d_yf\neq 0$ on $f^{-1}(0)$, one can verify that
$$N_{G_{\Lambda}(f)}\cap WF(\delta_0^{n+1})=\emptyset.$$
It means that if we freeze $f$, the pull--back $G_\Lambda(f)^*\delta_0^{n+1}$ is a well--defined
distribution on $\mathring{T}^*U\times\IR^*$.
 \item We consider the map $G_{\Lambda}:\ml{H}_{\Lambda}\times\mathring{T}^*U\times\IR_{+}^*\rightarrow \IR^{n+1}$. In order to apply H\"ormander's result, we would need to verify
$$N_{G_{\Lambda}}\cap WF(\delta_0^{n+1})=\emptyset,$$
which is slightly less obvious.
\item We fix $(y,\eta,t)$ in $\mathring{T}^*U\times\IR^*$, and we consider the partial map $G_{\Lambda}(y,\eta,t):\ml{H}_{\Lambda}\rightarrow \IR^{n+1}$. The situation in this case is also slightly different from the first case. 

\end{enumerate}
The second and the third cases are in some sense related to the notion of ampleness appearing for instance in~\cite{GaWe14b, Le14}, and they are contained in the following lemma:
\begin{lemm}\label{l:ample} There exists $\Lambda_0>0$ such that, for every $\Lambda\geq\Lambda_0$,
$$N_{G_{\Lambda}}\cap WF(\delta_0^{n+1})=\emptyset,$$
and, for every $(y,\eta,t)$ in $\mathring{T}^*U\times\IR^*$, one has 
 $$N_{G_{\Lambda}(y,\eta,t)}\cap WF(\delta_0^{n+1})=\emptyset.$$
In particular, $G_{\Lambda}^*(\delta_0^{n+1})$ is well defined as a distribution on $\ml{H}_{\Lambda}\times \mathring{T}^*U\times\IR^*$ and $G_{\Lambda}(y,\eta,t)^*(\delta_0^{n+1})$ as a distribution on $\ml{H}_{\Lambda}$.
\end{lemm}

\begin{rema}\label{r:extension} We note that in the second and in the third cases, the distributions are defined on the whole space $\ml{H}_{\Lambda}$, and not only on $\Omega_{\Lambda}$. In particular, we can consider their restriction to the open subset $\Omega_{\Lambda}$. 
\end{rema}

Before continuing our discussion on the properties of these pulled-back distributions, we start by giving the proof of the lemma.

\begin{proof} Suppose that $(0,\tau)$ belongs to the normal set $N_{G_{\Lambda}(y,\eta,t)}$ (the other case works similarly). Then, for every $f$ in $\ml{H}_{\Lambda}$, one has
$$\tau_0 f(y)+\sum_{j=1}^{N(\Lambda)}\tau_j\frac{t\partial_{y^j} f}{\Lambda}=0.$$
Consider now a family of functions $f_0^{\Lambda}, f_1^{\Lambda},\ldots ,f_n^{\Lambda}$  in $\ml{H}_{\Lambda}$, we write
$$\text{det}\left(\left(f_l^{\Lambda}(y),\frac{t\partial_{y^1} f_l^{\Lambda}}{\Lambda},\ldots, \frac{t\partial_{y^n} f_l^{\Lambda}}{\Lambda}\right)_{0\leq l\leq n}\right)= \frac{t^n}{\Lambda^n}\text{det}\left(\left(f_l^{\Lambda}(y),\partial_{y^1} f_l^{\Lambda},\ldots, \partial_{y^n} f_l^{\Lambda}\right)_{0\leq l\leq n}\right).$$
If we are able to find a family $(f_l^{\Lambda})_{0\leq l\leq n}$ in $\ml{H}_{\Lambda}$ such that the right-hand side of this equality does not vanish, we will deduce that $\tau=0$. For that purpose, we note that we can find\footnote{We may have to pick a slightly smaller neighborhood $U$.} a family of smooth functions $f_0, f_1,\ldots, f_n$ on $M$ such that the right hand side does not vanish for every $y$ in $U$. Unfortunately, these functions do not a priori belong to $\ml{H}_{\Lambda}$ but we can solve this issue since by Lemma \ref{eigenfunctionsschauderbasis} proved in appendix, 
finite linear combinations of eigenfunctions are everywhere dense in $C^\infty(M)$. 
Said concretely, to every $0\leq l\leq n$ corresponds a sequence $(f_l^{\Lambda})_{\Lambda}$ in $\ml{H}_{\Lambda}$ which converges to $f_l$ in every $\ml{C}^k$ norm (this is just an approximation property like the Stone Weierstrass Theorem but on Riemannian manifolds). From this, we deduce the statement of the lemma for $\Lambda$ large enough. 
\end{proof}
We refer the reader to the section~\ref{sectionappendixwhitneyampleness} of the appendix where a similar method yields
an unusual proof of the Whitney embedding Theorem.

As a direct consequence of the pull-back theorem of H\"ormander, we can also describe the wavefront set of $G_{\Lambda}^*(\delta_0^{n+1})$ as follows:
\begin{coro}\label{c:wavefront} For $\Lambda>0$ large enough, the wavefront set of $G_{\Lambda}^*(\delta_0^{n+1})$ is included in the following subset of $T^*(\ml{H}_{\Lambda}\times\mathring{T}^*U\times\IR^*$):
\begin{eqnarray*}
WF(G_{\Lambda}^*(\delta_0^{n+1}))&\subset&\{(y^j,\eta_j,t,f;\widehat{y}_j,\widehat{\eta}^j,\widehat{t},\widehat{f} )\text{ such that } G_{\Lambda}(f,y,\eta,t)=(0,\dots,0)\\
&& \widehat{y}_j=\tau_0 \partial_{y^j}f+\tau_1\frac{t}{\Lambda} \partial^2_{y^jy^1}f+\dots + 
\tau_n\frac{t}{\Lambda} \partial^2_{y^jy^n}f,\ \widehat{\eta}_j=\tau_j,\\
&&\widehat{t}=\frac{1}{\Lambda}\sum_{j=1}^n\tau_j \partial_{y^j}f, \ 
\forall 1\leq i\leq N(\Lambda),\ \widehat{f}_i=\tau_0e_i+\sum_{j=1}^n\frac{t\tau_j}{\Lambda}\partial_{x^j}e_i\\
&&\text{for }(\tau_0,\tau_1,\dots,\tau_n)\neq (0,\dots,0) \}.
\end{eqnarray*}
 
\end{coro}

\begin{rema}
This corollary has the following consequence. The set $WF(G_{\Lambda}^*(\delta_0^{n+1}))$ does not contain any element of the form $(y^j,\eta_j,t,f;0,0,0,\widehat{f} ),$ with $\widehat{f}\neq 0$ or $(y^j,\eta_j,t,f;\widehat{y}_j,\widehat{\eta}^j,\widehat{t},0 )$ with $(\widehat{y}_j,\widehat{\eta}^j,\widehat{t})\neq 0$. The first observation is more or less immediate if one uses Theorem~\ref{t:bin} in the case $\alpha=\beta=0$. The second one follows from the proof of lemma~\ref{l:ample}.
\end{rema}

The relevance of such a remark is due to the following general result which can be found in \cite[Proposition 1.3 p.~502]{ChaPi82}~:
\begin{prop}\label{p:Chazarainprop}
Let $u$ be a distribution in $\mathcal{D}^\prime\left(\mathbb{R}^{n_1}\times\mathbb{R}^{n_2}\right)$.
We denote by $(x,y;\xi,\eta)\in T^*\left(\mathbb{R}^{n_1}\times\mathbb{R}^{n_2}\right)$ the coordinates
in cotangent space. If $WF(u)$ contains no elements of the form $\{(x,y;\xi,0)\text{ such that }\xi\neq 0 \}$, then
$u(x,y)$ is smooth in $x$ with value distribution in $y$ in the sense that for all test function $\varphi_2(y)$,
$x\longmapsto \langle u(x,.), \varphi_2(.)\rangle_{\IR^{n_2}} $ is smooth in $x$.
\end{prop}

Thus, one finds that there exists $f\in\ml{H}_{\Lambda}\mapsto T_{\Lambda}(f)\in\ml{D}'(\mathring{T}^*U\times\IR^*)$ such that, for every $\psi_1\in\ml{D}(\ml{H}_{\Lambda})$ and for every $\psi_2\in\ml{D}(\mathring{T}^*U\times\IR^*)$,
$$f\mapsto \la T_{\Lambda}(f),\psi_2\ra_{\mathring{T}^*U\times\IR^*}\in\ml{C}^{\infty}(\ml{H}_{\Lambda}),$$
and
$$(y,\eta,t)\mapsto \la G_{\Lambda}(y,\eta,t)^*\delta_0^{n+1},\psi_1\ra_{\Omega_{\Lambda}}\in\ml{C}^{\infty}(\mathring{T}^*U\times\IR^*).$$
Moreover, one has the ``following distributional version of the Fubini Theorem''~: 
\begin{eqnarray*}
\forall (\psi_1,\psi_2)\in\mathcal{D}(\ml{H}_\Lambda)\times \mathcal{D}(\mathring{T}^*U\times\IR^*)&&\\
 \left\la \psi_1,\left\la T_{\Lambda}(f),\psi_2\right\ra_{\mathring{T}^*U\times\IR^*}\right\ra_{\ml{H}_{\Lambda}} & = & \left\la G_{\Lambda}^*(\delta_0^{n+1}),\psi_1\boxtimes\psi_2\right\ra_{\ml{H}_{\Lambda}\times \mathring{T}^*U\times\IR^*}\\
 & =  & \left\la \psi_2,\left\la G_{\Lambda}(y,\eta,t)^*(\delta_0^{n+1}),\psi_1\right\ra_{\ml{H}_{\Lambda}}\right\ra_{\mathring{T}^*U\times\IR^*}.
\end{eqnarray*}
\begin{rema}
 Note that $T_{\Lambda}(f)$ is well defined on the whole space $\ml{H}_{\Lambda}$ while $G_{\Lambda}(f)^*(\delta_0^{n+1})$ was only defined on $\Omega_{\Lambda}$ by Theorem~\ref{t:hormander-pull-back}. For a fixed $f$ in $\Omega_{\Lambda}$, one has in fact $T_{\Lambda}(f)=G_{\Lambda}(f)^*(\delta_0^{n+1})$.
\end{rema}

By density of the functions of the form $\psi_1(f)\boxtimes\psi_2(y,\eta,t)$ (where $\boxtimes$ is the exterior tensor product) in $\ml{D}(\Omega_{\Lambda}\times\mathring{T}^*U\times\IR^*)$ -- see Ch.~IV in~\cite{Sch66}, 
given any test function $\Psi(f,y,\eta,t)$ in $\ml{D}(\Omega_{\Lambda}\times\mathring{T}^*U\times\IR^*)$, one has the following Fubini equality~:
\begin{eqnarray}
 \left\la 1,\left\la T_{\Lambda},\Psi\right\ra_{\mathring{T}^*U\times\IR^*}\right\ra_{\ml{H}_{\Lambda}} & = & \left\la G_{\Lambda}^*(\delta_0^{n+1}),\Psi\right\ra_{\ml{H}_{\Lambda}\times \mathring{T}^*U\times\IR^*}\\
\nonumber & =  & \left\la 1,\left\la G_{\Lambda}(y,\eta,t)^*(\delta_0^{n+1}),\Psi\right\ra_{\ml{H}_{\Lambda}}\right\ra_{\mathring{T}^*U\times\IR^*}.
\end{eqnarray}

\begin{rema}\label{r:positive} We emphasize that all the distributions considered so far are constructed from the Dirac distribution $\delta_0^{n+1}$ and that they are positive distributions. 
In particular, according to~\cite{Sch66} (Ch.~I), \emph{they can all be identified with positive Radon measures}. Therefore, it makes sense to test them against nonnegative measurable functions which are 
not necessarily integrable (or integrable function which are not necessarily smooth). This identification is used several times in our proof of proposition~\ref{l:L1}.
\end{rema}

\begin{rema} We now make a final useful comment. We fix $(\psi_1,\psi_2)\in\mathcal{D}(\ml{H}_\Lambda)\times \mathcal{D}(\mathring{T}^*U\times\IR^*)$. According to~\eqref{e:fubini-distribution}, one has
$$\int_{\mathring{T}^*U\times\IR^*}\psi_2(y,\eta,t)\left(\int_{\ml{H}_{\Lambda}}\psi_1(f) G_{\Lambda}(y,\eta,t)^*(\delta_0^{n+1})(df)\right)d^nyd^n\eta dt$$
$$=\int_{\ml{H}_{\Lambda}}\psi_1(f)\left(\int_{\mathring{T}^*U\times\IR^*}\psi_2(y,\eta,t) T_{\Lambda}(d^ny,d^n\eta,dt)\right)df.$$ 
From the above discussion, one knows that
$$f\mapsto \psi_1(f)\left(\int_{\mathring{T}^*U\times\IR^*}\psi_2(y,\eta,t) T_{\Lambda}(d^ny,d^n\eta,dt))\right)$$
is a smooth compactly supported function. As $\Omega_{\Lambda}=\ml{H}_{\Lambda}\backslash D_{\Lambda}$ and $D_{\Lambda}$ has zero Lebesgue measure, one has then
$$\int_{\mathring{T}^*U\times\IR^*}\psi_2(y,\eta,t)\left(\int_{\ml{H}_{\Lambda}}\psi_1(f) G_{\Lambda}(y,\eta,t)^*(\delta_0^{n+1})(df)\right)d^nyd^n\eta dt$$
$$=\int_{\Omega_{\Lambda}}\psi_1(f)\left(\int_{\mathring{T}^*U\times\IR^*}\psi_2(y,\eta,t) G_{\Lambda}(f)^*(\delta_0^{n+1})(d^ny,d^n\eta,dt)\right)df.$$ 
$$=\int_{\mathring{T}^*U\times\IR^*}\psi_2(y,\eta,t)\left(\int_{\Omega_{\Lambda}}\psi_1(f) G_{\Lambda}(y,\eta,t)^*(\delta_0^{n+1})(df)\right)d^nyd^n\eta dt,$$ 
where the second equality follow from the fact that 
$$G_{\Lambda}(f)^*(\delta_0^{n+1})(d^ny,d^n\eta,dt)df \text{ and } G_{\Lambda}(y,\eta,t)^*(\delta_0^{n+1})(df)d^nyd^n\eta dt$$ 
define the same Radon measure on the space $\Omega_{\Lambda}\times\mathring{T}^*U\times\IR^*.$ 
Note also that this equality holds for any smooth compactly supported function $\psi_2$. In particular, one finds that, for a.e. $(y,\eta,t)$ in $\mathring{T}^*U\times\IR^*$,
\begin{equation}\label{e:distrib-restriction}
 \int_{\ml{H}_{\Lambda}}\psi_1(f) G_{\Lambda}(y,\eta,t)^*(\delta_0^{n+1})(df)=\int_{\Omega_{\Lambda}}\psi_1(f) G_{\Lambda}(y,\eta,t)^*(\delta_0^{n+1})(df).
\end{equation}

\end{rema}

\section{A functional analytic proof of Whitney embedding Theorem}
\label{sectionappendixwhitneyampleness}
In this short section, 
we would like to prove briefly
the approximation property used in the proof
of Lemma \ref{l:ample}, and to illustrate it by giving 
a simple proof of the Whitney embedding Theorem.
In fact, on a smooth compact manifold $M$, 
if we have enough smooth functions $e_1,\dots,e_N$ so that 
their linear combinations approximates
any smooth function with enough accuracy, then we will prove that we
can embed the manifold $M$ in $\mathbb{R}^N$
for $N$ sufficiently large by the simple map~:
$$ x\in M\mapsto (e_1(x),\dots,e_N(x))\in \mathbb{R}^N .$$
\subsection{Approximation of smooth functions}
We denote by $H^s(M)$ the Sobolev space
of functions or distributions of order $s$ i.e. the set of all distributions $t$
such that $\Delta_g^{\frac{s}{2}}t$ belongs to $L^2(M)$ with norm $\Vert .\Vert_{H^s(M)}=\Vert (1-\Delta_g)^{\frac{s}{2}}. \Vert_{L^2(M)}$. Note by the spectral Theorem that
$$\Vert (1-\Delta_g)^{\frac{s}{2}}t\Vert_{L^2(M)}=\left(\sum_{i=1}^{+\infty }(1+\lambda_i^2)^{\frac{s}{2}} \vert\langle t,e_i\rangle_{L^2}\vert^2\right)^{\frac{1}{2}},$$ 
where we use the conventions of the introduction. First, we show that eigenfunctions
of the Laplace operator span a vector space which is everywhere
dense in $C^\infty(M)$. This property was used to prove lemma~\ref{l:ample}.
\begin{lemm}\label{eigenfunctionsschauderbasis}
Let $(M,g)$ be a smooth compact boundaryless Riemannian manifold. 
Then, finite combination of eigenfunctions
$(e_i)_i$
of the Laplace-Beltrami operator $\Delta_g$ spans a dense subspace of $C^\infty(M)$. 
More precisely, for every
$C^k$ norm $\Vert . \Vert_{C^k}$~:
$$  \Vert f-\sum_{i=1}^N \langle f,e_i \rangle_{L^2} e_i\Vert_{C^k} \underset{N\rightarrow \infty}{\rightarrow} 0 .$$
\end{lemm}
\begin{proof} Eigenfunctions
$(e_i)_i$ of the Laplace operator are smooth by elliptic regularity. Thus, any finite combination of them belongs to $\ml{C}^{\infty}(M)$.
By the Sobolev embedding Theorem, it suffices to show that for any $s>0$,
$$  \Vert f-\sum_{i=1}^N \langle f,e_i \rangle_{L^2} e_i\Vert_{H^s} \underset{N\rightarrow \infty}{\rightarrow} 0 .$$
For that purpose, we write
$$  \Vert f-\sum_{i=1}^N \langle f,e_i \rangle_{L^2} e_i\Vert_{H^s}^2=\sum_{i=N+1}^\infty (1+\lambda_i^2)^{\frac{s}{2}}\vert \langle f,e_i \rangle_{L^2}\vert^2,$$
which tends to $0$ as $f$ is smooth.
\end{proof}

\subsection{Whitney embedding Theorem}
We will now show how this approximation property can be used to prove the following version of Whitney embedding Theorem:

\begin{theo}
Let $(M,g)$ be a smooth compact Riemannian manifold without boundary.
Let $(e_i)_i$ be the sequence of eigenfunctions of the Laplace-Beltrami operator as defined in the introduction. 
Then, there exists $N_0$ such that for
any $N\geqslant N_0$, the map
\begin{equation}
i_N:x\in M\mapsto (e_1(x),\dots,e_N(x))\in \mathbb{R}^N
\end{equation}
is an embedding.
\end{theo}

\begin{proof}
We first prove that $i_N$ is an immersion then we show it is injective.

Fix some $N$.
To show that $i_N$ is an immersion, we proceed as in the proof of lemma~\ref{l:ample}. It suffices to show that
for all $x\in M$, $di_N(x)$ has maximal rank in other words
$$ x\mapsto (de_1,\dots,de_N)$$ has rank $\dim M=n$.
Equivalently, it means that near every $x_0\in M$, there is a neighborhood
$U$, $n$ sequences $(a^1_i)_{1\leqslant i\leqslant N},\dots,(a^n_i)_{1\leqslant i\leqslant N}$ 
such that
for all $x\in U$, the covectors 
$$\sum_{i=1}^Na^1_ide_i(x),\dots,\sum_{i=1}^Na^n_ide_i(x)$$ are linearly independent.
In other words, the $n$-form
$$(\sum_{i=1}^N a^1_ide_i)\wedge\dots\wedge (\sum_{i=1}^Na^n_ide_i) $$
never vanishes on $U$. But this is easy to achieve since
for every $x_0$, we can always find some sufficiently
small neighborhood 
$U$ of $x_0$ and functions $f_1,\dots,f_n$ so that
$df_1\wedge \dots\wedge df_n(x)\neq 0,\forall x\in U$.
And since the family $(e_i)_i$ spans a dense subspace of $C^\infty(M)$,
we can always find $N_0$ large enough so that
the vector space spanned by $(e_i)_{i\leqslant N_0} $ is very close to $f_1,\dots,f_n$ in
the $C^1$ topology.

The last part consists in proving that
the above map is injective for $N$ large enough. 
By the previous part, we can fix $N_0$ large enough so that $i_{N_0}$ is an immersion.
Since an immersion is locally injective and $M$ is compact, 
there is some $\varepsilon>0$, so that
for all $x\in M$, consider the ball $B(x,\varepsilon)$
centered around $x$ of radius $\varepsilon$, the map
$i_{N_0}:B(x,\varepsilon) \mapsto \mathbb{R}^{N_0}$
is injective. 
In other words if $i_{N_0}(x)=i_{N_0}(y)$ for some $(x,y)\in M^2$ then $x=y$ 
or $d(x,y)>\varepsilon$.
We denote by $\Delta_M$ the diagonal in $M\times M$.
Consider 
the map $i_{N_0}\times i_{N_0}:\left(M\times M\right)\setminus \Delta_M \mapsto \mathbb{R}^{N_0}\times \mathbb{R}^{N_0}$
which reaches the diagonal $\Delta\subset \mathbb{R}^{N_0}\times \mathbb{R}^{N_0}$ 
exactly at the
pair of points $(x,y)$ where $i_{N_0}$ fails to be injective.
First note that 
$(i_{N_0}\times i_{N_0})^{-1}(\Delta)=\{(x,y)\text{ such that }i_{N_0}(x)=i_{N_0}(y),d(x,y)\geqslant \varepsilon\}$ is a closed subset of $M\times M\setminus \Delta_M$. 
Since $M\times M\setminus \Delta_M$ is precompact, the closed subset $(i_{N_0}\times i_{N_0})^{-1}(\Delta)$ 
is compact.
Note by definition of $i_N$ that for all $N$,
$(i_{N+1}\times i_{N+1})^{-1}(\Delta)\subset (i_{N}\times i_{N})^{-1}(\Delta)$.
Hence consider the sequence $\left((i_N\times i_N)^{-1}(\Delta)\right)_N $ for $N\geqslant N_0$, it is 
a decreasing sequence of compact subsets in 
$M\times M\setminus\{(x,y)\in M^2 \text{ s.t. }d(x,y)<\varepsilon \}$.
By contradiction, assume that for all $N\geqslant N_0$, 
the map $i_N$ fails to be injective,
then it means that $(i_N\times i_N)^{-1}(\Delta)$ is 
non empty for all $N$, then 
for all $N$
there is a pair $(x_N,y_N)\in M\times M\setminus\{(x,y)\in M^2 \text{ s.t. }d(x,y)<\varepsilon \}$
such that $i_N(x_N)=i_N(y_N)$ for all $N$. By compactness of
$M\times M\setminus\{(x,y)\in M^2 \text{ s.t. }d(x,y)<\varepsilon \}$, 
we may assume that
the sequence $(x_N,y_N)_N$ converges to some element $(x,y)$ in 
$M\times M\setminus\{(x,y)\in M^2 \text{ s.t. }d(x,y)<\varepsilon \}$.
It follows by definition of $i_N$ that $e_i(x_k)=e_i(y_k)$ for all
$k\geqslant i$ and by continuity of the 
the functions $(e_i)_i$ and the fact that $x_k\rightarrow x,y_k\rightarrow y$, we also have that $e_i(x)=e_i(y)$ for all $i$.

On the other hand, we know that 
there is some smooth function $f$ such that
$f(x)\neq f(y)$. And since
$f=\sum_{i=1}^{+\infty} \langle f,e_i \rangle e_i$, this implies that $f(x)=\sum_{i=1}^{+\infty} \langle f,e_i \rangle e_i(x)=\sum_{i=1}^{+\infty} \langle f,e_i \rangle e_i(y)=f(y)$ which contradicts the fact that $f(x)\neq f(y)$.  
\end{proof}

\section{Background material on the theory of currents}
\label{a:currents}
In this appendix, we give formulas for the current of integration over submanifolds defined by systems 
of equations of the form $g_1=g_2=\ldots= g_d=0$ for some sufficiently nice functions $g_1,\ldots,g_d$. This kind of formulas can be found for instance in Schwartz's book~\cite{Sch66} -- section~V.$5$ (example $3$). 
The advantage of representing currents of integration with Dirac distributions is that they will then be easily integrated against the Gaussian measure 
$d\mu_\Lambda$ on $\mathcal{H}_\Lambda$ -- see section~\ref{s:proof}. For an introduction to the theory of currents, we also refer to the books of de Rham~\cite{dRh80} and Giaquinta--Modica--Soucek~\cite{GiMoSou98}.

\subsection{Currents on zero sets}

In mathematical physics, one is often interested in studying
subsets in $X$ given by equations $\{G=0\}$ for some $G=(g_i)_{1\leqslant i\leqslant d}$. In our article, we will encounter two classes of currents associated to this type of submanifolds.
The first class are currents \emph{representable by integration on some smooth submanifolds} which
form a subclass of the De Rham chains~-- see chapter $2$ in~\cite{dRh80}. 
Let $X$ be a smooth oriented manifold of dimension $m$, let $S$ be a smooth, oriented, compact submanifold of $X$ of dimension $d\leq m$. 
This defines an integration current  $[S]$ in $\ml{D}_{d}'(X)$ which acts on test forms $\omega\in\ml{D}_d(X)$ by
\begin{eqnarray} 
[S](\omega)=\int_{S} \omega.
\end{eqnarray}

The second class consists in Dirac distributions which are defined as follows. 
We denote by $\delta^d_0$ the usual Dirac distribution at $\{0\}$ on $\mathbb{R}^d$.
This is defined by 
\begin{equation}
\forall \varphi(x) dx^1\wedge \dots \wedge dx^d\in \mathcal{D}_d(\mathbb{R}^d), \la\delta^d_0,\varphi dx^1\wedge \dots \wedge dx^d\ra_{\IR^d}=\varphi(0).
\end{equation}
Now let us define the Dirac distribution $G^*(\delta_0^d)=\delta_0^{d}\circ G$ supported by the zeros of a function $G:X\rightarrow\IR^d$.

\begin{prop}\label{p:Dirac}
Let $G=(g_i)_{1\leqslant i\leqslant d}\in C^\infty(X)$ 
be such that, for all $x\in X$ satisfying $g_1(x)=\dots=g_d(x)=0$, the linear forms
$d_xg_1(x),\dots,d_xg_d(x)$ are linearly independent. 
Then 
\begin{enumerate}
\item the pull--back of $\delta^d_0\in\mathcal{D}^\prime_d(\mathbb{R}^d)$ by $G$
which is denoted by $G^*\delta_0^d$ (or also $\delta^d_0\circ G$) is well--defined, 
\item the wave front set of $G^*\delta_0^d$ is contained in 
$$N^*(\{G=0\})=\left\{(x,\xi)\in T^*X: G(x)=0, \xi=\tau\circ d_xG\ \text{for some}\ \tau\neq 0\in\IR^d\right\}$$
\item for every sequence $(T_p)_p$ which converges
to $\delta^d_0$ in $\mathcal{D}^\prime_{\{0\}\times(\IR^d)^*}(\mathbb{R}^d)$,
$G^*T_p \rightarrow G^*\delta^d_0$ in the normal topology of $\mathcal{D}^\prime_{N_0^*\{G=0\}}(X)$
which implies convergence in the weak topology of $\mathcal{D}^\prime(X)$.
\end{enumerate}
\end{prop}

\begin{rema}
For the notion of wave front set and for the precise definition of the set $\ml{D}_{\Gamma}'(X)$ of distributions whose wave front set is contained in $\Gamma$, we refer the reader to section $8.2$ of~\cite{Ho90} (see also~\cite{BrDaHe14a}) and for the definition of the normal topology on $\ml{D}_{\Gamma}'(X)$, we refer to \cite{BrDaHe14}. 
\end{rema}

\begin{proof}
Let $N_G=\{(y;\eta)\in T^*\mathbb{R^d} \text{ s.t. } G(x)=y\text{ and } \eta\circ d_xG=0 \text{ for some }x\in X  \}$ be the normal set of the map $G$. By assumption, $N_G$ does not meet the set $\{0\}\times(\IR^d)^*\subset T^*\mathbb{R}^d$
which is the wave front set of $\delta^d_0$. Therefore
by the pull--back Theorem of H\"ormander -- Th.~\ref{t:hormander-pull-back} below, $G^*\delta^d_0=\delta^d\circ G$ is well--defined
and its wave front set is included $G^*(\{0\}\times(\IR^d)^*)=N_0^*(\{G=0\})$.
Finally, the last claim follows from the sequential continuity of the pull-back operation -- Th.~$8.2.4$ in~\cite{Ho90} or~\cite{BrDaHe14}.
\end{proof}

Now the natural question which arises is how the integration current on $S=\{G=0\}$ and the
delta form $G^*\delta_0^d=\delta_0^d\circ G$ are related to each other.
Note that the Dirac distribution of the submanifold $S=\{G=0\}$ is obviously not 
the same as the integration current $[S]$ because we integrate \textbf{top forms} against $\delta_0^d\circ G$
but only $(m-d)$ forms on $[S]$.

%

\subsection{Orientations on submanifolds.}\label{r:orientation}
Before comparing these currents, we need to discuss briefly orientability questions. For any smooth manifold $X$ (orientable or not),
there exists a smooth
principal $\mathbb{Z}_2$ bundle $\text{O}_X\mapsto X$
called the orientation bundle.
Recall that if $X$
is an orientable manifold, then there exists 
a volume form $\omega\in \Gamma(X,\Lambda^nT^*X)$ where 
$n=\dim X$ or equivalently there exists 
a global section of $\text{O}_X$. 
Indeed, every 
choice of volume form $\omega$ on $X$
defines canonically a global section $[\omega]$ of the orientation
bundle $\text{O}_X$, also $[\varphi\omega]=[\omega]$ for all $\varphi>0$. 
Now, if we are given a submanifold $S\subset X$ defined
by equations $g_1=\dots=g_d=0$ and if
$X$ is oriented by $[\omega_1]\in \text{O}_X$ and $S$ by $[\omega_2]\in \text{O}_S$, we
choose representatives $\omega_1\in \Gamma(X,\Lambda^nT^*X)$ and $\omega_2\in \Gamma(S,\Lambda^{n-d}T^*S) $.
Then we say that the orientation of $dg_1\wedge \dots \wedge dg_d$
is compatible with the pair of orientations $[\omega_1],[\omega_2]$
if
$dg_1\wedge \dots \wedge dg_d\wedge \omega_2|_S=\varphi\omega_1|_S$
for some smooth function $\varphi$ s.t. $\varphi|_S>0$.

\subsection{Relation between the different types of currents}

In the next proposition, we give a concrete description of the action of
Dirac distributions $G^*\delta_0^d$ on differential forms.

\begin{prop}\label{deltadivisionprop}
Let $G=(g_i)_{1\leqslant i\leqslant d}$ be a smooth map from $X$ to $\IR^d$ such that the differentials $d_xg_i$ are linearly independent
for every $x$ satisfying $g_1(x)=\dots=g_d(x)=0$. 

Then, for any test form $\omega\in \mathcal{D}_m(X)$ and for any form $\alpha$ in $\ml{D}_{m-d}(X)$ 
such that $dg_1\wedge \dots\wedge dg_d\wedge\alpha=\omega$, one has
\begin{eqnarray}
\left\la G^*\delta^d_0,\omega\right\ra = \int_{\{G=0\}} \alpha.
\end{eqnarray}
\end{prop}


\begin{proof}

We first prove the claim 
in the case where $g=(g_i)_{1\leqslant i\leqslant d}$
is a linear map from $\mathbb{R}^n$ to $\mathbb{R}^d$ of maximal rank.
Then one can choose $(g_1,\dots,g_d)$ to be the first $d$ 
components of some system of linear coordinate functions
$(g_1,\dots,g_d,\dots,g_n)$ on $\mathbb{R}^n$, in this case $G$ is just the
linear projection $\mathbb{R}^n\mapsto \mathbb{R}^d$ on the first $d$ components
and the zero set 
$\{G=0\}$ is just $\{0\}\times\mathbb{R}^{n-d} \subset \mathbb{R}^n$. 
For any function $\psi$ on $\mathbb{R}^d$, the pull--back of $\psi$ by $G$
simply reads
$$G^*\psi(g_1,\dots,g_n)=\psi(g_1,\dots,g_d)\otimes 1(g_{d+1},\dots,g_n)$$
and therefore, by continuity of the pull--back operation~\cite{BrDaHe14a},
one has $\delta^d_0\circ G= \delta^d_0\otimes 1$. 
Any form $\omega$ has a unique 
representation in the basis $dg_1\wedge \dots \wedge dg_n$
as $$\omega=\omega_{1\dots n}(g_1,\dots,g_n)dg_1\wedge \dots \wedge dg_n$$ 
therefore~:
\begin{eqnarray*}
[\delta^d_0\circ G](\omega) &=& [\delta^d_0 \otimes 1(g_{d+1},\dots,g_n) ](\omega) \\
&=& \int_{\mathbb{R}^{n-d}} \omega_{1\dots n}(0,\dots,0,g_{d+1},\dots,g_n)dg_{d+1}\wedge \dots \wedge dg_n.
\end{eqnarray*}
Note that any form $\alpha$ such that $dg_1\wedge \dots\wedge dg_d \wedge \alpha=\omega$
will coincide with 
$$\omega_{1\dots n}(0,\dots,0,g_{d+1},\dots,g_n)dg_{d+1}\wedge \dots \wedge dg_n$$
when restricted to $\{G=0\}=\{0\}\times\mathbb{R}^{n-d}$. It follows that the pairing does not depend on 
the choice of $\alpha$.

Let us go back to the manifold case.
For any $m\in \{G=0\}$, let $U$ be some neighborhood of $m$
on which the differentials $dg_i$ are linearly independent.
Then one can choose $(g_1,\dots,g_d)$ to be the first $d$ 
components of some system of local coordinate functions
$(g_1,\dots,g_d,\dots,g_n)$ on $U$. 
What we just did was to define some local chart $\Phi: m\in U\subset X\mapsto (g_1(m),\dots,g_n(m))\in \mathbb{R}^n$ on
$U\subset X$ such that
$G\circ \Phi^{-1}:\mathbb{R}^n\mapsto \mathbb{R}^d$ is just the linear projection 
on the first $d$ components. Then 
the distribution $\delta^d_0\circ \left( G\circ \Phi^{-1} \right)$
is well defined by our previous discussion since $G\circ \Phi^{-1}$ is linear.
Then we have the identity~:
\begin{eqnarray*}
\langle\left( G\circ \Phi^{-1} \right)^*\delta^d_0,\omega\rangle = \int_{\{G\circ \Phi^{-1}=0\}}\alpha\\ 
\forall \alpha
\text{ s.t. }dg_1\circ \Phi^{-1}\wedge \dots\wedge dg_d\circ \Phi^{-1}\wedge\alpha=\omega.
\end{eqnarray*}
We just constructed a delta distribution $\left( G\circ \Phi^{-1} \right)^*\delta^d_0$ 
supported by some linear subspace and
finally, to go back to the open subset $U$ in the manifold $X$, 
we must pull--back the distribution $\left( G\circ \Phi^{-1} \right)^*\delta^d_0$ 
by $\Phi$. By the pull--back Theorem~\ref{t:hormander-pull-back}, this yields 
$\Phi^*\left( G\circ \Phi^{-1} \right)^*\delta^d_0=G^*\delta^d_0$ and, for every $\alpha
\text{ s.t. }dg_1\wedge \dots\wedge dg_d\wedge\alpha=\omega$,
\begin{eqnarray*}
\langle G^*\delta^d_0,\omega\rangle = \int_{\{G=0\}}\alpha.
\end{eqnarray*}
\end{proof}

 In fact, using the notations of paragraph~\ref{r:orientation} concerning the orientation of submanifolds
defined by systems of equations, the following corollary holds~: 
\begin{coro}\label{lemmadeltasubmanifolds}
Let $X$ be a smooth oriented manifold with orientation $[\omega_1]$, let
\begin{eqnarray*}
G:=(g_i)_{1\leqslant i\leqslant d}:X\rightarrow\IR^d
\end{eqnarray*}
be a smooth function such that the differentials $d_xg_i$ are linearly independent 
for every $x\in X$ satisfying $g_1(x)=\dots=g_d(x)=0$. Set $S$ to be the submanifold in $X$ defined by the regular system of equations
$\{g_1=\dots=g_d=0\}$ and oriented by $[\omega_2]$. If $dg_1\wedge \dots\wedge dg_d$ has
orientation compatible  
with the pair of orientations $[\omega_1],[\omega_2]$
then one has
\begin{equation}
[S]=G^*(\delta_0^d)\ dg_1\wedge \dots \wedge dg_d. 
\end{equation}
\end{coro}

\begin{proof} We note that this corollary is proved in~\cite{Sch66} (Section V.$5$) when $d=1$. It is in fact a direct consequence of the property of the Dirac distribution $\delta_0^d\circ G$
established in proposition~\ref{deltadivisionprop}, i.e.
$$\left\langle G^*\delta_0^d dg_1\wedge\dots\wedge dg_d,\alpha\right\rangle=\left\langle G^*\delta_0^d , dg_1\wedge\dots\wedge dg_d\wedge\alpha\right\rangle=\int_{S} \alpha. $$
\end{proof}

\end{document}